\documentclass[11pt]{article}
\usepackage{bbm}
\usepackage{mathrsfs}
\usepackage{amsfonts}
\usepackage{amssymb}
\usepackage{amssymb,amsmath,amsthm, amsfonts}
\providecommand{\U}[1]{\protect \rule{.1in}{.1in}}
\newtheorem{theorem}{Theorem}[section]

\newtheorem{corollary}{Corollary}[section]

\newtheorem{definition}{Definition}[section]

\newtheorem{lemma}{Lemma}[section]

\newtheorem{proposition}{Proposition}[section]
\newtheorem{remark}{Remark}[section]

\def\esssup{\hbox{\rm ess$\,$\rm sup$\,$}}
\def\esssup{\mathop{\rm esssup}}

\def\sup{\mathop{\rm sup}}
\makeatletter
   
   \@addtoreset{equation}{section}
\makeatother

\begin{document}
\title{Mean-field forward and backward SDEs with jumps. Associated nonlocal quasi-linear integral-PDEs
\thanks {The work has been supported in part by the NSF of P.R.China (No. 11222110), NSFC-RS (No. 11661130148), 111 Project (No. B12023). {\small \textit{E-mail: juanli@sdu.edu.cn.}}}}

\author{ Juan LI\\
{\small School of Mathematics
and Statistics, Shandong University, Weihai,}
{\small Weihai, 264209, P. R. China.}\\
}
\date{October 08, 2016}

\maketitle

\begin{abstract}
In this paper we consider a mean-field backward stochastic differential equation (BSDE) driven by a Brownian motion and an independent Poisson random measure. Translating the splitting method introduced by Buckdahn, Li, Peng and Rainer \cite{BLPR} to BSDEs, the existence and the uniqueness of the solution $(Y^{t,\xi}, Z^{t,\xi}, H^{t,\xi})$,  $(Y^{t,x,P_\xi}, Z^{t,x,P_\xi}, H^{t,x,P_\xi})$ of the split equations are proved. The first and the second order derivatives of the process $(Y^{t,x,P_\xi}, Z^{t,x,P_\xi}, H^{t,x,P_\xi})$  with respect to $x$, the derivative of the process $(Y^{t,x,P_\xi}, Z^{t,x,P_\xi}, H^{t,x,P_\xi})$  with respect to the measure $P_\xi$, and the derivative of the process $(\partial_\mu Y^{t,x,P_\xi}(y), \partial_\mu Z^{t,x,P_\xi}(y), \partial_\mu H^{t,x,P_\xi}(y))$  with respect to $y$ are studied under appropriate regularity assumptions on the coefficients, respectively. These derivatives turn out to be bounded and continuous in $L^2$. The proof of the continuity of the second order derivatives is particularly involved and requires subtle estimates. This regularity ensures that the value function $V(t,x,P_\xi):=Y_t^{t,x,P_\xi}$ is regular and allows to show with the help of a new It\^{o} formula that it is the unique classical solution of the related nonlocal quasi-linear integral-partial differential equation (PDE) of mean-field type.
\end{abstract}

\bigskip



\noindent \section{{\protect \large {Introduction}}}

\noindent Mean-field stochastic differential equations, also called McKean-Vlasov
equations, can be dated back to the works of Kac \cite{KAC1},\ \cite{KAC2} in the 1950s.
Nonlinear mean-field backward stochastic
differential equations had not been investigated before the work of Buckdahn, Djehiche Li and Peng
\cite{BDLP} in 2009. Since their work the theory of mean-field forward-backward
stochastic differential equations (FBSDEs), as well as that of the associated partial differential equations (PDEs) of mean-field type
has been intensively investigated. For example, Buckdahn, Li and Peng \cite{BLP} obtained for mean-field BSDEs an existence and uniqueness theorem, but also a comparison theorem. Using a BSDE approach, first introduced by Peng \cite{PENG} in 1997, the
authors also gave a probabilistic interpretation to related nonlocal partial differential equations.
Min, Peng and Qin \cite{MPQ} proved through a continuation method that fully coupled mean-field FBSDEs have a unique square integrable adapted solution. On the other hand, with the development of the theory of mean-field FBSDEs, many stochastic control problems in
the mean-field framework have also been considered. For instance, Li \cite{LI} studied a stochastic maximum principle for the mean-field controls.
A stochastic optimal control problem with delay and of mean-field type was considered by
Shen, Meng and Shi \cite{SMS}. With the help of the theory of FBSDEs involving the value function, but with frozen partial
initial values, Hao and Li investigated an optimal control problem with systems of decoupled controlled mean-field
FBSDEs \cite{HL1}, as well as fully coupled controlled mean-field
FBSDEs \cite{HL2}. We remark that, generally speaking, the dynamic programming
principle for mean-field FBSDEs does not hold true anymore because of the
presence of expectation terms in coefficients. For this reason, in \cite{BLP}, \cite{HL1} and \cite{HL2}, the authors adopted a new method: They fixed partially the initial values, to overcome this difficulty. Besides, there are
also many other works in the mean-field area, see, e.g., Kloeden and Lorenz \cite{KL},
Kotelenez and Kurtz \cite{KK}, Yong \cite{Yo} and references therein. In particular,  Lasry and Lions
\cite{LL}
extended the application areas for mean-field problems to Economics, Finance and game theory.

The lectures given by P.L. Lions \cite{LIONS} at $\emph{Coll\`{e}ge de France}$ and the notes edited
by Cardaliguet \cite{Ca1} give the definition of the derivative for a function $\varphi: \mathcal{P}_2(\mathbb{R}^d)\rightarrow\mathbb{R}$
with respect to measure. Many works adopt this definition, for example, R. Carmona and
F. Delarue \cite{CD2}, Cardaliaguet \cite{Ca2}.
Among all these works,  we refer in particular to that of Buckdahn, Li, Peng and Rainer \cite{BLPR}.
The authors considered general mean-field SDEs and related nonlocal PDEs, and proved that the solution $(X^{t,\xi},X^{t,x,P_\xi})$
of such a couple of forward SDEs satisfies the $\emph{flow property}$. This allowed to prove that the
associated nonlocal PDE has a unique classical solution. This approach overcame the drawback of partial freezing of initial
data, see \cite{BLP}. Recently, Chassagneux, Crisan and Delarue \cite{CCD} considered fully coupled
mean-field FBSDEs driven by Brownian motion, and proved that the fully coupled mean-field FBSDEs have unique solutions.

We are interested here in more general mean-field FBSDE with jumps. The theory of FBSDEs with jumps has developed very dynamically in the recent years because of its variable applications. There are many works on FBSDEs with jumps, see, e.g., Bass \cite{BASS}, Barles, Buckdahn and Pardoux
\cite{BBP}, Tang and Li \cite{TL}, Li and Peng \cite{LP}, Buckdahn, Li and Hu \cite{BLH},
Li and Wei \cite{LW1}, \cite{LW2}. On the other hand, Hao and Li \cite{HL3} studied  mean-field SDEs with jumps. They showed that the unique
solution  $(X^{t,\xi},X^{t,x,P_\xi})$ of the split mean-field SDE with jumps satisfies the $\emph{flow property}$, and using a new approach the authors succeeded in proving the existence and the uniqueness of classical solutions for the related nonlocal linear integral-PDEs. Inspired by the works of Hao, Li \cite{HL3} and Pardoux, Peng \cite{PP}, the objective of our present work is to associate the mean-field (forward) SDE with jumps with a mean-field BSDE driven by a Brownian motion and an independent Poisson random measure, and to describe the associated nonlocal integral-PDE
of mean-field type which unlike \cite{HL3} and \cite{BLPR} is quasi-linear. We emphasize that this generalization is far from being trivial and related with very subtle BSDE estimates.

More precisely, given the solution of the split forward SDE $(X^{t,\xi},X^{t,x,P_\xi})$ (see the equations (\ref{equ 3.1}) and (\ref{equ 3.2})), we consider the split BSDEs with jumps (see the equations (\ref{equ 4.1}) and (\ref{equ 4.2})), driven by the Brownian motion $B$ and the independent compensated Poisson random measure $N_\lambda$ (with associated L\'evy measure $\lambda$ defined over $K \!\subset \! \mathbb{R}^\ell\setminus\{0\}$). From Theorem 10.1 in the Appendix it follows equation (\ref{equ 4.1}) has a unique solution $(Y^{t,\xi},Z^{t,\xi},H^{t,\xi})$.
Once knowing $(Y^{t,\xi},Z^{t,\xi},H^{t,\xi})$, equation (\ref{equ 4.2}) can be treated as a classical BSDE with jumps, and it possesses a unique solution
$(Y^{t,x,\xi},Z^{t,x,\xi},H^{t,x,\xi})$. We show that this solution of (\ref{equ 4.2}) depends on $\xi$ only through its law, but not on $\xi$ itself (see Proposition 4.1), which allows to define $(Y^{t,x,P_\xi},Z^{t,x,P_\xi},H^{t,x,P_\xi})=(Y^{t,x,\xi},Z^{t,x,\xi},H^{t,x,\xi})$. The flow property of $(X^{t,x,P_\xi}, X^{t,\xi})$ (see (\ref{3.5})) leads to a corresponding property for $(Y^{t,x,P_\xi}, Y^{t,\xi})$ (see (\ref{4.6-1})), which is crucial to study the related nonlocal quasi-linear integral-PDE of mean-field type. As we are interested in classical solutions of the related PDEs, we have to study the regularity of $(Y^{t,x,P_\xi},Z^{t,x,P_\xi},H^{t,x,P_\xi})$, i.e., its twice continuous differentiability with respect to $x$, its continuous differentiability with respect to the law and the continuous differentiability of this latter derivative with respect to the variable which is generated by the derivative with respect to the law. The study of these second order derivatives for a BSDE leads to new BSDEs whose driver depends, in particular, on non-linear functions of $(Y^{t,x,P_\xi},Z^{t,x,P_\xi},H^{t,x,P_\xi})$ multiplied with the square of the first order derivatives of the processes $Z^{t,x,P_\xi}$ and $H^{t,x,P_\xi}$, which are only square integrable with respect to the time parameter. This makes the proof of the continuity of the second order derivatives of these processes very subtle and is related with very technical estimates (see, in particular, Section 8, Section 10.3), a point which in their study of classical BSDEs and classical solutions of associated PDEs in \cite{PP} was not developed there. The regularity of $(Y^{t,x,P_\xi},Z^{t,x,P_\xi},H^{t,x,P_\xi})$ yields that of the value function $V$ defined by $V(t,x,P_\xi)=Y_t^{t,x,P_\xi}.$
We prove that this value function $V(t,x,P_\xi)$ is the unique classical solution of the new nonlocal quasi-linear integral-PDE of mean-field type (\ref{equ 1.4}) (see Theorem 9.2). For this we first prove a new more general Ito's formula $F(t,U_t, P_{X_t})$, where $U$ and $X$ are It\^{o} processes with jumps, respectively. In particular, unlike \cite{BLPR} and \cite{HL3} we don't need the existence of the second order mixed derivatives $\partial_x\partial_\mu F$, $\partial_\mu\partial_x F$, $\partial_\mu^2 F$ for the It\^{o} formula, see Theorem 2.1. This new It\^{o} formula simplifies the proof of Theorem 9.2, even for the more special case studied in \cite{BLPR} and \cite{HL3}. We also get the representation formulas for the solutions of (\ref{equ 4.1}) and (\ref{equ 4.2}), see (\ref{991}) and  (\ref{9911}).

This paper is organized as follows. In Section 2 we recall the definition of the derivative
of a function $\varphi$ defined on $\mathcal{P}_2(\mathbb{R}^d)$ with respect to the measure. We also prove a new general It\^{o} formula. Section 3 studies mean-field SDEs with jumps. The properties of the solution for our split mean-field BSDEs with jumps are proved in Section 4. Section 5 shows that the first order derivatives of the process $X^{t,x,P_\xi}$ with respect to $x$ and the measure
$P_\xi$ exist, and the corresponding estimates are obtained.
Section 6 is devoted to study the first order derivatives of ($Y^{t,x,P_\xi},Z^{t,x,P_\xi},H^{t,x,P_\xi}$) with respect to $x$ and the measure $P_\xi$, respectively, which are bounded and Lipschitz continuous in $L^2$. In Section 7 the second order derivatives of $X^{t,x,P_\xi}$
are discussed. The second order derivatives of ($Y^{t,x,P_\xi},Z^{t,x,P_\xi},H^{t,x,P_\xi}$)
are investigated in Section 8. In Section 9 we prove by using our new It\^{o}'s formula that our associated integral-PDEs of mean-field type has a unique classical solution. Section 10 (the Appendix) gives the proof of Theorem 2.1 (Subsection 10.1), that of an auxiliary result for Proposition 9.1 (Subsection 10.3), and recalls some basic results on mean-field BSDE with jumps (Subsection 10.2).
\section{{\protect \large {Preliminaries}}}
Let us consider a complete probability space $(\Omega,{\cal
F},P)$ on which is defined a $d$-dimensional Brownian motion
$B(=(B^1,\dots,B^d))=(B_t)_{t\in [0,T]}$, and an independent Poisson random
measure $N$ on $\mathbb{R}_{+}\times K$. Here $K\subset \mathbb{R}^l\setminus\{0\}$ is a nonempty open set equipped with
its Borel field $\mathcal{K}$. The compensator $\nu(de,dt)=\lambda(de)dt$ of $N$ is such that $\big\{N_\lambda([0,t]\times E)=(N-\nu)([0,t]\times E)\big\}_{t\geq0}$ is a martingale for all $E\in\mathcal{K}$ satisfying $\lambda(E)<\infty$, and $\lambda$ is a given $\sigma$-finite L\'{e}vy measure on $(K, \mathcal{K})$, i.e., a measure on $(K,\mathcal{K})$ with the property that
$\int_K(1\wedge |e|^2)\lambda(de)< \infty$. Let $T>0$ denote an arbitrarily fixed time
horizon. We suppose that there is a sub-$\sigma$-field ${\cal
F}_0\subset {\cal F}$ such that

i) the Brownian motion $B$ and the Poisson random measure $N$
are
 independent of ${\cal F}_0$,

ii) ${\cal F}_0$ is ``rich enough'', i.e., ${\cal P}_2({\mathbb R}^k)=
\{P_\vartheta,\, \vartheta\in L^2({\cal F}_0;{\mathbb R}^k)\},\, k\ge 1,$

iii) $\mathcal{F}_0\supset \mathcal{N}_P,$
where $\mathcal{N}_P$ is the set of all $P$-null subsets of ${\cal
F}$.

\noindent By $\mathbb{F}=({\cal F}_t)_{t\in[0,T]}$ we denote the
filtration generated by this Browinan motion $B$ and the Poisson random measure $N$, augmented by ${\cal
F}_0$, i.e.,
$$
\begin{aligned}
\mathcal{F}^0_t&=\sigma\big\{B_s,\ N([0,s]\times E)\big|\ s\leq t,
E\in\mathcal{K}\big\},\\
\mathcal{F}_t:&=\mathcal{F}^0_{t+}\vee\mathcal{F}_0
\Big(=\big(\bigcap\limits_{s:s>t}\mathcal{F}^0_s \big)\vee\mathcal{F}_0
\Big),\ t\in[0,T].\\
\end{aligned}
$$
Note that $\mathbb{F}=({\cal F}_t)_{t\in[0,T]}$ satisfies the standard assumption of right-continuity and completeness. Let us introduce the following spaces which are needed in what follows.\\
\noindent $\bullet$ $\mathcal{H}^2_{\mathbb{F}}(t,T;\mathbb{R}^d):=\big\{
\psi|\ \psi:\Omega\times [t,T]\rightarrow\mathbb{R}^d$ is an $\mathbb{F}$-predictable
process with $E[\int_t^T|\psi_s|^2ds]<+\infty\big\}$;\\
\noindent $\bullet$ $\mathcal{S}^2_{\mathbb{F}}(t,T;\mathbb{R}^d):=\big\{
\varphi|\ \varphi:\Omega\times [t,T]\rightarrow\mathbb{R}^d$ is an $\mathbb{F}$-adapted c\`adl\`ag process with
$E[\sup_{0\leq s\leq T}|\varphi_s|^2ds]<+\infty\big\}$;\\
\noindent $\bullet$  $\mathcal{K}^2_\lambda(t,T;\mathbb{R}^d):=
\{H|\ H:\Omega\times [t,T]\times K\rightarrow\mathbb{R}^d$
is $\mathcal{P}^0\otimes \mathcal{B}(K)$-measurable and $E[\int_t^T\int_K|H_s(e)|^2\lambda(de)ds]$
 $ <+\infty\}.$

 Here $t\in[0,T]$ and $\mathcal{P}^0$ denotes the $\sigma$-field of $\mathbb{F}$-predictable subsets of $\Omega\times[0,T].$ Note that we may omit $\mathbb{R}^d$ and just write $\mathcal{H}^2_{\mathbb{F}}(t,T)$ when $d=1$, similar to other notations.

Let us introduce some notations and concepts, which are used frequently in what
follows. By $\mathcal{P}(\mathbb{R}^d)$ we denote the set of probability measures over $(\mathbb{R}^d,\mathcal{B}(\mathbb{R}^d))$; $\mathcal{P}_2(\mathbb{R}^d)$ denotes the set of probability measures $\mu$ from $\mathcal{P}(\mathbb{R}^d)$ with
$\int_{\mathbb{R}^d}|x|^2\mu(dx)< +\infty.$ Let $\mathcal{P}_2(\mathbb{R}^d)$ be endowed with the 2-Wasserstein  metric: For $\nu,\ \bar{\nu}\in \mathcal{P}_2(\mathbb{R}^d),$
$$W_2(\nu,\bar{\nu}):=\inf\bigg\{\Big(\int_{\mathbb{R}^{2d}}|x-y|^2\rho(dxdy)\Big)^{\frac{1}{2}},\
  \rho\in\mathcal{P}_2(\mathbb{R}^{2d}),\ \text{such that}$$
  $$\rho(A_1\times \mathbb{R}^d)=\nu(A_1),\ A_1\in\mathcal{B}(\mathbb{R}^d),\
  \rho(\mathbb{R}^d\times A_2)=\bar{\nu}(A_2),\ A_2\in\mathcal{B}(\mathbb{R}^d)\bigg\}.$$
We now introduce the notion of differentiability of a function defined on
$\mathcal{P}_2(\mathbb{R}^d)$ with respect to probability measure. Here we adopt
the approach introduced by Lions in his course at $\emph{Coll\`{e}ge de France}$ \cite{LIONS}
and later edited in the notes by Cardaliaguet \cite{Ca1}. Given a function $\varphi: \mathcal{P}_2(\mathbb{R}^d)\rightarrow \mathbb{R},$  we consider
the lifted function $\tilde{\varphi}(\xi):=\varphi(P_\xi),\ \xi\in L^2(\mathcal{F};\mathbb{R}^d)(:=L^2(\Omega, {\cal F}, P;\mathbb{R}^d)).$
If for a given $\mu_0\in {\cal P}_2(\mathbb{R}^d)$ there exists a random variable $\xi_0\in L^2(\mathcal{F}; \mathbb{R}^d)$ satisfying  $P_{\xi_0}=\mu_0,$
such that $\tilde{\varphi}: L^2(\mathcal{F};\mathbb{R}^d)\rightarrow \mathbb{R}$ is Fr\'{e}chet differentiable
at this point $\xi_0,$ then  we called that $\varphi$ is differentiable with respect to $\mu_0$. This is equivalent with the existence of a continuous linear mapping
$D\tilde{\varphi}(\xi_0): L^2(\mathcal{F};\mathbb{R}^d)\rightarrow \mathbb{R}$
(i.e.,

\noindent$D\tilde{\varphi}(\xi_0)\in L(L^2(\mathcal{F};\mathbb{R}^d);\mathbb{R})$) such that
\begin{equation}\label{equ 2.1}
\widetilde{\varphi}(\xi_0+\zeta)-\widetilde{\varphi}(\xi_0)=D\widetilde{\varphi}(\xi_0)(\zeta)+o(|\zeta|_{L^2}),
\end{equation}
for $\zeta\in L^2(\mathcal{F};\mathbb{R}^d)$ with $|\zeta|_{L^2}\rightarrow0.$
Riesz's Representation Theorem allows to show that there exists a unique $\eta\in L^2(\mathcal{F};\mathbb{R}^d)$
such that $D\tilde{\varphi}(\xi_0)(\zeta)=E[\eta\cdot\zeta],\ \zeta\in L^2(\mathcal{F};\mathbb{R}^d).$
But this random variable $\eta$ is a Borel measurable function of $\xi_0$, refer
to Cardaliaguet \cite{Ca1}. This means that $\eta$ is of the form $\eta=\psi(\xi_0)$, where
$\psi$ is a Borel measurable function depending on $\xi_0$ only through its law. Hence, combining (\ref{equ 2.1})
and the above argument, we have
$$
\varphi(P_{\xi_0+\zeta})-\varphi(P_{\xi_0})=E[\psi(\xi_0)\cdot\zeta]+o(|\zeta|_{L^2}).
$$
In the spirit of Lions and Cardaliaguet, the derivative of
$\varphi:\mathcal{P}_2(\mathbb{R}^d)\rightarrow\mathbb{R}$
with respect to the measure $P_{\xi_0}$ is denoted by
$\partial_\mu\varphi(P_{\xi_0},y):=\psi(y),\ y\in\mathbb{R}^d$.
Observe that $\partial_\mu\varphi(P_{\xi_0},y)$ is only $P_{\xi_0}(dy)$-a.e.
uniquely determined; see also Definition 2.1 in Buckdahn, Li, Peng and Rainer \cite{BLPR}.

The following two spaces are used frequently. For more details the
reader may refer to \cite{BLPR}.
\begin{definition}\label{det 2.1}
\indent $1)$\ We say that $\varphi$ belongs to
$C^{1,1}_b(\mathcal{P}_2(\mathbb{R}^d))$,
if $\varphi: \mathcal{P}_2(\mathbb{R}^d)\rightarrow \mathbb{R}$ is differentiable on $\mathcal{P}_2(\mathbb{R}^d)$ and $\partial_{\mu}\varphi(\cdot,\cdot): {\cal P}_2(\mathbb{R}^d)\times \mathbb{R}^d \rightarrow \mathbb{R}^d$\ is bounded and Lipschitz
continuous, i.e., there exists some positive constant $L$
such that\\
\indent $\mathrm{(i)}\ |\partial_\mu \varphi(\mu,y)|\leq L,\ \mu\in\mathcal{P}_2(\mathbb{R}^d),\ y\in\mathbb{R}^d,$\\
\indent $\mathrm{(ii)}\  |\partial_\mu \varphi(\mu, y)-\partial_\mu \varphi(\mu', y')|
\leq L (W_2(\mu,\mu')+|y-y'|),\
\mu,\ \mu'\in\mathcal{P}_2(\mathbb{R}^d),\ y,\ y'\in\mathbb{R}^d.$

\noindent$2)$\
By $C^{2}_b(\mathcal{P}_2(\mathbb{R}^d))$ we denote the space of all functions
$\varphi\in C^{1,1}_b(\mathcal{P}_2(\mathbb{R}^d))$ with $(\partial_\mu \varphi)_j(\mu,\cdot):\mathbb{R}^d\rightarrow\mathbb{R}$
is differentiable, for every $\mu\in\mathcal{P}_2(\mathbb{R}^d)$, and the derivative
$\partial_y\partial_\mu \varphi: \mathcal{P}_2(\mathbb{R}^d)\times\mathbb{R}^d\rightarrow
\mathbb{R}^d\otimes\mathbb{R}^d$ is bounded and continuous.\\
Here we use the notation $\partial_\mu \varphi(\mu,y):=\Big((\partial_\mu \varphi)_j(\mu,y)\Big)_{1\leq j\leq d},$
$(\mu,y)\in\mathcal{P}_2(\mathbb{R}^d)\times\mathbb{R}^d.$
\end{definition}
\begin{definition}\label{det 2.2}
\indent We say that $\varphi$ belongs to
$C^{1,2,2} ([0, T]\times \mathbb{R}^d\times\mathcal{P}_2(\mathbb{R}^d))$,
if $\varphi: [0, T]\times \mathbb{R}^d\times\mathcal{P}_2(\mathbb{R}^d)\rightarrow \mathbb{R}$ satisfies\\
\indent {\rm (i)} $\varphi(\cdot,\cdot,\mu)\in C^{1,2}([0,T]\times\mathbb{R}^d)$, for all $\mu\in{\cal P}_2(\mathbb{R}^d)$;\\
\indent {\rm (ii)} $\varphi(t,x,\cdot)\in C^2_b({\cal P}_2(\mathbb{R}^d))$, for all $(t,x)\in [0, T]\times\mathbb{R}^d$;\\
\indent {\rm (iii)} All derivatives of order 1 and 2 are continuous on $[0,T]\times\mathbb{R}^d\times{\cal P}_2(\mathbb{R}^d)\times\mathbb{R}^d$, $\partial_\mu \varphi$  and \\
 \indent $\partial_y(\partial_\mu \varphi)$ are bounded over $[0,T]\times\mathbb{R}^d\times{\cal P}_2(\mathbb{R}^d)\times\mathbb{R}^d$.
\end{definition}
\noindent Now we give a general It\^{o}'s formula for the jump case which generalizes that in \cite{BLPR} and \cite{HL3}.
\begin{theorem}\label{minThm2.1}{\rm(}It\^{o}'s formula{\rm)}\\ Let $F\in C^{1,2,2} ([0, T]\times \mathbb{R}^d\times\mathcal{P}_2(\mathbb{R}^d))$.
We consider the following two It\^{o} processes: \\
 \begin{equation}
X_t= X_0+\int_0^t b_sds+\int_0^t\sigma_sdB_s+\int_0^t\int_K\beta_s(e)N_\lambda(ds,de),\ t\in[0,T],
\end{equation}
where $b\in \mathcal{H}^2_{\mathbb{F}}(0,T;\mathbb{R}^d)$, $\sigma\in \mathcal{H}^2_{\mathbb{F}}(0,T;\mathbb{R}^{d\times d})$, $\beta\in {\cal K}^2_\lambda(0,T;\mathbb{R}^d)$, $X_0\in L^2({\cal F}_0;\mathbb{R}^d)$, and
\begin{equation}
U_t= U_0+\int_0^t u_sds+\int_0^tv_sdB_s+\int_0^t\int_K \gamma_s(e)N_\lambda(ds,de),\ t\in[0,T],
\end{equation}
where\footnote{1) $L^0_{\mathbb{F}}(\Omega;L^1([0,T];\mathbb{R}^d))$ is the set of $\mathbb{F}$-adapted processes $u: [0, T]\times \Omega \rightarrow \mathbb{R}^d$ with $\int_0^T |u(s)|ds< +\infty$, P-a.s.;\\ 2) $L^0_{\mathbb{F}}(\Omega;L^2([0,T];\mathbb{R}^{d\times d}))$ is the set of $\mathbb{F}$-predictable  processes $v: [0, T]\times \Omega \rightarrow \mathbb{R}^{d\times d}$ with $\int_0^T |v(s)|^2ds< +\infty$, P-a.s.;\\ 3) ${\cal K}^0_\lambda(0,T;\mathbb{R}^d)$ is the set of $\mathcal{P}^0\otimes \mathcal{B}(K)$-measurable processes $\gamma: [0, T]\times \Omega \times K \rightarrow \mathbb{R}^{d}$ with $\int_0^T\int_K |\gamma_s(e)|^2\lambda(de)ds< +\infty$, P-a.s.}   $u\in L^0_{\mathbb{F}}(\Omega;L^1([0,T];\mathbb{R}^d))$, $v\in L^0_{\mathbb{F}}(\Omega;L^2([0,T];\mathbb{R}^{d\times d}))$, $\gamma\in {\cal K}^0_\lambda(0,T;\mathbb{R}^d)$ is such that $|\gamma_s(e)|\leq \zeta (1\wedge |e|)$, P-a.s., $(s,e)\in [0, T]\times K$, with $\zeta\geq0, \zeta\in L^0({\cal F}_0)$, and $U_0\in L^0({\cal F}_0;\mathbb{R}^d)$.\\
Then, for all $t\in[0,T]$ we have

\begin{equation}
\begin{aligned}
&\ F(t,U_t,P_{X_t})-F(0,U_0,P_{X_0})\\
=&\displaystyle\ \int_0^t\Big\{(\partial_s F)(s,U_s,P_{X_s})+\sum\limits_{i=1}^{d}(\partial_{x_i}F)(s,U_s,P_{X_s})u^i_s+\frac{1}{2}\sum\limits_{i,j,k=1}^{d}(\partial^2_{x_ix_j}F)(s,U_s,P_{X_s})v^{ik}_sv^{jk}_s\\
+&\displaystyle\ \int_K\Big(F(s,U_s+\gamma_s(e),P_{X_s})-F(s,U_s,P_{X_s})-\sum\limits_{i=1}^{d}(\partial_{x_i}F)(s,U_s,P_{X_s})\gamma^i_s(e)\Big)\lambda(de)\Big\}ds\\
+&\displaystyle\ \int_0^t\widehat{E}\Big[\sum\limits_{i=1}^{d}(\partial_\mu F)_i(s,U_s,P_{X_s},\widehat{X_s})\widehat{b_s^i}+\frac{1}{2}\sum\limits_{i,j,k=1}^{d}\partial_{y_i}(\partial_\mu F)_j(s,U_s,P_{X_s},\widehat{X_s})\widehat{\sigma^{ik}_s}\widehat{\sigma^{jk}_s}\\
 +&\displaystyle\ \int_K\int_0^1\sum\limits_{i=1}^{d}\Big\{(\partial_\mu F)_i(s,U_s,P_{X_s},\widehat{X_s}+\rho\widehat{\beta_s(e)})-(\partial_\mu F)_i(s,U_s,P_{X_s},P_{X_s},\widehat{X_s})\Big\}\widehat{\beta^i_s(e)}d\rho\lambda(de)\Big]ds\\
+&\displaystyle\ \int_0^t\sum\limits_{i,j=1}^{d}(\partial_{x_i} F)(s,U_s,P_{X_s})v_s^{i,j}dB_s^j+\int_0^t\int_K\Big(F(s,U_{s-}+\gamma_s(e),P_{X_s})-F(s,U_{s-},P_{X_s})\Big)N_\lambda(ds,de).
\end{aligned}
\end{equation}
\end{theorem}
\noindent Here $(\widehat{X}, \widehat{b}, \widehat{\sigma}, \widehat{\beta})$ denotes an independent copy of $(X, b, \sigma, \beta)$, defined on a probability space $(\widehat{\Omega}, \widehat{{\cal F}}, \widehat{P})$. The expectation $\widehat{E}[\cdot]$ on $(\widehat{\Omega}, \widehat{{\cal F}}, \widehat{P})$ concerns only random variables endowed with the superscript $\widehat{ }$ .
The proof of this theorem is given in the Appendix for the convenience.
\begin{remark} Observe that unlike \cite{BLPR} and \cite{HL3} we don't need the existence of the second order mixed derivatives $\partial_x\partial_\mu F$, $\partial_\mu\partial_x F$, $\partial_\mu^2 F$ for the It\^{o} formula. This is why they are neither introduced in the definition of the space $C^{1,2,2} ([0, T]\times \mathbb{R}^d\times\mathcal{P}_2(\mathbb{R}^d))$.
\end{remark}

\section{{\protect \large {Mean-field stochastic differential equations with jumps}}}

From now on let us be given deterministic Lipschitz functions $\sigma:{\mathbb R}^d\times
{\cal P}_2({\mathbb R}^d)\rightarrow {\mathbb R}^{d\times d}$, $b:{\mathbb R}^d\times
{\cal P}_2({\mathbb R}^d)\rightarrow {\mathbb R}^{d}$,
and $\beta:{\mathbb R}^d\times
{\cal P}_2({\mathbb R}^d)\times K\rightarrow {\mathbb R}^{d}$ satisfying
\smallskip

\noindent$\mathbf{Assumption\ (H3.1)}$ (i) $b$ and $\sigma$ are bounded and Lipschitz continuous on $\mathbb{R}^d\times\mathcal{P}_2(\mathbb{R}^d)$;\\
(ii) There exists a positive constant $L$ such that, for all $e\in K$, $x,\ \bar{x}\in \mathbb{R}^d,$
$\nu,\ \bar{\nu}\in \mathcal{P}_2(\mathbb{R}^d),$
$|\beta(x,\nu,e)|\leq L(1\wedge|e|),$\
$|\beta(x,\nu,e)-\beta(\bar{x},\bar{\nu},e)|\leq L(1\wedge|e|)(|x-\bar{x}|+W_2(\nu,\bar{\nu})).$
\smallskip

We consider for the
initial data $(t,x)\in[0,T]\times {\mathbb R}^d$ and $\xi\in L^2({\cal
F}_t;{\mathbb R}^d)$ the following both stochastic differential equations (SDEs) with jumps:
\begin{equation}\label{equ 3.1}
X_s^{t,\xi}=\xi+\!\!\int^s_tb(X_r^{t,\xi},P_{X_r^{t,\xi}})dr
+\!\!\int^s_t\sigma(X_r^{t,\xi},P_{X_r^{t,\xi}})dB_r +\!\!\int^s_t\int_K\beta(X_{r-}^{t,\xi},P_{X_r^{t,\xi}},e)N_{_\lambda}(dr,de),
\end{equation}
and
\begin{equation}\label{equ 3.2}
X_s^{t,x,\xi}=x+\!\!\int^s_tb(X_r^{t,x,\xi},P_{X_r^{t,\xi}})dr
+\!\!\int^s_t\sigma(X_r^{t,x,\xi},P_{X_r^{t,\xi}})dB_r+\!\!\int^s_t\int_K\beta(X_{r-}^{t,x,\xi},P_{X_r^{t,\xi}},e)N_{_\lambda}(dr,de),
\end{equation}
where \noindent $s\in[t,T]$.

We recall that under the assumption (H3.1) the both SDEs have a unique solution in ${\cal S}^2_{\mathbb{F}}(t,T;{\mathbb R}^d)$ (see, e.g., Hao and Li~\cite{HL3}). In particular, the solution $X^{t,\xi}$ of the equation (\ref{equ 3.1}) allows to determine that of (\ref{equ 3.2}), and $X^{t,x,\xi}\in {\cal S}^2_{\mathbb{F}}(t,T;{\mathbb R}^d)$ is independent of ${\cal F}_t$. As SDE standard estimates show, we have for some
$C\in {\mathbb R}_+$ depending only on the Lipschitz constants of $\sigma$,\ $b$  and $\beta$,
\begin{equation}\label{3.3} E[\sup_{s\in[t,T]}|X_s^{t,x,\xi}
-X_s^{t,x',\xi}|^2]\le C|x-x'|^2,
\end{equation}
\noindent for all $t\in[0,T],\ x,\ x'\in {\mathbb R}^d,\ \xi\in L^2(
{\cal F}_t;{\mathbb R}^d)$. This allows to substitute in (\ref{equ 3.2}) for $x$ the random variable $\xi$ and shows that $X^{t,x,\xi}|_{x=\xi}$ solves the same SDE as $X^{t,\xi}$. From the uniqueness of the solution we conclude
\begin{equation}\label{3.4}X^{t,\xi}_s=X^{t,x,\xi}_s\big|_{x=\xi}=X^{t,\xi,\xi}_s,\quad s\in[t,T].
\end{equation}
Moreover, we deduce the following {\it flow property}
\begin{equation}\label{3.5}(X^{s,X^{t,x,\xi}_s,X^{t,\xi}_s}_r,X^{s,X^{t,\xi}_s}_r)
=(X^{t,x,\xi}_r,X^{t,\xi}_r),\ r\in[s,T],
\mbox{ for all } 0\le t\le s\le T,\, x\in {\mathbb R}^d,\ \xi\in L^2({\cal
F}_t;{\mathbb R}^d).
\end{equation}
In fact, putting $\eta=X^{t,\xi}_s\in L^2({\cal
F}_s;{\mathbb R}^d)$, and considering the SDEs (\ref{equ 3.1}) and (\ref{equ 3.2}) with the initial data $(s, \eta)$\ and $(s,y)$, respectively,
\begin{equation}\label{equ 3.6}
 X^{s,\eta}_r=\eta+\int_s^rb(X^{s,\eta}_u,P_{X^{s,\eta}_u})du+\int_s^r\sigma(X^{s,\eta}_u,P_{X^{s,\eta}_u})dB_u
 +\int_s^r\beta(X^{s,\eta}_{u-},P_{X^{s,\eta}_u},e)N_\lambda(du,de)
,
\end{equation}
and
\begin{equation}\label{equ 3.7}
 X^{s,y,\eta}_r=y+\int_s^rb(X^{s,y,\eta}_u,P_{X^{s, \eta}_u})du+\int_s^r\sigma(X^{s,y,\eta}_u,P_{X^{s,\eta}_u})dB_u+\int_s^r\beta(X^{s,y,\eta}_{u-},P_{X^{s,\eta}_u},e)N_\lambda(du,de)
,
\end{equation}
\noindent$r\in[s,T]$, we get from the uniqueness of the solution of (\ref{equ 3.1}) that $X^{s,\eta}_r=X^{t,\xi}_r$, $r\in [s, T]$, and, consequently,
from the uniqueness of the solution of  (\ref{equ 3.2})  $X^{s,X^{t,x,\xi}_s,\eta}_r=X^{t,x,\xi}_r,$\
$r\in [t, T]$, i.e., we have (\ref{3.5}).

We have to show that the solution $X^{t,x,\xi}$ does not depend on $\xi$ itself
but only on its law $P_\xi$. For this, the following lemma is very useful; please refer to Hao and Li \cite{HL3}.
\begin{lemma}\label{le 3.1} For all $p\ge 2$ there is a
constant $C_p>0$ only depending on
the Lipschitz constants of $\sigma$, $b$ and $\beta$, such that we have the
following estimates
\begin{equation} \label{equ 3.8}
\begin{aligned}
&\mathrm{i)}\ E[\sup_{s\in[t,T]}|X^{t,\xi}_s-X^{t,\widehat{\xi}}_s|^p|\mathcal{F}_t]\le C_p \left(|\xi-\widehat{\xi}|^p+W_2(P_\xi,P_{\widehat{\xi}})^p\right),\\
&\mathrm{ii)}\ E[\sup_{s\in[t,T]}|X^{t,x,\xi}_s-X^{t,\widehat{x},\widehat{\xi}}_s|^p|\mathcal{F}_t]\le C_p\left(|x-\widehat{x}|^p+W_2(P_{\xi},P_{\widehat{\xi}})^p\right),\\
&\mathrm{iii)}\ E[\sup_{s\in[t,T]}|X^{t,x,\xi}_s|^p|\mathcal{F}_t]\le C_p(1+|x|^p),\\
&\mathrm{iv)}\ E[\sup_{s\in[t,T]}|X^{t,\xi}_s|^p|\mathcal{F}_t]\le C_p(1+|\xi|^p),\\
&\mathrm{v)}\ \sup_{s\in[t,T]}W_2(P_{X_s^{t,\xi}},P_{X_s^{t,\widehat{\xi}}})\leq C_2W_2(P_{\xi},P_{\widehat{\xi}}),\\
&\mathrm{vi)}\ E[\sup_{[t,t+h]}(|X_s^{t,\xi}-\xi|^p+|X_s^{t,x,\xi}-x|^p)|\mathcal{F}_t]\leq C_ph,
\end{aligned}
\end{equation}
for all $t\in [0,T],\  x,\ \widehat{x}\in {\mathbb R}^d,\ \xi,\ \widehat{\xi} \in L^2({\cal F}_t;{\mathbb R}^d).$
\end{lemma}

\begin{remark} An immediate consequence of the above
Lemma 3.1-{\emph{\rm{ii)}}} is that, given $(t,x)\in[0,T]\times {\mathbb R}^d$, the processes
$X^{t,x,\xi_1}$ and $X^{t,x,\xi_2}$ are indistinguishable, whenever the laws
of $\xi_1\in  L^2({\cal F}_t;{\mathbb R}^d)$ and $\xi_2\in  L^2({\cal F}_t;{\mathbb R}^d)$
are the same. But this means that we can define
\begin{equation}\label{3.15} X^{t,x,P_\xi}:=X^{t,x,\xi},\ (t,x)\in
[0,T]\times {\mathbb R}^d,\ \xi\in L^2({\cal F}_t;{\mathbb R}^d),\end{equation}

\noindent and, extending the notation introduced in the preceding
section for functions to random variables and processes, we
shall consider the lifted process $\displaystyle \widetilde{X}^{t,x,\xi}_s:=X^{t,x,P_\xi}_s=
X^{t,x,\xi}_s,\ s\in[t,T],\ (t,x)\in[0,T]\times {\mathbb R}^d,\ \xi\in
L^2({\cal F}_t;{\mathbb R}^d).$ However, we prefer to continue to write $X^{t,x,\xi}$ and reserve the notation
$\widetilde{X}^{t,x,P_\xi}$ for an independent copy of
$X^{t,x,P_\xi}$, which we will introduce later.\end{remark}

\section{{\protect \large {Mean-field BSDEs with jumps}}}

 In this section we consider mean-field BSDEs driven by a Brownian motion and an independent compensated Poisson random measure. The
existence and the uniqueness of the solution for this type of BSDEs is proved; for more details please refer to Section 10.2 in the Appendix.

Let $f:\mathbb{R}^d\times \mathbb{R}\times\mathbb{R}^{d}\times \mathbb{R}\times
\mathcal{P}_2(\mathbb{R}^{d}\times\mathbb{R}\times\mathbb{R}^d\times \mathbb{R})
\rightarrow \mathbb{R}$ and
$\Phi:\mathbb{R}^d\times\mathcal{P}_2(\mathbb{R}^d)\rightarrow \mathbb{R}$
be deterministic and satisfy:\\
\noindent \textbf{Assumption\ (H4.1)} The functions $f$ and $\Phi$ are bounded and Lipschitz, i.e., there exists a constant
$C>0$ such that, for all
$x,\ x'\in \mathbb{R}^d,\ y,\ y'\in \mathbb{R},\ z,\ z'\in \mathbb{R}^{d},$ $h,\ h'\in \mathbb{R},$
$\mu,\ \mu'\in \mathcal{P}_2(\mathbb{R}^{d}\times\mathbb{R}\times\mathbb{R}^d\times \mathbb{R}),$
$$
\begin{aligned}
&|f(x,y,z,h,\mu)-f(x',y',z',h',\mu')|+|\Phi(x,\mu)-\Phi(x',\mu')|\\
&\leq C(|x-x'|+|y-y'|+|z-z'|+|h-h'|+W_2(\mu,\mu')).
\end{aligned}
$$

Given $x\in \mathbb{R}^{d}$ and $\xi\in L^2(\mathcal{F}_t;\mathbb{R}^d)$ we consider the following both BSDEs with jumps:
\begin{equation}\label{equ 4.1}
\left\{
\begin{aligned}
dY_s^{t,\xi}&=
-f(\Pi^{t,\xi}_s,P_{\Pi^{t,\xi}_s})ds+Z_s^{t,\xi}dB_s+\int_KH^{t,\xi}_s(e)N_\lambda(de,ds),\ s\in[t,T],\\
Y^{t,\xi}_T&=\Phi(X_T^{t,\xi},P_{X_T^{t,\xi}}),
\end{aligned}
 \right.
\end{equation}
and
\begin{equation}\label{equ 4.2}
\left\{
\begin{aligned}
dY_s^{t,x,\xi}&=
-f(\Pi^{t,x,\xi}_s,P_{\Pi^{t,\xi}_s})ds+Z_s^{t,x,\xi}dB_s+\int_KH^{t,x,\xi}_s(e)N_\lambda(de,ds),\ s\in[t,T],\\
Y^{t,x,\xi}_T&=\Phi(X_T^{t,x,\xi},P_{X_T^{t,\xi}}),
\end{aligned}
 \right.
\end{equation}
where $$\Pi^{t,\xi}_s:=(X^{t,\xi}_s,Y^{t,\xi}_s,Z^{t,\xi}_s,\int_KH^{t,\xi}_s(e)l(e)\lambda(de)),\
\Pi^{t,x,\xi}_s:=(X^{t,x,\xi}_s,Y^{t,x,\xi}_s,Z^{t,x,\xi}_s,\int_KH^{t,x,\xi}_s(e)l(e)\lambda(de)),$$
and $l: K\rightarrow \mathbb{R}$ is a Borel function with growth condition $|l(e)|\leq C(1\wedge|e|)$, $e\in K$. Recall that the processes $X^{t,\xi}$ and $X^{t,x,\xi}$ are the solution of SDEs (\ref{equ 3.1}) and (\ref{equ 3.2}), respectively.

Under Assumption (H4.1) we know that from Theorem \ref{thli 4.12.1} in the Appendix the equation (\ref{equ 4.1})
has a unique solution $(Y^{t,\xi},Z^{t,\xi},H^{t,\xi})\in \mathcal{S}^2_{\mathbb{F}}(t,T;\mathbb{R})
\times \mathcal{H}^2_{\mathbb{F}}(t,T;\mathbb{R}^{d})\times \mathcal{K}^2_\lambda(t,T;\mathbb{R})$.
On the other hand, once having the solution of (\ref{equ 4.1}), under Assumption (H4.1) the BSDE (\ref{equ 4.2}) becomes classical and possesses a unique solution
$(Y^{t,x,\xi},$ $Z^{t,x,\xi},H^{t,x,\xi})\in \mathcal{S}^2_{\mathbb{F}}(t,T;\mathbb{R})
\times \mathcal{H}^2_{\mathbb{F}}(t,T;\mathbb{R}^{d})
\times \mathcal{K}^2_\lambda(t,T;\mathbb{R}).$

Indeed, once we have got $\Pi_s^{t,\xi}=(X_s^{t,\xi},Y_s^{t,\xi},Z_s^{t,\xi},\int_KH_s^{t,\xi}(e)l(e)\lambda(de))$, we define
$$
\begin{aligned}
\tilde{f}(s,y,z,h)=f(X_s^{t,x,\xi},y,z,h, P_{\Pi_s^{t,\xi}}),\ \tilde{\xi}=\Phi(X_T^{t,x,\xi},P_{X_T^{t,\xi}}).
\end{aligned}
$$
Obviously, $\tilde{f}$ and $\tilde{\xi}$ satisfy all assumptions of classical
BSDEs with jumps, hence, the BSDE (\ref{equ 4.2}) has a unique solution $(Y^{t,x,\xi},Z^{t,x,\xi},H^{t,x,\xi})\in \mathcal{S}^2_{\mathbb{F}}(t,T;\mathbb{R})
\times \mathcal{H}^2_{\mathbb{F}}(t,T;\mathbb{R}^{d})
\times \mathcal{K}^2_\lambda(t,T;\mathbb{R})$ (see, e.g., Li and Wei \cite{LW1}).

From the flow property (\ref{3.5}) and the uniqueness of the solution of (\ref{equ 4.1}) and (\ref{equ 4.2}) we have the following properties: For all $0\leq t\leq s\leq T,\ x\in \mathbb{R}^{d},\ \xi\in L^2({\cal
F}_t;{\mathbb R}^d),$
\begin{equation}\label{4.2-1}
\begin{array}{lll}
{\rm (i)}&  (Y^{s,X^{t,x,\xi}_s,X^{t,\xi}_s}_r,Y^{s,X^{t,\xi}_s}_r)
=(Y^{t,x,\xi}_r,Y^{t,\xi}_r),\ r\in[s,T],\ \mbox{P-a.s.};\\
{\rm (ii)}&  (Z^{s,X^{t,x,\xi}_s,X^{t,\xi}_s}_r,Z^{s,X^{t,\xi}_s}_r)
=(Z^{t,x,\xi}_r,Z^{t,\xi}_r),\  \mbox{drdP-a.e.};\\
{\rm (iii)}&  (H^{s,X^{t,x,\xi}_s,X^{t,\xi}_s}_r,H^{s,X^{t,\xi}_s}_r)
=(H^{t,x,\xi}_r,H^{t,\xi}_r),\  \mbox{drd}\lambda\mbox{dP-a.e.}\\
\end{array}
\end{equation}
\begin{proposition}\label{pro 4.3}
Suppose the Assumption (H4.1) holds true. Then, for all $p\geq2$, there exists a constant
$C_p>0$ only depending on the Lipschitz constants of $\sigma$, $b$, $\beta$, $f$ and $\Phi$, such that, for $t\in[0,T]$, $x,\ \widehat{x}\in \mathbb{R}^d,$
$\xi,\ \widehat{\xi}\in L^2(\mathcal{F}_t;\mathbb{R}^d),$
\begin{equation*}
\begin{split}
{\rm (i)}& E\big[\sup_{s\in[t,T]}|Y_{s}^{t,x,\xi}|^{p}+(\int_{t}^{T}|Z_{s}^{t,x,\xi}|^{2}ds)^{\frac{p}{2}}
+(\int_{t}^{T}\int_{K}|H_{s}^{t,x,\xi}(e)|^{2}\lambda(de)ds)^{\frac{p}{2}}|\mathcal{F}_t\big]\leq C_p;\\
{\rm (ii)}& E\big[\!\!\sup_{s\in[t,T]}|Y_{s}^{t,x,\xi}-Y_{s}^{t,\widehat{x},\widehat{\xi}}|^{p}
+(\int_{t}^{T}\!\!\!|Z_{s}^{t,x,\xi}-Z_{s}^{t,\widehat{x},\widehat{\xi}}|^{2}ds)^{\frac{p}{2}} +(\int_{t}^{T}\!\!\!\int_{K}\!\!\!|H_{s}^{t,x,\xi}(e)-H_{s}^{t,\widehat{x},\widehat{\xi}}(e)|^{2}\lambda(de)ds)^{\frac{p}{2}}|\mathcal{F}_t\big]\\
&\leq C_{p}(|x-\widehat{x}|^{p}+W_{2}(P_{\xi},P_{\widehat{\xi}})^{p});\\
{\rm (iii)}&\int_{t}^{T}W_{2}(P_{\Pi_{s}^{t,\xi}},P_{\Pi_{s}^{t,\widehat{\xi}}})^{2}ds
\leq CW_{2}(P_{\xi},P_{\widehat{\xi}})^{2}.
\end{split}
\end{equation*}
\end{proposition}
\begin{proof} From Lemma 10.1-2) we get (i) directly. Now we prove (ii) and (iii).

Notice that $\Pi^{t,x,\xi}$ is independent of $\mathcal{F}_{t}$ and, hence, of $\xi\in L^{2}(\mathcal{F}_{t};\mathbb{R}^{d})$. This allows to consider $\Pi^{t,x,\xi}\big|_{x=\xi}$,
and from the uniqueness of the solution of (\ref{equ 4.1}) and (\ref{equ 4.2}), it follows from (\ref{3.4}) that $\Pi^{t,\xi}=\Pi^{t,x,\xi}\big|_{x=\xi}$.
On the other hand, it also follows that, if $\xi'\in L^{2}(\mathcal{F}_{t};\mathbb{R}^{d})$ has the same law as $\xi$, then also $\Pi^{t,\xi',\xi}:=\Pi^{t,x,\xi}\big|_{x=\xi'}$
and $\Pi^{t,\xi}$ are of the same law. Hence, $P_{\Pi_{s}^{t,\xi}}=P_{\Pi_{s}^{t,\xi',\xi}}$, ds-a.e. Then, for given $\xi_{i}\in L^{2}(\mathcal{F}_{t};\mathbb{R}^{d})$ and $\xi_{i}'\in L^2(\mathcal{F}_t;\mathbb{R}^d)$ of the same law as $\xi_{i}$, we consider the following BSDE:
\begin{equation}\nonumber
\left\{
\begin{array}{l}
dY_s^{t,\xi_{i}',\xi_{i}}=-f(\Pi_{s}^{t,\xi_{i}',\xi_{i}},P_{\Pi_{s}^{t,\xi_{i}}})ds+Z_{s}^{t,\xi_{i}',\xi_{i}}dB_{s}
+\int_{K}H_{s}^{t,\xi_{i}',\xi_{i}}(e)N_{\lambda}(ds,de),\ s\in[t,T],\\
Y_{T}^{t,\xi_{i}',\xi_{i}}=\Phi(X_{s}^{t,\xi_{i}',\xi_{i}},P_{X_{T}^{t,\xi_{i}}}).
\end{array}
\right.
\end{equation}
From Lemma 10.1-1) and (H4.1) we have, for all $\delta>0$ (to be specified later) there exists $\beta>0$, such that
\begin{equation}\label{199}
\begin{split}
&E[\int_{t}^{T}e^{\beta(s-t)}(|Y_{s}^{t,\xi_{1}',\xi_{1}}-Y_{s}^{t,\xi_{2}',\xi_{2}}|^{2}+
|Z_{s}^{t,\xi_{1}',\xi_{1}}-Z_{s}^{t,\xi_{2}',\xi_{2}}|^{2}
+\int_{K}|H_{s}^{t,\xi_{1}',\xi_{1}}(e)-H_{s}^{t,\xi_{2}',\xi_{2}}(e)|^{2}\lambda(de))ds]\\
\leq& Ce^{\beta(T-t)}E[|X_{T}^{t,\xi_{1}',\xi_{1}}-X_{T}^{t,\xi_{2}',\xi_{2}}|^{2}+W_{2}(P_{X_{T}^{t,\xi_{1}}},P_{X_{T}^{t,\xi_{2}}})^{2}]+C^1\delta E[\int_{t}^{T}e^{\beta(s-t)}(|X_{s}^{t,\xi_{1}',\xi_{1}}-X_{s}^{t,\xi_{2}',\xi_{2}}|^{2}\\
&\hspace{15pt}+W_{2}(P_{\Pi_{s}^{t,\xi_{1}',\xi_{1}}},P_{\Pi_{s}^{t,\xi_{2}',\xi_{2}}})^{2})ds]\\
\leq& C_{\beta,\delta}\big(E[\!\!\sup_{s\in[t,T]}|X_{s}^{t,\xi_{1}',\xi_{1}}-X_{s}^{t,\xi_{2}',\xi_{2}}|^{2}]+W_{2}(P_{X_{T}^{t,\xi_{1}}},P_{X_{T}^{t,\xi_{2}}})^{2}\big)
+C^1\delta\!\! \int_{t}^{T}\!\!e^{\beta(s-t)}W_{2}(P_{\Pi_{s}^{t,\xi_{1}',\xi_{1}}},P_{\Pi_{s}^{t,\xi_{2}',\xi_{2}}})^{2}ds,
\end{split}
\end{equation}
where $C^1$ depends only on the Lipschitz constants of $f$ and $\Phi$, while $C_{\beta,\delta}$ depends also on $\beta$ and $\delta$.
From Lemma 3.1 we get
\begin{equation*}
\begin{split}
{\rm i)}\ &W_{2}(P_{X_{T}^{t,\xi_{1}}},P_{X_{T}^{t,\xi_{2}}})\leq CW_{2}(P_{\xi_{1}},P_{\xi_{2}});\\
{\rm ii)}\ &E[\sup_{s\in[t,T]}|X_{s}^{t,\xi_{1}',\xi_{1}}-X_{s}^{t,\xi_{2}',\xi_{2}}|^{2}]
=E[E[\sup_{s\in[t,T]}|X_{s}^{t,x_{1},\xi_{1}}-X_{s}^{t,x_{2},\xi_{2}}|^{2}|\mathcal{F}_{t}]\bigg|_{\stackrel{x_{1}=\xi_{1}'}{x_{2}=\xi_{2}'}}]\\
&\leq CE[|\xi_{1}'-\xi_{2}'|^{2}+W_{2}(P_{\xi_{1}},P_{\xi_{2}})^{2}].
\end{split}
\end{equation*}

\noindent Therefore, from the above (\ref{199}) and the definition of 2-Wasserstein metric we get
\begin{equation}\label{200}
\begin{split}
&E[\int_{t}^{T}e^{\beta(s-t)}(|Y_{s}^{t,\xi_{1}',\xi_{1}}-Y_{s}^{t,\xi_{2}',\xi_{2}}|^{2}+
|Z_{s}^{t,\xi_{1}',\xi_{1}}-Z_{s}^{t,\xi_{2}',\xi_{2}}|^{2}
+\int_{K}|H_{s}^{t,\xi_{1}',\xi_{1}}(e)-H_{s}^{t,\xi_{2}',\xi_{2}}(e)|^{2}\lambda(de))ds]\\
\leq& C_{\beta,\delta}E[|\xi_{1}'-\xi_{2}'|^{2}+W_{2}(P_{\xi_{1}},P_{\xi_{2}})^{2}]
+C^1\delta \int_{t}^{T}e^{\beta(s-t)}W_{2}(P_{\Pi_{s}^{t,\xi_{1}',\xi_{1}}},P_{\Pi_{s}^{t,\xi_{2}',\xi_{2}}})^{2}ds\\
\leq& C_{\beta,\delta}E[|\xi_{1}'-\xi_{2}'|^{2}+W_{2}(P_{\xi_{1}},P_{\xi_{2}})^{2}]
+C^1\delta E[\int_{t}^{T}e^{\beta(s-t)}(|Y_{s}^{t,\xi_{1}',\xi_{1}}-Y_{s}^{t,\xi_{2}',\xi_{2}}|^{2}+
|Z_{s}^{t,\xi_{1}',\xi_{1}}-Z_{s}^{t,\xi_{2}',\xi_{2}}|^{2}\\
&\ +\int_{K}|H_{s}^{t,\xi_{1}',\xi_{1}}(e)-H_{s}^{t,\xi_{2}',\xi_{2}}(e)|^{2}\lambda(de))ds].
\end{split}
\end{equation}
Now we take $\delta>0$ small enough such that $C^1\delta\leq\frac{1}{2}$ we get
\begin{equation}\label{201}
\begin{split}
&E[\int_{t}^{T}e^{\beta(s-t)}(|Y_{s}^{t,\xi_{1}',\xi_{1}}-Y_{s}^{t,\xi_{2}',\xi_{2}}|^{2}+
|Z_{s}^{t,\xi_{1}',\xi_{1}}-Z_{s}^{t,\xi_{2}',\xi_{2}}|^{2}
+\int_{K}|H_{s}^{t,\xi_{1}',\xi_{1}}(e)-H_{s}^{t,\xi_{2}',\xi_{2}}(e)|^{2}\lambda(de))ds]\\
\leq& CE[|\xi_{1}'-\xi_{2}'|^{2}+W_{2}(P_{\xi_{1}},P_{\xi_{2}})^{2}].
\end{split}
\end{equation}
Furthermore, from the properties of $W_{2}$, ii) and (\ref{201}) we get
\begin{equation*}
\begin{split}
&\int_{t}^{T}W_{2}(P_{\Pi_{s}^{t,\xi_{1}}},P_{\Pi_{s}^{t,\xi_{2}}})^{2}ds=\int_{t}^{T}W_{2}(P_{\Pi_{s}^{t,\xi_{1}',\xi_{1}}},P_{\Pi_{s}^{t,\xi_{2}',\xi_{2}}})^{2}ds\\
\leq& E[\int_{t}^{T} (|X_{s}^{t,\xi_{1}',\xi_{1}}-X_{s}^{t,\xi_{2}',\xi_{2}}|^{2}+|Y_{s}^{t,\xi_{1}',\xi_{1}}-Y_{s}^{t,\xi_{2}',\xi_{2}}|^{2}+
|Z_{s}^{t,\xi_{1}',\xi_{1}}-Z_{s}^{t,\xi_{2}',\xi_{2}}|^{2}\\
&\quad\quad +\int_{K}|H_{s}^{t,\xi_{1}',\xi_{1}}(e)-H_{s}^{t,\xi_{2}',\xi_{2}}(e)|^{2}\lambda(de))ds]\\
\leq& CE[|\xi_{1}'-\xi_{2}'|^{2}+W_{2}(P_{\xi_{1}},P_{\xi_{2}})^{2}].
\end{split}
\end{equation*}

\noindent Hence, taking the infimum over all $\xi_{1}',\ \xi_{2}'\in L^2(\mathcal{F}_t;\mathbb{R}^d)$ with $P_{\xi_{i}'}=P_{\xi_{i}}, i=1, 2$, we get
\begin{equation}\label{ast}\int_{t}^{T}W_{2}(P_{\Pi_{s}^{t,\xi_{1}}},P_{\Pi_{s}^{t,\xi_{2}}})^{2}ds
\leq CW_{2}(P_{\xi_{1}},P_{\xi_{2}})^{2},\ \xi_{1},\ \xi_{2}\in L^{2}(\mathcal{F}_{t};\mathbb{R}^{d}).\end{equation}
This allows now to apply Lemma \ref{leli 4.1}-2) to BSDE (\ref{equ 4.2}) with $g_{i}(s,y,z,h):= f(X_{s}^{t,x_{i},\xi_{i}},y,z,h,P_{\Pi_{s}^{t,\xi_{i}}})$,
$\theta_{i}:=\Phi(X_{T}^{t,x_{i},P_{\xi_{i}}},P_{X_{T}^{t,\xi_{i}}})$. Then, thanks to Lemma 3.1 and (\ref{ast}), for any $p\geq2$, there exists some $C_{p}>0$ only depending on the Lipschitz constants of $b,\ \sigma,\ \beta,\ f$ and $\Phi$, such that for any $\xi_{1},\ \xi_{2}\in$

\noindent$\in L^{2}(\mathcal{F}_{t};\mathbb{R}^{d}),\ x_{1},\ x_{2}\in\mathbb{R}^{d}$,
\begin{equation*}
\begin{split}
&E\Big[\sup_{s\in[t,T]}|Y_{s}^{t,x_{1},\xi_{1}}-Y_{s}^{t,x_{2},\xi_{2}}|^{p}
+(\int_{t}^{T}|Z_{s}^{t,x_{1},\xi_{1}}-Z_{s}^{t,x_{2},\xi_{2}}|^{2}ds)^{\frac{p}{2}}\\
&\quad\quad\quad\quad\quad\quad+(\int_{t}^{T}\int_{K}|H_{s}^{t,x_{1},\xi_{1}}(e)-H_{s}^{t,x_{2},\xi_{2}}(e)|^{2}\lambda(de)ds)^{\frac{p}{2}}|\mathcal{F}_{t}\Big]\\
&\leq C_{p}\Big(E[|X_{T}^{t,x_{1},P_{\xi_{1}}}-X_{T}^{t,x_{2},P_{\xi_{2}}}|^{p}
+W_{2}(P_{X_{T}^{t,\xi_{1}}},P_{X_{T}^{t,\xi_{2}}})^{p}\\
&\quad\quad\quad\quad\quad\quad+(\int_{t}^{T}(|X_{s}^{t,x_{1},P_{\xi_{1}}}-X_{s}^{t,x_{2},P_{\xi_{2}}}|^{2}
+W_{2}(P_{\Pi_{s}^{t,\xi_{1}}},P_{\Pi_{s}^{t,\xi_{2}}})^{2})ds)^{\frac{p}{2}}|\mathcal{F}_{t}]\Big)\\
&\leq C_{p}(|x_{1}-x_{2}|^{p}+W_{2}(P_{\xi_{1}},P_{\xi_{2}})^{p}).
\end{split}
\end{equation*}
The proof is complete.
\end{proof}
Recalling that $(Y^{t,\xi},Z^{t,\xi},H^{t,\xi})=(Y^{t,x,\xi},Z^{t,x,\xi},H^{t,x,\xi})\big|_{x=\xi}$, we have the following result.
\begin{corollary}\label{corli 4.1} Suppose the Assumption (H4.1) holds true. Then, for all $p\geq2$, there exists a constant
$C_p>0$ only depending on the Lipschitz constants of the coefficients, such that, for $t\in[0,T]$, $\xi_1,\ \xi_2\in L^2(\mathcal{F}_t;\mathbb{R}^d),$
\begin{equation*}
\begin{split}
&E\Big[\sup_{s\in[t,T]}|Y_{s}^{t,\xi_{1}}-Y_{s}^{t,\xi_{2}}|^{p}
+(\int_{t}^{T}|Z_{s}^{t,\xi_{1}}-Z_{s}^{t,\xi_{2}}|^{2}ds)^{\frac{p}{2}}
+(\int_{t}^{T}\int_{K}|H_{s}^{t,\xi_{1}}(e)-H_{s}^{t,\xi_{2}}(e)|^{2}\lambda(de)ds)^{\frac{p}{2}}\Big]\\
&\leq C_{p}(E[|\xi_{1}-\xi_{2}|^{p}]+W_{2}(P_{\xi_{1}},P_{\xi_{2}})^{p})
\leq C_{p}E[|\xi_{1}-\xi_{2}|^{p}].
\end{split}
\end{equation*}
\end{corollary}
From Proposition \ref{pro 4.3} the processes $Y^{t,x,\xi}=\{Y_s^{t,x,\xi}\}_{s\in[t,T]},
Z^{t,x,\xi}=\{Z_s^{t,x,\xi}\}_{s\in[t,T]}$ and $H^{t,x,\xi}=\{H_s^{t,x,\xi}\}_{s\in[t,T]}$
depend on $\xi$ only through its distribution, which means $(Y^{t,x,\xi},Z^{t,x,\xi},H^{t,x,\xi})$
and $(Y^{t,x,\bar{\xi}},Z^{t,x,\bar{\xi}},H^{t,x,\bar{\xi}}) $ are indistinguishable as long as $\xi$ and $\bar{\xi}$ have the
same distribution. Hence we can define $Y^{t,x,P_\xi},\ Z^{t,x,P_\xi}$ and $H^{t,x,P_\xi}$ by
$$Y^{t,x,P_\xi}:=Y^{t,x,\xi},\ Z^{t,x,P_\xi}:=Z^{t,x,\xi},\ H^{t,x,P_\xi}:=H^{t,x,\xi}.$$
And it follows from the uniqueness of the solution of BSDEs (\ref{equ 4.1}) and  (\ref{equ 4.2}) that
\begin{equation}\label{4.6-111}Y^{t,\xi}=Y^{t,x,\xi}|_{x=\xi},\ Z^{t,\xi}=Z^{t,x,\xi}|_{x=\xi},\ H^{t,\xi}=H^{t,x,\xi}|_{x=\xi}.\end{equation}
In particular, from (\ref{4.2-1}) for $0\le t\le s\le T,\, x\in {\mathbb R}^d,\ \xi\in L^2({\cal
F}_t;{\mathbb R}^d)$, it holds
\begin{equation}\label{4.6-1}
\begin{array}{lll}
{\rm (i)}&  (Y^{s,X^{t,x,\xi}_s,P_{X^{t,\xi}_s}}_r,Y^{s,X^{t,\xi}_s}_r)
=(Y^{t,x,P_\xi}_r,Y^{t,\xi}_r),\ r\in[s,T],\ \mbox{P-a.s.};\\
{\rm (ii)}&  (Z^{s,X^{t,x,\xi}_s,P_{X^{t,\xi}_s}}_r,Z^{s,X^{t,\xi}_s}_r)
=(Z^{t,x,P_\xi}_r,Z^{t,\xi}_r),\ \mbox{drdP-a.e.};\\
{\rm (iii)}&  (H^{s,X^{t,x,\xi}_s,P_{X^{t,\xi}_s}}_r,H^{s,X^{t,\xi}_s}_r)
=(H^{t,x,P_\xi}_r,H^{t,\xi}_r),\  \mbox{drd}\lambda\mbox{dP-a.e.}\\
\end{array}
\end{equation}

Now we introduce the value function
\begin{equation}\label{4.6-2}V(t,x,P_\xi):= Y^{t,x,P_\xi}_t.\end{equation}
Notice that $V(t,x,P_\xi)$ is deterministic because we are in the Markovian case. On the other hand, from Proposition \ref{pro 4.3} we can get
\begin{equation}\label{4.6-3}V(s, X_s^{t,x,P_\xi}, P_{X_s^{t,\xi}})=Y_s^{s,X_s^{t,x,P_\xi},P_{X_s^{t,\xi}}}=Y_s^{t,x,P_\xi},\ s\in[t,T].\end{equation}

An immediate consequence of Proposition \ref{pro 4.3} is
\begin{proposition}\label{pro 4.4}
For $t\in[0,T],\ x,\ \bar{x}\in \mathbb{R}^d, P_\xi,\ P_{\bar{\xi}}\in \mathcal{P}_2({\mathbb{R}^d}),$
$$
|V(t,x,P_\xi)-V(t,\bar{x},P_{\bar{\xi}})|\leq C(|x-\bar{x}|+W_2(P_{\xi},P_{\bar{\xi}})).
$$
\end{proposition}
In fact, the value function $V(t,x,P_\xi)$ is also $\frac{1}{2}$-H\"{o}lder continuous with
respect to $t$.
\begin{proposition}\label{pro 4.5}
There exists some constant $C>0$ such that, for all $t,\ t'\in[0,T],\ x\in\mathbb{R}^d$,\ $\xi\in L^2(\mathcal{F}_t;\mathbb{R}^d),$
$$
|V(t,x,P_\xi)-V(t',x,P_{\xi})|\leq C|t-t'|^\frac{1}{2}.
$$
\end{proposition}
\begin{proof} Without loss of generality let $0\leq t<t'$. Then we have
\begin{equation}\label{equli 4.10.1}
|V(t,x,P_\xi)-V(t',x,P_\xi)|\leq |E[Y_t^{t,x,P_\xi}-Y_{t'}^{t,x,P_\xi}]|+E[|Y_{t'}^{t,x,P_\xi}-Y_{t'}^{t',x,P_\xi}|].
\end{equation}
We begin with estimating $|E[Y_t^{t,x,P_\xi}-Y_{t'}^{t,x,P_\xi}]|$, as $f$ is bounded we have
\begin{equation}\label{equli 4.10.2}
|E[Y_t^{t,x,P_\xi}-Y_{t'}^{t,x,P_\xi}]|\leq E[\int_t^{t'}|f(\Pi_s^{t,x,P_\xi},P_{\Pi_s^{t,\xi}})|ds]\leq C(t'-t).\end{equation}
On the other hand, from (\ref{4.6-1}), Proposition \ref{pro 4.3} and Lemma \ref{le 3.1}-v) and vi) we have
\begin{equation}\label{equli 4.10.3}
\begin{split}
&E[|Y_{t'}^{t,x,P_\xi}-Y_{t'}^{t',x,P_\xi}|]\leq \Big(E\big[E[|Y_{t'}^{t',X_{t'}^{t,x,P_\xi},P_{X^{t,\xi}_{t'}}}-Y_{t'}^{t',x,P_\xi}|^2|\mathcal{F}_{t'}]\big]\Big)^{\frac{1}{2}}\\
\leq& C\Big(E[|X_{t'}^{t,x,P_\xi}-x|^2]+W_2(P_{X_{t'}^{t,\xi}},P_\xi)^2\Big)^{\frac{1}{2}}
\leq C\Big(E[|X_{t'}^{t,x,P_\xi}-x|^2+|X^{t,\xi}_{t'}-\xi|^2]\Big)^{\frac{1}{2}}\\
\leq& C|t'-t|^{\frac{1}{2}}.
\end{split}
\end{equation}
Hence, from (\ref{equli 4.10.1}), (\ref{equli 4.10.2}) and (\ref{equli 4.10.3}), we get $|V(t,x,P_\xi)-V(t',x,P_{\xi})|\leq C|t-t'|^\frac{1}{2}.$
\end{proof}

\section{{\protect \large {First order derivatives of $X^{t,x,P_\xi}$}}}

In this section we revisit the first order derivatives of $X^{t,x,P_\xi}$ with
respect to $x$ and the measure $P_\xi$, studied by Hao and Li \cite{HL3}. For the reader's convenience we give the main results here, for more details the reader is referred to \cite{HL3}, or \cite{BLPR} for the case without jumps.\\
\noindent \textbf{Assumption\ (H5.1)}
For each $e\in K$, the triple of  coefficients ($b,\ \sigma, \beta(\cdot,\cdot,e))$ belongs
to $C_b^{1,1}(\mathbb{R}^d\times\mathcal{P}_2(\mathbb{R}^d)\rightarrow
\mathbb{R}^d\times \mathbb{R}^{d\times d}\times\mathbb{R}^d)$, i.e.,
the components $b_j,\ \sigma_{i,j},\ \beta_j(\cdot,\cdot,e),\ 1\leq i,\ j\leq d$,
satisfy the following properties:\\
\indent (i)\ For all $x\in\mathbb{R}^d, e\in K$, $\sigma_{ij}(x,\cdot), b_j(x,\cdot),\beta_j(x,\cdot,e)\in C_b^{1,1}(\mathcal{P}_2(\mathbb{R}^d))$;\\
\indent (ii)\ For all $\nu\in\mathcal{P}_2(\mathbb{R}^d), e\in K, $
$\sigma_{ij}(\cdot,\nu), b_j(\cdot,\nu),\beta_j(\cdot,\nu,e)\in C_b^{1}(\mathbb{R}^d)$;\\
\indent(iii)\ The derivatives $\partial_x\sigma_{i,j},\ \partial_xb_j:
\mathbb{R}^d\times\mathcal{P}_2(\mathbb{R}^d)\rightarrow\mathbb{R}^d$ and
$\partial_\mu\sigma_{i,j},\ \partial_\mu b_j:\mathbb{R}^d\times
\mathcal{P}_2(\mathbb{R}^d)\times \mathbb{R}^d\rightarrow\mathbb{R}^d$
are Lipschitz continuous and bounded; \\
\indent(iv)\ There is a constant $C\in\mathbb{R}_{+}$ such that
$\partial_x\beta_j(\cdot,\cdot,e):
\mathbb{R}^d\times\mathcal{P}_2(\mathbb{R}^d)\rightarrow\mathbb{R}^d$ and
$\partial_\mu\beta_j(\cdot,\cdot,\cdot,e):\mathbb{R}^d\times
\mathcal{P}_2(\mathbb{R}^d)\times \mathbb{R}^d\rightarrow\mathbb{R}^d$
have $C(1\wedge|e|)$ as bound and as Lipschitz constant, i.e.,
 for all
$x,\ x',\ y,\ y'\in\mathbb{R}^d$,
 $\nu,\ \nu'\in\mathcal{P}_2(\mathbb{R}^d)$,
 $e\in K$,\ $1\leq j\leq d;$\\
\indent (v)\ $|\partial_{\mu}\beta_j(x,\nu,e,y)|\leq C(1\wedge|e|),$\ \
 $|\partial_{x}\beta_j(x,\nu,e)|\leq C(1\wedge|e|);$\\
\indent (vi) \
$|\partial_{x}\beta_j(x,\nu,e)-\partial_{x}\beta_j(x',\nu',e)|\leq
C(1\wedge|e|)(|x-x'|+W_2(\nu,\nu')),$\\
\mbox{}\quad\quad\quad\quad$|\partial_{\mu}\beta_j(x,\nu,e,y)-\partial_{\mu}\beta_j(x',\nu',e,y')|\leq
C(1\wedge|e|)(|x-x'|+|y-y'|+W_2(\nu,\nu')).$\\

Now we give the first order derivative of $X^{t,x,P_\xi}$ with respect to $x$.
\begin{theorem}\label{th 5.1}
Suppose Assumption (H5.1) holds true. Then the $L^2$-derivative of $X^{t,x,P_\xi}$ with respect to $x$
exists, which is denoted by $\partial_xX^{t,x,P_\xi}=(\partial_xX^{t,x,P_\xi,j})_{1\leq j\leq d}$, and it satisfies the following SDE with jumps: $s\in[t,T],\ 1\leq i,\ j\leq d,$
\begin{equation}\label{equ 5.1}
\begin{aligned}
&\partial_{x_i}X_s^{t,x,P_\xi,j}=\delta_{ij}
+\sum\limits_{k=1}^d\int^s_t\partial_{x_k}b_j(X_r^{t,x,P_{\xi}},P_{X_r^{t,\xi}})
\partial_{x_i}X_r^{t,x,P_{\xi},k}dr\\
&\quad\quad\quad\quad\quad\quad +\sum\limits_{k,l=1}^d\int^s_t\partial_{x_k}\sigma_{j,l}(X_r^{t,x,P_{\xi}},P_{X_r^{t,\xi}})
\partial_{x_i}X_r^{t,x,P_{\xi},k}dB^l_r\\
&\quad\quad\quad\quad\quad\quad +\sum\limits_{k=1}^d\int^s_t\int_K\partial_{x_k}\beta_j(X_{r-}^{t,x,P_{\xi}},P_{X_r^{t,\xi}},e)
\partial_{x_i}X_{r-}^{t,x,P_{\xi},k}N_{_\lambda}(dr,de).
\end{aligned}
\end{equation}
\end{theorem}
For the proof the reader is referred to Theorem 4.1 in \cite{HL3}, and for the case without jumps also to Theorem 3.1  in \cite{BLPR}. From the standard estimates of classical SDEs with jumps we have
\begin{proposition}\label{pro 5.1} For all $p\geq2$, there exists a constant $C_p>0$ only depending on the Lipschitz constants of
$\partial_x \sigma,\ \partial_x b$ and $\partial_x \beta$, such that, for all $t\in[0,T],\ x,\ x'\in \mathbb{R}^d,\ \xi,\ \xi'\in L^2({\cal F}_t;\mathbb{R}^d),$
P-a.s.,
\vskip0.2cm
$\mathrm{(i)}$ $E\Big[\mathop{\rm sup}\limits_{s\in[t,T]}
|\partial_x X_s^{t,x,P_\xi}|^p|\mathcal{F}_t\Big]
\leq C_p, $
\vskip0.2cm
$\mathrm{(ii)}$ $E\Big[\mathop{\rm sup}\limits_{s\in[t,T]}
|\partial_x X_s^{t,x,P_\xi}-\partial_x X_s^{t,x',P_{\xi'}}|^p|\mathcal{F}_t\Big]
\leq C_p\big(|x-x'|^p+W_2(P_\xi,P_{\xi'})^p\big),$\\
\vskip 0.2cm
\indent $\mathrm{(iii)}$
$E[\mathop{\rm sup}\limits_{s\in[t,t+h]}|\partial_xX_s^{t,x,P_\xi}-I_{d\times d}|^p|\mathcal{F}_t]\leq C_p h,
\ 0\leq t \leq t+h\leq T.$
\end{proposition}

The following theorem shows that the unique solution $X^{t,x,\xi}$ of equation (\ref{equ 3.2}) interpreted as a functional of $\xi\in L^2({\mathcal{F}_t;\mathbb{R}^d})$ is Fr\'{e}chet differentiable.

\begin{theorem}\label{th 5.2}
Let $(\sigma,b,\beta)$ satisfy Assumption (H5.1). Then for all
$0\leq t\leq s\leq T,\ x\in \mathbb{R}^d,$ the lifted process
$L^2(\mathcal{F}_t; \mathbb{R}^d)\ni\xi\rightarrow X^{t,x,\xi}_s:=X^{t,x,P_\xi}_s\in L^2(\mathcal{F}_s;\mathbb{R}^d)$
is Fr\'{e}chet differentiable, and the Fr\'{e}chet derivative is characterized by

$ \quad\quad\quad\quad\quad D X_s^{t,x,\xi}(\eta)=\tilde{E}\big[U_s^{t,x,P_\xi}(\tilde{\xi})\cdot\tilde{\eta}\big]
=\Big(\tilde{E}\big[\sum\limits_{j=1}^dU_{s,i,j}^{t,x,P_\xi}(\tilde{\xi})\cdot\tilde{\eta}_j\big] \Big)_{1\leq i\leq d},
$

\noindent for all $\eta=(\eta_1,\eta_2,\cdot\cdot\cdot,\eta_d)\in L^2(\mathcal{F}_t;\mathbb{R}^d)$, where for all
$y\in \mathbb{R}^d$,\  $U^{t,x,P_\xi}(y)=((U^{t,x,P_{\xi}}_{s,i,j}(y))_{s\in[t,T]})_{1\leq i,j\leq d}
\in \mathcal{S}^2_{\mathbb{F}}(t,T;\mathbb{R}^{d\times d})$ is the unique solution of the following SDE:
\begin{equation*}\begin{aligned}
&U^{t,x,P_\xi}_{s,i,j}(y) =\sum\limits_{k=1}^d\int^s_t\partial_{x_k} b_i(X_r^{t,x,P_\xi}, P_{X_r^{t,\xi}})U^{t,x,P_\xi}_{r,k,j}(y)dr+\sum\limits_{k,\ell=1}^d\int^s_t
\partial_{x_k}\sigma_{i,\ell}(X_r^{t,x,P_\xi}, P_{X_r^{t,\xi}})U^{t,x,P_\xi}_{r,k,j}(y)dB^{\ell}_r\\
&+\sum\limits_{k=1}^d\int^s_t\int_K\partial_{x_k} \beta_i(X_{r-}^{t,x,P_\xi}, P_{X_r^{t,\xi}},e)U^{t,x,P_\xi}_{r-,k,j}(y)
N_{_\lambda}(dr,de)\\
&+\sum\limits_{k,\ell=1}^d\int^s_t\!\!E[(\partial_\mu\sigma_{i,\ell})_k(z,P_{X_r^{t,\xi}},X_r^{t,y,P_{\xi}})
\partial_{x_j}X^{t,y,P_\xi,k}_{r}+(\partial_\mu \sigma_{i,\ell})_k(z,P_{X_r^{t,\xi}},X_r^{t,\xi})\cdot U_{r,k,j}^{t,\xi}(y)]\big|_{z=X_r^{t,x,P_\xi}}dB^{\ell}_r\\
&+\sum\limits_{k=1}^d\int^s_tE[(\partial_\mu b_i)_k(z,P_{X_r^{t,\xi}},X_r^{t,y,P_{\xi}})
\partial_{x_j}X^{t,y,P_\xi,k}_{r}
+(\partial_\mu b_{i})_k(z,P_{X_r^{t,\xi}},X_r^{t,\xi})\cdot U_{r,k,j}^{t,\xi}(y)]\big|_{z=X_r^{t,x,P_\xi}}dr\\
\end{aligned}
\end{equation*}
\begin{equation}\label{equ 5.5}\begin{aligned}
&+\sum\limits_{k=1}^d\int^s_t\int_KE[(\partial_\mu \beta_i)_k(z,P_{X_r^{t,\xi}},X_{r}^{t,y,P_{\xi}},e)
\partial_{x_j}X^{t,y,P_\xi,k}_{r}\\
&\qquad+(\partial_\mu \beta_i)_k(z,P_{X_r^{t,\xi}},X_{r}^{t,\xi},e)
\cdot U_{r,k,j}^{t,\xi}(y)]\big|_{z=X_{r-}^{t,x,P_\xi}}N_{_\lambda}(dr,de), \ s\in [t,T],\ 1\leq i,j\leq d,\\
\end{aligned}
\end{equation}
where  $U^{t,\xi}(y)=((U^{t,\xi}_{s,i,j}(y))_{s\in[t,T]})_{1\leq i,j\leq d}=U^{t,x,P_\xi}(y)|_{x=\xi}
\in \mathcal{S}^2_{\mathbb{F}}(t,T;\mathbb{R}^{d\times d})$ satisfies (\ref{equ 5.5}) with $x$ replaced by $\xi$.
\end{theorem}

\begin{proposition}\label{pro 5.2}
For every $p\geq2$, we know that there
exists a constant $C_p>0$ only depending on the Lipschitz constants of $b$ and $\sigma$,  such that,
for all $t\in[0,T]$, $x,\ x',\ y,\ y'\in\mathbb{R}^d$ and
$\xi,\ \xi'\in L^2(\mathcal{F}_t;\mathbb{R}^d),$\\

$\mathrm{(i)}$ $
E\Big[\mathop{\rm sup}\limits_{s\in[t,T]}
(|U_s^{t,x,P_\xi}(y)|^p
+|U_s^{t,\xi}(y)|^p)
\Big]
\leq C_p,$\\
\indent $\mathrm{(ii)}$
$E\Big[\mathop{\rm sup}\limits_{s\in[t,T]}
(|U_s^{t,x,P_\xi}(y)-U_s^{t,x',P_{\xi'}}(y')|^p
+
|U_s^{t,\xi}(y)-U_s^{t,\xi'}(y')|^p)
\Big]$\\
\mbox{}\hskip 1.5cm$\leq C_p\Big(
|x-x'|^p+|y-y'|^p+W_2(P_\xi,P_{\xi'})^p
\Big),$\\
\indent $\mathrm{(iii)}$
$E[\mathop{\rm sup}\limits_{s\in[t,t+h]}|U_s^{t,x,P_\xi}(y)|^p]\leq C_p h,
\ 0\leq h\leq T-t.$
\end{proposition}

For the proof of Theorem \ref{th 5.2} and Proposition \ref{pro 5.2} we refer the reader to Section 4 in \cite{HL3}.

In the spirit of Lions and Cardaliaguet (refer to \cite{LIONS}, \cite{Ca1}), the derivative of $X_s^{t,x,P_\xi}\ $ with respect to the probability measure
can be defined as follows
$$
\partial_\mu X_s^{t,x,P_\xi}(y):= U_s^{t,x,P_\xi}(y),\
s\in[t,T],\
 t\in[0,T],\
x\in\mathbb{R}^d,\ \xi\in L^2(\Omega, {\cal F}_t, P;\mathbb{R}^d), y\in\mathbb{R}^d.
$$
With this definition we have $\ DX_s^{t,x,\xi}(\eta)
=\bar{E}\big[\partial_\mu X_s^{t,x,P_\xi}(\bar{\xi})\bar{\eta}\big],\ \text{for\ all}\ \eta\in L^2({\cal F}_t;\mathbb{R}^d).$

As an immediate result of Proposition \ref{pro 5.2}, we have
\begin{proposition}\label{pro 5.3}
For all $p\geq 2,$  there exists a constant $C_p>0$ only depending on the Lipschitz constants of $b$ and $\sigma$, such that, for
$t\in [0,T],\  x,\ x'\ ,y,\ y'\in {\mathbb R}^d,\ \xi,\ {\xi'}\in L^2(\Omega, {\cal F}_t, P;\mathbb{R}^d)$,

\smallskip
\indent $\mathrm{i)}$\ $E\big[\mathop{\rm sup}\limits_{s\in[t,T]}\big|\partial_\mu X_s^{t,x,P_\xi}(y)\big|^p\big|{\cal F}_t\big]\leq C_p$;\\
\indent $\mathrm{ii)}$ \ $E\big[\mathop{\rm sup}\limits_{s\in[t,T]}\big|\partial_\mu X_s^{t,x',P_{\xi'}}(y')-\partial_\mu
X_s^{t,x,P_\xi}(y)\big|^p\big|{\cal F}_t\big]
\leq C_p(|x-x'|^p+|y-y'|^p+W_2(P_\xi,P_{\xi'})^p)$;\\
\indent $\mathrm{iii)}$\ $E\big[\mathop{\rm sup}\limits_{s\in[t,t+h]}\big|\partial_\mu X_s^{t,x,P_\xi}(y)\big|^p\big|{\cal F}_t\big]\leq C_ph,\ 0\leq h\leq T-t$.\\
\end{proposition}

\section{{\protect \large {First order derivatives of $(Y^{t,x,P_\xi},\ Z^{t,x,P_\xi},\ H^{t,x,P_\xi})$}}}

We recall from Proposition \ref{pro 4.3} that $(Y^{t,x,\xi},Z^{t,x,\xi},H^{t,x,\xi})$ depends on $\xi$
only through its law, which allows to define $(Y^{t,x,P_\xi},Z^{t,x,P_\xi},H^{t,x,P_\xi}):=(Y^{t,x,\xi},Z^{t,x,\xi},H^{t,x,\xi}).$
This section is devoted to study the first order derivatives of
$(Y^{t,x,P_\xi},\ Z^{t,x,P_\xi},\ H^{t,x,P_\xi})$
with respect to $x$ and $P_\xi$, respectively.

\smallskip

\noindent $\mathbf{Assumption\ (H6.1)}$
Let
$\Phi\in C_b^{1,1}(\mathbb{R}^d\times\mathcal{P}_2(\mathbb{R}^d))$
and
$f\in C_b^{1,1}\big(
\mathbb{R}^{d+1+d+1}\times \mathcal{P}_2(\mathbb{R}^{d+1+d+1})\big)$, i.e., $\Phi$ and $f$ satisfy:\\
 i) For all $x\in\mathbb{R}^d,\ y\in\mathbb{R},\ z\in \mathbb{R}^{d},$
$h\in \mathbb{R}$,
$\Phi(x,\cdot)\in C_b^{1,1}(\mathcal{P}_2(\mathbb{R}^d))$,
$f(x,y,z,h,\cdot)\in C_b^{1,1}(\mathcal{P}_2(\mathbb{R}^{d+1+d+1}))$;\\
 ii) For all $\nu\in \mathcal{P}_2(\mathbb{R}^d),$
 $\Phi(\cdot,\nu)\in C_b^1(\mathbb{R}^d)$, and for all
$\nu\in \mathbb{R}^{d+1+d+1}$,
$f(\cdot,\nu)\in C_b^1(\mathbb{R}^{d+1+d+1})$;\\
 iii) The derivatives $\partial_x\Phi: \mathbb{R}^d\times \mathcal{P}_2(\mathbb{R}^d)\rightarrow\mathbb{R}^d,$\ $
(\partial_x, \partial_y, \partial_z, \partial_h)f: \mathbb{R}^{d+1+d+1}\times \mathcal{P}_2(\mathbb{R}^{d+1+d+1})\rightarrow \mathbb{R}^{d+1+d+1},$
and
$
\partial_\mu \Phi:\mathbb{R}^d\times \mathcal{P}_2(\mathbb{R}^d)\times \mathbb{R}^d\rightarrow\mathbb{R}^d,\
\partial_\mu f:
{\mathbb R}^{d+1+d+1}\times {\cal P}_2({\mathbb R}^{d+1+d+1})\times{\mathbb R}^{d+1+d+1}\rightarrow {\mathbb R}^{d+1+d+1}
$
are bounded and Lipschitz continuous.

\begin{theorem}\label{th 6.1}
Under the Assumptions (H5.1) and (H6.1) the $L^2$-derivative of the solution of the equation (\ref{equ 4.2}) with respect to $x$,
$(\partial_xY^{t,x,P_\xi},\partial_xZ^{t,x,P_\xi},\partial_xH^{t,x,P_\xi})$
exists and is the unique solution of the following BSDE with jumps:
\begin{equation}\label{equ 6.1}
\begin{aligned}
&\partial_{x_i}Y_s^{t,x,P_\xi}=\sum_{\ell=1}^d\partial_{x_\ell}\Phi(X_T^{t,x,P_\xi},P_{X_T^{t,\xi}})\partial_{x_i}X_T^{t,x,P_\xi,\ell}
+\int_s^T\Big\{\sum_{\ell=1}^d\partial_{x_\ell}f\big(\Pi_r^{t,x,P_\xi},
P_{\Pi_r^{t,\xi}}\big)\partial_{x_i}X_r^{t,x,P_\xi,\ell}\\
& +\partial_y f\big(\Pi_r^{t,x,P_\xi},P_{\Pi_r^{t,\xi}}\big)\partial_{x_i}Y_r^{t,x,P_\xi}
+\sum_{\ell=1}^d\partial_{z_\ell}f\big(\Pi_r^{t,x,P_\xi},P_{\Pi_r^{t,\xi}}\big)\partial_{x_i}Z_r^{t,x,P_\xi,\ell}\\
&+\partial_h f\big(\Pi_r^{t,x,P_\xi},P_{\Pi_r^{t,\xi}}\big)\int_K\partial_{x_i}H_r^{t,x,P_\xi}(e)l(e)\lambda(de)\Big\}dr-\int_s^T\sum_{\ell=1}^d\partial_{x_i}Z_r^{t,x,P_\xi,\ell}dB_r^\ell\\
&-\int_s^T\int_K\partial_{x_i}H_r^{t,x,P_\xi}(e)N_\lambda(dr,de),\ \ s\in [t,T],\ 1\leq i\leq d,
\end{aligned}
\end{equation}
where $\displaystyle\Pi_r^{t,x,P_\xi}=\big(X_r^{t,x,P_\xi},Y_r^{t,x,P_\xi},Z_r^{t,x,P_\xi},
\int_K H_r^{t,x,P_\xi}(e)l(e)\lambda(de)\big),\ \Pi_r^{t,\xi}=\Pi_r^{t,x,P_\xi}|_{x=\xi}=\big(X_r^{t,\xi},Y_r^{t,\xi},Z_r^{t,\xi},$

\noindent$\displaystyle \int_KH_r^{t,\xi}(e)l(e)\lambda(de)\big)$.
\end{theorem}
As the $L^2$-derivative of the driving coefficient $f(\Pi_s^{t,x,P_\xi},P_{\Pi_s^{t,\xi}})$ concerns only
$\Pi_s^{t,x,P_\xi}$ but not the law $P_{\Pi_s^{t,\xi}}$, the arguments of the proof are standard; the reader
is referred, for instance, to \cite{PP}.

From Lemma 10.1 the standard estimates for classical BSDEs with jumps, combining with Lemma \ref{le 3.1}, Proposition \ref{pro 4.3}, Corollary \ref{corli 4.1} and Proposition \ref{pro 5.1} we have that, for every $p\geq2$, there exists a constant $C_p>0$ only depending on the Lipschitz constants of the coefficients such that,  for all $t\in[0,T],\ x,\ x'\in\mathbb{R}^d, P_\xi, P_{\xi'}\in\mathcal{P}_2(\mathbb{R}^d)$,
\begin{equation}\label{equ 6.11}
\begin{array}{lll}
&\!\!\!{\rm i)}\ \displaystyle E[\sup\limits_{s\in[t,T]}|\partial_x Y_s^{t,x,P_\xi}|^p+(\int_t^T|\partial_x Z_s^{t,x,P_\xi}|^2ds)^{p/2}
+(\int_t^T\int_K|\partial_x H_s^{t,x,P_\xi}|^2\lambda(de)ds)^{p/2}]\leqslant C_p;\\
&\!\!\!{\rm ii)}\ \displaystyle E[\sup\limits_{s\in[t,T]}|\partial_x Y_s^{t,x,P_\xi}-\partial_x Y_s^{t,x',P_{\xi'}}|^p
+(\int_t^T|\partial_x Z_s^{t,x,P_\xi}-\partial_x Z_s^{t,x',P_{\xi'}}|^2ds)^{p/2}\\
&\!\!\!\quad\displaystyle+(\int_t^T\int_K|\partial_x H_s^{t,x,P_\xi}(e)-\partial_x H_s^{t,x',P_{\xi'}}(e)|^2\lambda(de)ds)^{p/2}]
\leqslant C_p(|x-x'|^p+W_2(P_\xi,P_{\xi'})^p).\\
\end{array}
\end{equation}
\begin{theorem}\label{th 6.2} Assume the Assumptions (H5.1) and (H6.1) hold. Then, for all $0\le t\le s\le T,\ x\in{\mathbb R}^d,$\
the lifted processes $L^2({\cal F}_t,{\mathbb R}^d)\ni\xi\mapsto Y_s^{t,x,\xi}:= Y_s^{t,x,P_\xi}\in L^2({\cal F}_s;{\mathbb R});\ L^2({\cal F}_t,{\mathbb R}^d)\ni\xi\mapsto (Z_s^{t,x,\xi}):=(Z_s^{t,x,P_\xi})_{s\in [t,T]}\in \mathcal{H}^2_{\mathbb{F}}(t, T;{\mathbb R}^d); $
and $L^2({\cal F}_t,{\mathbb R}^d)\ni\xi\mapsto  (H_s^{t,x,\xi}):=(H_s^{t,x,P_\xi})_{s\in [t,T]}\in \mathcal{K}^2_{\lambda}(t, T;{\mathbb R}),$
are Fr\'{e}chet differentiable, with Fr\'{e}chet derivatives
\begin{equation}\begin{array}{lll}
&\!\!\!\! DY_s^{t,x,\xi}(\eta)=\overline{E}\big[O_s^{t,x,P_\xi}(\overline {\xi})\overline{\eta}\big],\ s\in [t, T],\ \mbox{P-a.s.},\
DZ_s^{t,x,\xi}(\eta)=\overline{E}\big[Q_s^{t,x,P_\xi}(\overline {\xi})\overline{\eta}\big],\ \mbox{dsdP-a.e.}, \\
&\!\!\!\! DH_s^{t,x,\xi}(\eta)=\overline{E}\big[R_s^{t,x,P_\xi}(\overline {\xi})\overline{\eta}\big],\ \mbox{dsd}\lambda\mbox{dP-a.e.},\\
\end{array}
\end{equation}
\noindent for all $\eta=(\eta_1,... ,\eta_d)\in L^2({\cal F}_t;{\mathbb R}^d),$
 where, for all $ y\in {\mathbb R}^d,$ $\big(O^{t,x,P_\xi}(y),Q^{t,x,P_\xi}(y),R^{t,x,P_\xi}(y) \big)
=\bigg(\big((O_{s,j}^{t,x,P_\xi}(y))_{s\in [t,T]}\big)_{1\le j\le d},
\big((Q_{s,i,j}^{t,x,P_\xi}(y))_{s\in [t,T]}\big)_{1\le i,j\le d},
\big((R_{s,j}^{t,x,P_\xi}(y))_{s\in [t,T]}\big)_{1\le j\le d}
\bigg)
\in S^2_{\mathbb{F}}(t,T;{\mathbb R}^d)\times \\
{\cal H}_{\mathbb{F}}^2(t,T;{\mathbb R}^{d\times d})
\times \mathcal{K}_{\lambda}^2(t,T;{\mathbb R}^{d})
$ is the unique solution of the following BSDE:
\begin{equation}\label{equ 6.13}\begin{array}{lll}
&\displaystyle\!\!\!\!\!\!   O_{s,j}^{t,x,P_\xi}(y)=\sum_{k=1}^d\partial_{x_k}\Phi (X_T^{t,x,P_\xi},P_{X_T^{t,\xi}})\partial_\mu X^{t,x,P_\xi,k}_{T,j}(y)\\
&\displaystyle\!\!\!\!\!\!   +\sum_{k=1}^dE\big[(\partial_\mu\Phi)_k(z,P_{X_T^{t,\xi}},X^{t,y,P_\xi}_r)\partial_{x_j}X_T^{t,y,P_\xi,k}+(\partial _\mu\Phi)_k(z,P_{X_T^{t,\xi}},X_T^{t,\xi})\partial_\mu X_{T,j}^{t,\xi,k}(y)\big]\big|
_{z=X_T^{t,x,P_\xi}}\\
&\displaystyle\!\!\!\!\!\!   +\int_s^T\Big[\sum_{k=1}^d\partial_{x_k}f(\Pi_r^{t,x,P_\xi},P_{\Pi_r^{t,\xi}})\partial_\mu X_{r,j}^{t,x,P_\xi,k}(y)
+\partial_{y}f(\Pi_r^{t,x,P_\xi},P_{\Pi_r^{t,\xi}}) O_{r,j}^{t,x,P_\xi}(y)\\
&\displaystyle\!\!\!\!\!\!   +\sum_{k=1}^d\partial_{z_k}f(\Pi_r^{t,x,P_\xi},P_{\Pi_r^{t,\xi}}) Q_{r,k,j}^{t,x,P_\xi}(y)
+\partial_{h}f(\Pi_r^{t,x,P_\xi},P_{\Pi_r^{t,\xi}}) \int_KR_{r,j}^{t,x,P_\xi}(y,e)l(e)\lambda(de)\Big]dr\\
&\displaystyle\!\!\!\!\!\!   +\int_s^T\sum_{k=1}^dE\Big[(\partial_\mu f)_k(z,P_{\Pi_r^{t,\xi}},
\Pi^{t,y,P_\xi}_r)\partial_{x_j}X^{t,y,P_\xi,k}_{r}
+(\partial_\mu f)_k(z,P_{\Pi_r^{t,\xi}},\Pi^{t,\xi}_r)
\partial_\mu X_{r,j}^{t,\xi,k}(y)\Big]\Big|_{z=\Pi_r^{t,x,P_\xi}}dr\\
&\displaystyle\!\!\!\!\!\!   +\int_s^T E\Big[(\partial_\mu f)_{d+1}(z,P_{\Pi_r^{t,\xi}},\Pi^{t,y,P_\xi}_r)\partial_{x_j}Y_{r}^{t,y,P_\xi}
+(\partial_\mu f)_{d+1}(z,P_{\Pi_r^{t,\xi}},\Pi^{t,\xi}_r)O_{r,j}^{t,\xi}(y)\Big]\Big|_{z=\Pi_r^{t,x,P_\xi}}dr\\
&\displaystyle\!\!\!\!\!\!   +\sum_{k=1}^d\!\!\int_s^T\!\! E\Big[(\partial_\mu f)_{d+1+k}(z,P_{\Pi_r^{t,\xi}},\Pi^{t,y,P_\xi}_r)\partial_{x_j}Z_{r}^{t,y,P_\xi,k}\! +\!(\partial_\mu f)_{d+1+k}(z,P_{\Pi_r^{t,\xi}},\Pi^{t,\xi}_r)Q_{r,k,j}^{t,\xi}(y)\Big]\Big|_{z=\Pi_r^{t,x,P_\xi}}dr\\
&\displaystyle\!\!\!\!\!\!   +\int_s^T\!\! E\big[(\partial_\mu f)_{2d+2}(z,P_{\Pi_r^{t,\xi}},\Pi_r^{t,y,P_\xi})
\int_K\!\!\partial_{x_j}H_{r}^{t,y,P_\xi}(e)l(e)\lambda(de)\\
&\displaystyle\quad \ \  \ \ \ \ \ \  +(\partial_\mu f)_{2d+2}(z,P_{\Pi_r^{t,\xi}},\Pi_r^{t,\xi})\cdot\int_K R_{r,j}^{t,\xi}(y,e)l(e)\lambda(de)\big]\big|_{z=\Pi_r^{t,x,P_\xi}}dr\\
&\displaystyle\!\!\!\!\!\!   -\sum_{k=1}^d\int_s^T Q_{r,k,j}^{t,x,P_\xi}(y)dB_r^k-\int_s^T\int_KR_{r,j}^{t,x,P_\xi}(y,e)N_\lambda(dr,de),
\   s\in [t,T],\,\,\,\,1\le j\le d,
\end{array}
\end{equation}
\noindent where
$(O^{t,\xi},$ $Q^{t,\xi},R^{t,\xi})=(O^{t,\xi,P_\xi},Q^{t,\xi,P_\xi},R^{t,\xi,P_\xi})$ is the unique solution of the above BSDE (\ref{equ 6.13}) with $x$
replaced by $\xi$.
\end{theorem}

In order to prove Theorem \ref{th 6.2} we need the following three lemmas.
For simplicity of redaction but w.l.o.g., let us restrict to the dimension $d=1$ and to $f(x,y,z,h,\gamma)=f(z,h,\gamma(\mathbb{R}\times\mathbb{R}\times\cdot))$,
$(x,y,z,h)\in\mathbb{R}\times\mathbb{R}\times\mathbb{R}\times \mathbb{R}$, $\gamma\in\mathcal{P}_2(\mathbb{R}\times\mathbb{R}\times\mathbb{R}\times \mathbb{R})$ and $\Phi(x,\gamma)=\Phi(x)$, $(x,\gamma)\in\mathbb{R}\times\mathcal{P}_2(\mathbb{R})$.
We first consider the following BSDE with jumps, which is obtained by formal differentiation of the lifted solution $(Y^{t,x,\xi+h\eta},Z^{t,x,\xi+h\eta},H^{t,x,\xi+h\eta})$ of BSDE (\ref{equ 4.2}) (with $\xi+h\eta$ instead of $\xi$, $\xi,\eta\in L^2(\mathcal{F}_t)$) with respect to $h$ at $h=0$. This formal $L^2$-differentiation (which will be made rigorous later) leads to a triple of processes $(\mathcal{O}^{t,x,\xi}(\eta),\mathcal{Q}^{t,x,\xi}(\eta),\mathcal{R}^{t,x,\xi}(\eta))$ solving the BSDE:
\begin{equation*}
\begin{array}{lll}
&\displaystyle\mathcal{O}_s^{t,x,\xi}(\eta)=\partial_x\Phi(X_T^{t,x,\xi})\mathcal{U}_T^{t,x,\xi}(\eta)\\
&\displaystyle +\int_s^T\big[
\partial_z f(\Pi_r^{t,x,\xi},P_{\Pi_r^{t,\xi}})\mathcal{Q}_r^{t,x,\xi}(\eta)+\partial_h f(\Pi_r^{t,x,\xi},P_{\Pi_r^{t,\xi}})\int_K\mathcal{R}_r^{t,x,\xi}(\eta,e)l(e)\lambda(de)\big]dr\\
&\displaystyle+\int_s^T\widehat{E}\big[(\partial_\mu f)_1(\Pi_r^{t,x,\xi},P_{\Pi_r^{t,\xi}},\widehat{\Pi}_r^{t,\widehat{\xi},P_\xi})\partial_x\widehat{Z}_r^{t,
\widehat{\xi},P_\xi}\widehat{\eta}+(\partial_\mu f)_1(\Pi_r^{t,x,\xi},P_{\Pi_r^{t,\xi}},\widehat{\Pi}_r^{t,\widehat{\xi}})\widehat{\mathcal{Q}}_r^{t,
\widehat{\xi}}(\widehat{\eta})\big]dr\\
&\displaystyle+\int_s^T\widehat{E}\big[(\partial_\mu f)_2(\Pi_r^{t,x,\xi},P_{\Pi_r^{t,\xi}}, \widehat{\Pi}_r^{t,\widehat{\xi},P_\xi})
\int_K\partial_x\widehat{H}_r^{t,\widehat{\xi},P_\xi}(e)\widehat{\eta}l(e)\lambda(de)\\
\end{array}
\end{equation*}
\begin{equation}\label{equ 6.15}
\begin{array}{lll}
&\displaystyle\quad\ \ \ \ \ \  \ \ \ \ +(\partial_\mu f)_2(\Pi_r^{t,x,\xi},P_{\Pi_r^{t,\xi}},\widehat{\Pi}_r^{t,\widehat{\xi}})
\int_K\widehat{\mathcal{\mathcal{R}}}_r^{t,\widehat{\xi}}(\widehat{\eta},e)l(e)\lambda(de)\big]dr\quad\quad\quad\quad\quad\\
&\displaystyle-\int_s^T\mathcal{Q}_r^{t,x,\xi}(\eta)dB_r-\int_s^T\int_K\mathcal{R}_r^{t,x,\xi}(\eta,e)N_\lambda(dr, de),\ s\in[t,T],\quad\quad\quad\quad\quad\\
\end{array}
\end{equation}
where $(\mathcal{O}^{t,\xi}(\eta),\mathcal{Q}^{t,\xi}(\eta),\mathcal{R}^{t,\xi}(\eta))=
(\mathcal{O}^{t,x,\xi}(\eta),\mathcal{Q}^{t,x,\xi}(\eta),\mathcal{R}^{t,x,\xi}(\eta))|_{x=\xi}$ is the solution of (\ref{equ 6.15}) for
$x$ replaced by $\xi$, and $\mathcal{U}_r^{t,x,\xi}(\eta):= DX_s^{t,x,\xi}(\eta)=\widetilde{E}[\partial_\mu X_s^{t,x,P_\xi}(\widetilde{\xi})\widetilde{\eta}]$ and $\mathcal{U}_r^{t,\xi}(\eta)=\mathcal{U}_r^{t,x,\xi}(\eta)|_{x=\xi}$, $r\in[t,T]$.
Of course, in the above BSDE we still use the notations
$$\Pi_s^{t,x,\xi}=(X_s^{t,x,\xi},Y_s^{t,x,\xi},Z_s^{t,x,\xi},\int_KH_s^{t,x,\xi}(e)l(e)\lambda(de)),\
\Pi_s^{t,\xi}=(X_s^{t,\xi},Y_s^{t,\xi},Z_s^{t,\xi},\int_KH_s^{t,\xi}(e)l(e)\lambda(de)),$$
and $(\widehat{\Omega},\widehat{\mathcal{F}},\widehat{P})$ is a  probability space carrying with
$(\widehat{\xi},\widehat{\eta},\widehat{B},\widehat{N}_\lambda)$\ an (independent)
copy of $(\xi,\eta,B,N_\lambda)$ (defined on $(\Omega,\mathcal{F},P)$);
$(\widehat{X}^{t,x,P_\xi},\widehat{Y}^{t,x,P_\xi},\widehat{Z}^{t,x,P_\xi},\widehat{H}^{t,x,P_\xi})$
(resp.,\ $(\widehat{X}^{t,\widehat{\xi}},\widehat{Y}^{t,\widehat{\xi}},\widehat{Z}^{t,\widehat{\xi}},\widehat{H}^{t,\widehat{\xi}})$)
is the solution of the same equation as that for $(X^{t,x,P_\xi},Y^{t,x,P_\xi},Z^{t,x,P_\xi},H^{t,x,P_\xi})$
 (resp.,\ $(X^{t,\xi},Y^{t,\xi},Z^{t,\xi},H^{t,\xi})$),
but with the data $(\widehat{\xi},\widehat{B},\widehat{N}_\lambda)$ instead of $(\xi,B,N_\lambda)$.

From Theorem 10.1 the equation
 (\ref{equ 6.15}) with $x$ replaced by $\xi$ has a unique solution $(\mathcal{O}^{t,\xi}(\eta),$ $\mathcal{Q}^{t,\xi}(\eta),
\mathcal{R}^{t,\xi}(\eta))\in S^2_{\mathbb{F}}(t,T)$ $\times \mathcal{H}^2_{\mathbb{F}}(t,T)\times
\mathcal{K}_\lambda^2(t,T).$  Moreover, from Theorem 10.3 we have that, for all $p\geq2,$ there exists a constant $C_p>0$ depending only on $p$ and the bounds of the coefficients, such that
\begin{equation}\label{1000}
 E[\sup\limits_{s\in[t,T]}|\mathcal{O}_s^{t,\xi}(\eta)|^p+(\int_t^T|\mathcal{Q}_s^{t,\xi}(\eta)|^2ds)^{p/2}
+(\int_t^T\int_K|\mathcal{R}_s^{t,\xi}(\eta,e)|^2\lambda(de)ds)^{p/2}]\leqslant C_p.
\end{equation}
Once having $(\mathcal{O}^{t,\xi}(\eta),\mathcal{Q}^{t,\xi}(\eta),
\mathcal{R}^{t,\xi}(\eta))$, from Theorem 10.1 and Theorem 10.3 again that (\ref{equ 6.15})
possesses a unique solution $(\mathcal{O}^{t,x,\xi}(\eta),\mathcal{Q}^{t,x,\xi}(\eta),
\mathcal{R}^{t,x,\xi}(\eta))\in S^2_{\mathbb{F}}(t,T)$ $\times \mathcal{H}^2_{\mathbb{F}}(t,T)\times
\mathcal{K}_\lambda^2(t,T)$, and that for all $p\geq2$, there is a constant $C_p>0$ only depending on
the bounds of the coefficients, such that
\begin{equation}\label{1001}
 E[\sup\limits_{s\in[t,T]}|\mathcal{O}_s^{t,x,\xi}(\eta)|^p+(\int_t^T|\mathcal{Q}_r^{t,x,\xi}(\eta)|^2dr)^{p/2}
+(\int_t^T\int_K|\mathcal{R}_r^{t,x,\xi}(\eta,e)|^2\lambda(de)dr)^{p/2}]\leqslant C_p.
\end{equation}
\begin{lemma}\label{le 6.1}
Suppose (H5.1) and  (H6.1) hold true. Then, for all $(t,x)\in [0, T]\times\mathbb{R},\ \xi\in L^2({\cal F}_t)$, there exist three stochastic processes $O^{t,x,P_\xi}(y)\in S^2_{\mathbb{F}}(t,T),\ Q^{t,x,P_\xi}(y)\in\mathcal{H}^2_{\mathbb{F}}(t,T),\
R^{t,x,P_\xi}(y)\in\mathcal{K}_\lambda^2(t,T)$, depending measurably on $y\in \mathbb{R}$,
such that
$$
\begin{aligned}
\mathcal{O}^{t,x,\xi}_s(\eta)&=\bar{E}[O_s^{t,x,P_\xi}(\bar{\xi})\cdot\bar{\eta}],\ \mbox{P-a.s.},\ s\in[t,T],\ \  \mathcal{Q}^{t,x,\xi}_s(\eta)=\bar{E}[Q_s^{t,x,P_\xi}(\bar{\xi})\cdot\bar{\eta}],\ \mbox{dsdP-a.e.}, \\
\mathcal{R}^{t,x,\xi}_s(\eta)&=\bar{E}[R_s^{t,x,P_\xi}(\bar{\xi})\cdot\bar{\eta}],\ \mbox{dsd}\lambda\mbox{dP-a.e.}\
\end{aligned}
$$
In particular, for all $x\in\mathbb{R}$, $0\leq t\leq s\leq T$, $\xi\in L^2(\mathcal{F}_t)$, the mappings
$$
\begin{aligned}
\mathcal{O}_s^{t,x,\xi}(\cdot): L^2(\mathcal{F}_t)\mapsto L^2(\mathcal{F}_s),\
\mathcal{Q}^{t,x,\xi}(\cdot): L^2(\mathcal{F}_t)\mapsto \mathcal{H}^2_{\mathbb{F}}(t,T),\
\mathcal{R}^{t,x,\xi}(\cdot): L^2(\mathcal{F}_t)\mapsto \mathcal{K}^2_\lambda(t,T),
\end{aligned}
$$
are linear and continuous.
\end{lemma}
\begin{remark}\label{re 6.1}
For $(\mathcal{O}^{t,\xi}_s(y),\mathcal{Q}^{t,\xi}_s(y),\mathcal{R}^{t,\xi}_s(y)):=
(\mathcal{O}^{t,x,\xi}_s(y),\mathcal{Q}^{t,x,\xi}_s(y),\mathcal{R}^{t,x,\xi}_s(y))|_{x=\xi},\ s\in[t, T],\ \xi\in L^2(\mathcal{F}_t),\ y\in \mathbb{R}$, we
see directly from Lemma \ref{le 6.1} that
$$
\begin{aligned}
\mathcal{O}^{t,\xi}_s(\eta)&=\bar{E}[O_s^{t,\xi}(\bar{\xi})\cdot\bar{\eta}],\ \mbox{P-a.s.},\  s\in[t,T],\ \
\mathcal{Q}^{t,\xi}_s(\eta)=\bar{E}[Q_s^{t,\xi}(\bar{\xi})\cdot\bar{\eta}],\ \mbox{dsdP-a.e.}, \\
\mathcal{R}^{t,\xi}_s(\eta)&=\bar{E}[R_s^{t,\xi}(\bar{\xi})\cdot\bar{\eta}],\ \mbox{dsd}\lambda\mbox{dP-a.e.},\ \eta\in L^2(\mathcal{F}_t).
\end{aligned}
$$
\end{remark}
\begin{proof}
For $y\in \mathbb{R}$,
let
$(O^{t,x,P_\xi}(y), Q^{t,x,P_\xi}(y),R^{t,x,P_\xi}(y))\in \mathcal{S}^2_\mathbb{F}(t,T)\times \mathcal{H}^2_{\mathbb{F}}(t,T)\times
\mathcal{K}^2_\lambda(t,T)$ be the unique solution of BSDE (\ref{equ 6.13}), which, for our special case ($d=1$ and $f=f(z,h,\gamma(\mathbb{R}\times\mathbb{R}\times \cdot))$), writes as follows
\begin{equation}\label{equ 6.17}
\begin{array}{lll}
&O_s^{t,x,P_\xi}(y)
=\partial_x\Phi(X_T^{t,x,\xi})\partial_\mu X_T^{t,x,P_\xi}(y)\\
&\displaystyle+\int_s^T\big[\partial_z f(\Pi_r^{t,x,\xi},P_{\Pi_r^{t,\xi}})Q_r^{t,x,P_\xi}(y)+\partial_h f(\Pi_r^{t,x,\xi},P_{\Pi_r^{t,\xi}})
\int_KR_r^{t,x,P_\xi}(y,e)l(e)\lambda(de)\big]dr\\
&\displaystyle+\int_s^T\widehat{E}\big[(\partial_\mu f)_1(\Pi_r^{t,x,\xi},P_{\Pi_r^{t,\xi}},\widehat{\Pi}_r^{t,y,P_\xi})\partial_x\widehat{Z}_r^{t,
y,P_\xi}+(\partial_\mu f)_1(\Pi_r^{t,x,\xi},P_{\Pi_r^{t,\xi}},\widehat{\Pi}_r^{t,\widehat{\xi}})\widehat{Q}_r^{t,
\widehat{\xi}}(y)\big]dr\\
&\displaystyle+\int_s^T\widehat{E}\big[(\partial_\mu f)_2(\Pi_r^{t,x,\xi},P_{\Pi_r^{t,\xi}},\widehat{\Pi}_r^{t,y,P_\xi})
\int_K\partial_x\widehat{H}_r^{t,y,P_\xi}(e)l(e)\lambda(de)\\
&\displaystyle\ \ \  \ \ \ \ \ \ +(\partial_\mu f)_2(\Pi_r^{t,x,\xi},P_{\Pi_r^{t,\xi}},\widehat{\Pi}_r^{t,\widehat{\xi}})
\int_K\widehat{R}_r^{t,\widehat{\xi}}(y,e)l(e)\lambda(de)
\big]dr\\
&\displaystyle-\int_s^TQ_r^{t,x,P_\xi}(y)dB_r-\int_s^T\int_KR_r^{t,x,P_\xi}(y,e)N_\lambda(dr, de),\ s\in[t,T],
\end{array}
\end{equation}
where $(O^{t,\xi}(y),Q^{t,\xi}(y),R^{t,\xi}(y)):=(O^{t,x,\xi}(y),Q^{t,x,\xi}(y),R^{t,x,\xi}(y))|_{x=\xi}\in
\mathcal{S}_{\mathbb{F}}^2(t,T)\times\mathcal{H}_{\mathbb{F}}^2(t,T)\times\mathcal{K}_{\lambda}^2(t,T)$
is the unique solution of (\ref{equ 6.17}) with $x$ replaced by $\xi$,
$\Pi_s^{t,x,\xi}=(Z_s^{t,x,\xi},\int_KH_s^{t,x,\xi}(e)$ $l(e)\lambda(de)),$ and
$\Pi_s^{t,\xi}=\Pi_s^{t,x,\xi}|_{x=\xi}$. It follows from Theorem 10.3 that, for any $p\geq2$,
there is some constant $C_p>0$ only depending on
the bounds of the coefficients such that,
for all $t\in[0,T],\ x\in\mathbb{R},\ \xi\in L^2(\mathcal{F}_t;\mathbb{R}), y\in\mathbb{R}$,
\begin{equation}\label{equ 6.19}
\begin{aligned}
&E[\sup\limits_{s\in[t,T]}|O_s^{t,\xi}(y)|^p+(\int_t^T|Q_s^{t,\xi}(y)|^2ds)^{p/2}
+(\int_t^T\int_K|R_s^{t,\xi}(y,e)|^2\lambda(de)ds)^{p/2}]\leqslant C_p,
\end{aligned}
\end{equation}
then again from Theorem 10.3 we get
\begin{equation}\label{1002}
\begin{aligned}
&E[\sup\limits_{s\in[t,T]}|O_s^{t,x,P_\xi}(y)|^{p}+(\int_t^T|Q_s^{t,x,P_\xi}(y)|^2ds)^{\frac{p}{2}}
+(\int_t^T\int_K|R_s^{t,x,P_\xi}(y,e)|^2\lambda(de)ds)^{\frac{p}{2}}]\leqslant C_p.\\
\end{aligned}
\end{equation}
Let the couple $(\bar{\xi},\bar{\eta})$ defined on some probability space $(\bar{\Omega},\bar{\mathcal{F}},\bar{P})$ be
an independent copy of $(\xi,\eta)$ on $(\Omega,\mathcal{F},P)$ and, in particular, also an independent copy of $(\hat{\xi}, \hat{\eta})$
on $(\hat{\Omega},\hat{\mathcal{F}},\hat{P}$). Substituting in (\ref{equ 6.17}) for $x$ the random variable $\xi$ and for $y$ the random variable  $\bar{\xi},$
and then multiplying $\bar{\eta}$ on both sides of the such obtained equation and taking expectation $\bar{E}[\cdot],$
we obtain
\begin{equation}\label{equ 6.20}
\begin{array}{lll}
&\displaystyle\bar{E}[O_s^{t,\xi}(\bar{\xi})\cdot\bar{\eta}]=\partial_x\Phi(X_T^{t,\xi})\bar{E}[\partial_\mu X_T^{t,\xi}(\bar{\xi})\cdot\bar{\eta}]\\
&\displaystyle+\bar{E}\bigg[\int_s^T\big[\partial_z f(\Pi_r^{t,\xi},P_{\Pi_r^{t,\xi}})Q_r^{t,\xi}(\bar{\xi})\cdot\bar{\eta}
+\partial_h f(\Pi_r^{t,\xi},P_{\Pi_r^{t,\xi}})\int_KR_r^{t,\xi}(\bar{\xi},e)\cdot\bar{\eta}l(e)\lambda(de)\big]dr\bigg]\\
&\displaystyle+\bar{E}\bigg\{\int_s^T\widehat{E}\big[(\partial_\mu f)_1(\Pi_r^{t,\xi},P_{\Pi_r^{t,\xi}},\widehat{\Pi}_r^{t,\bar{\xi},P_\xi})
\partial_x\widehat{Z}_r^{t,\bar{\xi},P_\xi}\cdot\bar{\eta}
+(\partial_\mu f)_1(\Pi_r^{t,\xi},P_{\Pi_r^{t,\xi}},\widehat{\Pi}_r^{t,\widehat{\xi}})\widehat{Q}_r^{t,
\widehat{\xi}}(\bar{\xi})\cdot\bar{\eta}\big]dr\\
&\displaystyle\  \ \ \ \ \ +\int_s^T\widehat{E}\big[(\partial_\mu f)_2(\Pi_r^{t,\xi},P_{\Pi_r^{t,\xi}},\widehat{\Pi}_r^{t,\bar{\xi},P_\xi})
\int_K\partial_x\widehat{H}_r^{t,\bar{\xi},P_\xi}(e)\cdot\bar{\eta}l(e)\lambda(de)\\
&\displaystyle\  \ \ \ \ \ \ \ \ +(\partial_\mu f)_2(\Pi_r^{t,\xi},P_{\Pi_r^{t,\xi}},\widehat{\Pi}_r^{t,\widehat{\xi}})
\int_K\widehat{R}_r^{t,\widehat{\xi}}(\bar{\xi},e)\cdot\bar{\eta}l(e)\lambda(de)
\big]dr\bigg\}\\
&\displaystyle-\bar{E}[\int_s^TQ_r^{t,\xi}(\bar{\xi})\cdot\bar{\eta}dB_r]
-\bar{E}[\int_s^T\int_KR_r^{t,\xi}(\bar{\xi},e)\cdot\bar{\eta}N_\lambda(dr, de)],\ s\in[t,T].\quad\quad\quad
\end{array}
\end{equation}
Since $(\bar{\xi},\bar{\eta})$
is independent of $(\xi,\eta,\Pi^{t,x,\xi})$ and $(\hat{\xi},\hat{\eta},\hat{\Pi}^{t,x,P_\xi})$, and of the same law as  $(\hat{\xi},\hat{\eta})$,
we have
$$
\begin{aligned}
&{\rm{i)}}\bar{E}\big[\widehat{E}\big[(\partial_\mu f)_1(\Pi_r^{t,\xi},P_{\Pi_r^{t,\xi}},
\widehat{\Pi}_r^{t,\bar{\xi},P_\xi})
\partial_x\widehat{Z}_r^{t,\bar{\xi},P_\xi}\cdot\bar{\eta}\big]\big]
=\widehat{E}[
(\partial_\mu f)_1(\Pi_r^{t,\xi},P_{\Pi_r^{t,\xi}},\widehat{\Pi}_r^{t,\widehat{\xi},P_\xi})
\partial_x\widehat{Z}_r^{t,\widehat{\xi},P_\xi}\cdot\widehat{\eta}];\\
&{\rm{ii)}}\bar{E}\big[\widehat{E}\big[(\partial_\mu f)_2(\Pi_r^{t,\xi},P_{\Pi_r^{t,\xi}},
\widehat{\Pi}_r^{t,\bar{\xi},P_\xi})
\int_K\partial_x\widehat{H}_r^{t,\bar{\xi},P_\xi}(e)\cdot\bar{\eta}]l(e)\lambda(de)\big]\\
&\quad=\widehat{E}[
(\partial_\mu f)_2(\Pi_r^{t,\xi},P_{\Pi_r^{t,\xi}},
\widehat{\Pi}_r^{t,\widehat{\xi},P_\xi})
\int_K\partial_x\widehat{H}_r^{t,\widehat{\xi},P_\xi}(e)\cdot\widehat{\eta}l(e)\lambda(de)],
\end{aligned}
$$
similar to other terms. From the above equalities and the uniqueness of the solution of equation (\ref{equ 6.15}) with $x$ replaced by $\xi$ it follows
\begin{equation}
\begin{aligned}
&\mathcal{O}_s^{t,\xi}(\eta)=\bar{E}[O_s^{t,\xi}(\bar{\xi})\cdot\bar{\eta}],\ \mbox{P-a.s.},\ s\in[t,T],\ \
\mathcal{Q}_s^{t,\xi}(\eta)=\bar{E}[Q_s^{t,\xi}(\bar{\xi})\cdot\bar{\eta}],\ \mbox{dsdP-a.e.},\\
&\mathcal{R}_s^{t,\xi}(\eta)=\bar{E}[R_s^{t,\xi}(\bar{\xi})\cdot\bar{\eta}],\ \mbox{dsd}\lambda\mbox{dP-a.e.}
\end{aligned}
\end{equation}
Furthermore, from (\ref{equ 6.19}) we get
\begin{equation}\label{1003}
\begin{aligned}
&E[|\mathcal{O}_s^{t,\xi}(\eta)|^2]=E[|\bar{E}[O_s^{t,\xi}(\bar{\xi})\cdot\bar{\eta}]|^2]
\leq \bar{E}[E[|O_s^{t,\xi}(\bar{\xi})|^2\cdot|\bar{\eta}|^2]]\\
&=\bar{E}[E[|O_s^{t,\xi}(y)|^2]|_{y=\bar{\xi}}\cdot|\bar{\eta}|^2]
\leq C\bar{E}[|\bar{\eta}|^2]=CE[|\eta|^2].
\end{aligned}
\end{equation}
That means $|\mathcal{O}_s^{t,\xi}(\eta)|_{L^2}\leq C|\eta|_{L^2},$\ for every $\eta\in L^2(\mathcal{F}_t).$
Hence, $\mathcal{O}_s^{t,\xi}(\cdot): L^2(\mathcal{F}_t)\rightarrow L^2({\mathcal{F}_s})$ is a linear and
continuous mapping, for all $s\in[t,T]$, and $|\mathcal{O}_s^{t,\xi}(\cdot)|_{L(L^2,L^2)}\leq C$.\\
Furthermore, also
\begin{equation}\label{1004}
\begin{aligned}
&E[\int_t^T|\mathcal{Q}_s^{t,\xi}(\eta)|^2ds]=E[\int_t^T|\bar{E}[Q_s^{t,\xi}(\bar{\xi})\cdot\bar{\eta}]|^2ds]
\leq \bar{E}[E[\int_t^T|Q_s^{t,\xi}(\bar{\xi})|^2ds\cdot|\bar{\eta}|^2]]\\
&=\bar{E}[E[\int_t^T|Q_s^{t,\xi}(y)|^2ds]|_{y=\bar{\xi}}\cdot|\bar{\eta}|^2]
\leq CE[|\eta|^2],
\end{aligned}
\end{equation}
and

\begin{equation}\label{1005}
\begin{aligned}
&E[\int_t^T\int_K|\mathcal{R}_s^{t,\xi}(\eta,e)|^2\lambda(de)ds]
=E[\int_t^T\int_K|\bar{E}[R_s^{t,\xi}(\bar{\xi},e)\cdot\bar{\eta}]|^2\lambda(de)ds]\\
&\leq \bar{E}[\int_t^T\int_K E[|R_s^{t,\xi}(\bar{\xi},e)|^2\cdot|\bar{\eta}|^2]\lambda(de)ds]
=\bar{E}[E[\int_t^T\int_K|R_s^{t,\xi}(y,e)|^2\lambda(de)ds]|_{y=\bar{\xi}}\cdot|\bar{\eta}|^2]\\
&\leq CE[|\eta|^2].
\end{aligned}
\end{equation}
Therefore, $\mathcal{Q}^{t,\xi}(\cdot):L^2(\mathcal{F}_t)\mapsto \mathcal{H}^2_{\mathbb{F}}(t,T)$ and
$\mathcal{R}^{t,\xi}(\cdot): L^2(\mathcal{F}_t)\mapsto \mathcal{K}^2_\lambda(t,T)$ are continuous linear
mappings. Making use of the above argument, but for
$(\mathcal{O}^{t,x,P_\xi}(\eta),\ \mathcal{Q}^{t,x,P_\xi}(\eta),\ \mathcal{R}^{t,x,P_\xi}(\eta))$
instead of $(\mathcal{O}^{t,\xi}(\eta),\ \mathcal{Q}^{t,\xi}(\eta),\ \mathcal{R}^{t,\xi}(\eta)),$
we also have, for all $\eta\in L^2(\mathcal{F}_t)$,
\begin{equation}
\begin{aligned}
&\mathcal{O}_s^{t,x,\xi}(\eta)=\bar{E}[O_s^{t,x,\xi}(\bar{\xi})\cdot\bar{\eta}],\ \mbox{P-a.s.},\ s\in[t,T], \ \
\mathcal{Q}_s^{t,x,\xi}(\eta)=\bar{E}[Q_s^{t,x,\xi}(\bar{\xi})\cdot\bar{\eta}],\ \mbox{dsdP-a.e.},\\
&\mathcal{R}_s^{t,x,\xi}(\eta)=\bar{E}[R_s^{t,x,\xi}(\bar{\xi})\cdot\bar{\eta}],\ \mbox{dsd}\lambda\mbox{dP-a.e.}
\end{aligned}
\end{equation}
Moreover, by using (\ref{1002}), similar to (\ref{1003}), (\ref{1004}) and (\ref{1005}), we
obtain that $\mathcal{O}^{t,x,P_\xi}(\cdot),\mathcal{Q}^{t,x,P_\xi}(\cdot),$ $ \mathcal{R}^{t,x,P_\xi}(\cdot)$
are  also linear and continuous mappings (over the same spaces as $\mathcal{O}^{t,\xi}$, $\mathcal{Q}^{t,\xi}$, $\mathcal{R}^{t,\xi}$).
\end{proof}

Now we prove the following estimate for the solution of equation (\ref{equ 6.17}).
\begin{proposition}\label{pro 6.1}
For all $p\geq1$, there exists a constant $C_p>0$ only depending on the Lipschitz constant of the coefficients, such that, for all $t\in[0,T],\ x,\  {x'},\ y,\ {y'}\in\mathbb{R}^d,$
and $\xi, \xi^{'}\in L^2(\mathcal{F}_t; \mathbb{R}^d)$,
\begin{equation}\label{equ 6.23}
\begin{aligned}
& E[\sup\limits_{s\in[t,T]}|O_s^{t,x,P_\xi}(y)-O_s^{t,{x'},P_{\xi^{'}}}(y')|^{2p}
+(\int_t^T|Q_s^{t,x,P_\xi}(y)-Q_s^{t,{x'},P_{\xi^{'}}}(y')|^2ds)^{p}\\
&\ \ +(\int_t^T\int_K|R_s^{t,x,P_\xi}(y,e)-R_s^{t,{x'},P_{\xi^{'}}}(y',e)|^2\lambda(de)ds)^{p}]
 \leq C_p(|x-x'|^{2p}+|y-y'|^{2p}+W_2(P_\xi,P_{\xi^{'}})^{2p}).
\end{aligned}
\end{equation}
\end{proposition}
\begin{proof}
Recall that for simplicity of redaction, $d=1$ and $f(\Pi_r^{t,x,\xi},P_{\Pi_r^{t,\xi}})$ depends only on $\Pi_r^{t,x,\xi}=(Z_r^{t,x,\xi},$ $\int_KH_r^{t,x,\xi}(e)l(e)\lambda(de))$ and $P_{\Pi_r^{t,\xi}}$; $ \Pi_r^{t,x,\xi}=\Pi_r^{t,x,P_{\xi}}$ and $\Pi_r^{t,\xi}=\Pi_r^{t,\xi,P_{\xi}}$.

Let $\xi,\xi',\vartheta,\vartheta'\in L^2(\mathcal{F}_t)$ be such that $P_{\vartheta}=P_{\xi}$, $P_{\vartheta'}=P_{\xi'}$. Notice that $\Pi_s^{t,x,P_{\xi}}$ and  $(O_s^{t,x,P_{\xi}}(y),$ $ Q_s^{t,x,P_{\xi}}(y), R_s^{t,x,P_{\xi}}(y)),\ t\leq s\leq T, $ are independent of ${\cal F}_t$. Hence, from (\ref{equ 6.17}) we get
the following BSDE:
\begin{equation}\label{equ 4.11.1}
\begin{split}
&O_s^{t,x,P_{\xi}}(y)-O_s^{t,x',P_{\xi'}}(y')=\Xi(x,x')+\int_s^TR(r,x,x')dr\\
&+\left.\int_s^T\right\{(\partial_zf)(\Pi_r^{t,x,\xi},P_{\Pi_r^{t,\xi}})(Q_r^{t,x,P_\xi}(y)-Q_r^{t,x',P_{\xi'}}(y'))\\
&+(\partial_hf)(\Pi_r^{t,x,\xi},P_{\Pi_r^{t,\xi}})\big(\int_K(R_r^{t,x,P_\xi}(y,e)-R_r^{t,x',P_{\xi'}}(y',e))l(e)
\lambda(de)\big)\\
&+\widehat{E}[(\partial_\mu f)_1(\Pi_r^{t,x,\xi},P_{\Pi_r^{t,\xi}},\widehat{\Pi}_r^{t,\widehat{\vartheta},P
_\xi})(\widehat{Q}_r^{t,\widehat{\vartheta},P_\xi}(y)-\widehat{Q}_r^{t,\widehat{\vartheta'},P_{\xi'}}(y'))]\\
&+\widehat{E}[(\partial_\mu f)_2(\Pi_r^{t,x,\xi},P_{\Pi_r^{t,\xi}},\widehat{\Pi}_r^{t,\widehat{\vartheta},P
_\xi})\big(\left.\int_K(\widehat{R}_r^{t,\widehat{\vartheta},P_\xi}(y,e)-\widehat{R}_r^{t,\widehat{\vartheta'},P_{\xi'}}(y',e))l(e)
\lambda(de)\big)]\right\}dr\\
&-\int_s^T(Q_r^{t,x,P_\xi}(y)-Q_r^{t,x',P_{\xi'}}(y'))dB_r-\int_s^T\int_K(R_r^{t,x,P_\xi}(y,e)-R_r^{t,x',P_{\xi'}}(y',e))N_\lambda(dr,de),\\
\end{split}
\end{equation}
where
\begin{equation*}
\begin{split}
&R(r,x,x')=\big((\partial_zf)(\Pi_r^{t,x,\xi},P_{\Pi_r^{t,\xi}})-(\partial_zf)(\Pi_r^{t,x',\xi'},P_{\Pi_r^{t,\xi'}})\big)Q_r^{t,x',P_{\xi'}}(y')\\
&+\big((\partial_hf)(\Pi_r^{t,x,\xi},P_{\Pi_r^{t,\xi}})-(\partial_hf)(\Pi_r^{t,x',\xi'},P_{\Pi_r^{t,\xi'}})\big)\cdot\int_K R_r^{t,x',P_{\xi'}}(y',e)l(e)\lambda(de)\\
&+\widehat{E}[\big((\partial_\mu f)_1(\Pi_r^{t,x,\xi},P_{\Pi_r^{t,\xi}},\widehat{\Pi}_r^{t,\widehat{\vartheta},P
_\xi})-(\partial_\mu f)_1(\Pi_r^{t,x',\xi'},P_{\Pi_r^{t,\xi'}},\widehat{\Pi}_r^{t,\widehat{\vartheta'},P_{\xi'}})\big)\widehat{Q}_r^{t,\widehat{\vartheta'},P_{\xi'}}(y')]\\
&+\widehat{E}[\big((\partial_\mu f)_2(\Pi_r^{t,x,\xi},P_{\Pi_r^{t,\xi}},\widehat{\Pi}_r^{t,\widehat{\vartheta},P_\xi})-(\partial_\mu f)_2(\Pi_r^{t,x',\xi'},P_{\Pi_r^{t,\xi'}},\widehat{\Pi}_r^{t,\widehat{\vartheta'},P_{\xi'}})\big)\int_K\widehat{R}_r^{t,\widehat{\vartheta'},P_{\xi'}}(y',e)l(e)\lambda(de)]\\
&+\widehat{E}[(\partial_\mu f)_1(\Pi_r^{t,x,\xi},P_{\Pi_r^{t,\xi}},\widehat{\Pi}_r^{t,y,P_\xi})\partial_x\widehat{Z}_r^{t,y,P_\xi}-(\partial_\mu f)_1(\Pi_r^{t,x',\xi'},P_{\Pi_r^{t,\xi'}},\widehat{\Pi}_r^{t,y',P_{\xi'}})\partial_x\widehat{Z}_r^{t,y',P_{\xi'}}]\\
&+\widehat{E}[(\partial_\mu f)_2(\Pi_r^{t,x,\xi},P_{\Pi_r^{t,\xi}},\widehat{\Pi}_r^{t,y,P_\xi})\big(\int_K\partial_x\widehat{H}_r^{t,y,P_\xi}(e)l(e)\lambda(de)\big)\\
&\ \ \ \ -(\partial_\mu f)_2(\Pi_r^{t,x',\xi'},P_{\Pi_r^{t,\xi'}},\widehat{\Pi}_r^{t,y',P_{\xi'}})\big(\int_K\partial_x\widehat{H}_r^{t,y',P_{\xi'}}(e)l(e)\lambda(de)\big)],
\end{split}
\end{equation*}
and
\begin{equation*}
\Xi(x,x')=(\partial_x\Phi)(X_T^{t,x,P_\xi})\partial_\mu X_T^{t,x,P_\xi}(y)-(\partial_x\Phi)(X_T^{t,x',P_{\xi'}})\partial_\mu X_T^{t,x',P_{\xi'}}(y').
\end{equation*}
From Propositions \ref{pro 5.1} and \ref{pro 5.3}, for all $p\geq1$, it holds
\begin{equation}\label{6.12+1}
E[|\Xi(x,x')|^{2p}]\leq C_p(|x-x'|^{2p}+W_2(P_{\xi},P_{\xi'})^{2p}+|y-y'|^{2p}).
\end{equation}
We now give the estimate for $R(r,x,x')$. From (\ref{1002}) and Proposition \ref{pro 4.3} we notice that
\begin{equation*}
\begin{aligned}
{\rm i)}\ & E\big[\Big(\int_t^T|(\partial_z f)(\Pi_r^{t,x,\xi},P_{\Pi_r^{t,\xi}})-(\partial_z f)(\Pi_r^{t,x',\xi'},P_{\Pi_r^{t,\xi'}})|\cdot|Q_r^{t,x',P_{\xi'}}(y')|dr\Big)^{2p}\big]\\
&\leq C_p\big(E[(\int_t^T|Q_r^{t,x',P_{\xi'}}(y')|^2dr)^{2p}]\big)^{\frac{1}{2}}\big(E[(\int_t^T(|\Pi_r^{t,x,\xi}-\Pi_r^{t,x',\xi'}|^2
+W_2(P_{\Pi_r^{t,\xi}},P_{\Pi_r^{t,\xi'}})^2)dr)^{2p}]\big)^{\frac{1}{2}}\\
&\leq C_p\big(|x-x'|^{2p}+W_2(P_{\xi},P_{\xi'})^{2p}\big);\\
{\rm ii)}\ &E\Big[\Big(\int_t^T\big|\widehat{E}[\big((\partial_\mu f)_1(\Pi_r^{t,x,P_{\xi}},P_{\Pi_r^{t,\xi}},\widehat{\Pi}_r^{t,\widehat{\vartheta},P_{\xi}})\\
&\ \ \  \ \ \ \ -(\partial_\mu f)_1(\Pi_r^{t,x',P_{\xi'}},P_{\Pi_r^{t,\xi'}},\widehat{\Pi}_r^{t,\widehat{\vartheta'},P_{\xi'}})\big)
\widehat{Q}_r^{t,\widehat{\vartheta'},P_{\xi'}}(y')]\big|dr\Big)^{2p}\Big]\\
&\leq C_p\big(\widehat{E}[(\int_t^T|\widehat{Q}_r^{t,\widehat{\vartheta'},P_{\xi'}}(y')|^2dr)^{2p}]\big)^{\frac{1}{2}}\cdot
\Big(E\big[\big(\int_t^T(|\Pi_r^{t,x,P_{\xi}}-\Pi_r^{t,x',P_{\xi'}}|^2
+W_2(P_{\Pi_r^{t,\xi}},P_{\Pi_r^{t,\xi'}})^2\\
&\quad+\widehat{E}[|\widehat{\Pi}_r^{t,\widehat{\vartheta},P_{\xi}}-\widehat{\Pi}_r^{t,\widehat{\vartheta'},P_{\xi'}}|^2])dr\big)^{2p}\big]\Big)^{\frac{1}{2}}\\
&\leq C_p(|x-x'|^{2p}+W_2(P_{\xi},P_{\xi'})^{2p}+(\widehat{E}[|\widehat{\vartheta}-\widehat{\vartheta'}|^2])^p);\\
{\rm iii)}\ &E\Big[\Big(\int_t^T\big|\widehat{E}[(\partial_\mu f)_1(\Pi_r^{t,x,\xi},P_{\Pi_r^{t,\xi}},\widehat{\Pi}_r^{t,y,P_{\xi}})\partial_x\widehat{Z}_r^{t,y,P_{\xi}}\\
&\ \ \  \ \ \ \ -(\partial_\mu f)_1(\Pi_r^{t,x',\xi'},P_{\Pi_r^{t,\xi'}},\widehat{\Pi}_r^{t,y',P_{\xi'}})\partial_x\widehat{Z}_r^{t,y',P_{\xi'}}]\big|dr\Big)^{2p}\Big]\\
&\leq C_p(\widehat{E}\int_t^T|\partial_x\widehat{Z}_r^{t,y,P_{\xi}}-\partial_x\widehat{Z}_r^{t,y',P_{\xi'}}|^2dr)^p
+C_p(\widehat{E}[\int_t^T|\partial_x\widehat{Z}_r^{t,y,P_{\xi}}|^2dr])^pE[\left(\int_t^T(|\Pi_r^{t,x,\xi}-\Pi_r^{t,x',\xi'}|^2\right.\\
&\quad+\left.W_2(P_{\Pi_r^{t,\xi}},P_{\Pi_r^{t,\xi'}})^2+\widehat{E}[|\widehat{\Pi}_r^{t,y,P_{\xi}}-\widehat{\Pi}_r^{t,y',P_{\xi'}}|^2])dr\right)^{2p}]^{\frac{1}{2}}\\
&\leq C_p(|x-x'|^{2p}+W_2(P_{\xi},P_{\xi'})^{2p}+|y-y'|^{2p}).\hskip6.5cm \ \ \mbox{ }
\end{aligned}
\end{equation*}
The other terms of $R(r,x,x')$ are estimated in a similar way. Consequently, we get
\begin{equation}\label{6.12+2}
E[(\int_t^T|R(r,x,x')|dr)^{2p}]\leq C_p(|x-x'|^{2p}+W_2(P_{\xi},P_{\xi'})^{2p}+(E[|\vartheta-\vartheta'|^2])^{p}+|y-y'|^{2p}).
\end{equation}
Substituting in BSDE (\ref{equ 4.11.1}) $x=\vartheta$ and $x'=\vartheta'$, since
$
E[|\Xi(\vartheta,\vartheta')|^2]=E[E[|\Xi(x,x')|^2]|_{x=\vartheta,x'=\vartheta'}]\leq C(E[|\vartheta-\vartheta'|^2]+W_2(P_{\xi},P_{\xi'})^2+|y-y'|^2),
$
and
\[
E[(\int_t^T|R(r,\vartheta,\vartheta')|dr)^2]=E[E[\int_t^T|R(r,x,x')|dr]|_{x=\vartheta,x'=\vartheta'}]\leq C(W_2(P_{\xi},P_{\xi'})^2+E[|\vartheta-\vartheta'|^2]+|y-y'|^2),
\]
it follows from Corollary \ref{BSDE estimate} that
\begin{equation}\label{6.12+3}
\begin{split}
&E[\sup_{s\in[t,T]}\big|O_s^{t,\vartheta,P_{\xi}}(y)-O_s^{t,\vartheta',P_{\xi'}}(y)\big|^2+\int_t^T(|Q_r^{t,\vartheta,P_{\xi}}(y)-Q_r^{t,\vartheta',P_{\xi'}}(y')|^2\\
&\ \ \  \ \ \ \ \ +\int_K|R_r^{t,\vartheta,P_{\xi}}(y,e)-R_r^{t,\vartheta',P_{\xi'}}(y',e)|^2\lambda(de))dr]\\
&\leq C(W_2(P_{\xi},P_{\xi'})^2+E[|\vartheta-\vartheta'|^2]+|y-y'|^2).
\end{split}
\end{equation}
This estimate allows to return to BSDE (\ref{equ 4.11.1}). Note that
\begin{equation}\label{6.12+4}
\begin{split}
&E\Big[\Big(\int_t^T\big|\widehat{E}[(\partial_\mu f)_1(\Pi_r^{t,x,\xi},P_{\Pi_r^{t,\xi}},\widehat{\Pi}_r^{t,\widehat{\vartheta},P_{\xi}})\big(\widehat{Q}_r^{t,\widehat{\vartheta},P_{\xi}}(y)
-\widehat{Q}_r^{t,\widehat{\vartheta'},P_{\xi'}}(y')\big)]\big|dr\Big)^{2p}\Big]\\
&\leq C_p\Big(\widehat{E}[\int_t^T|\widehat{Q}_r^{t,\widehat{\vartheta},P_{\xi}}(y)-\widehat{Q}_r^{t,\widehat{\vartheta'},P_{\xi'}}(y')|^2dr]\Big)^p\\
&\leq C_p\big(W_2(P_{\xi},P_{\xi'})^{2p}+(E[|\vartheta-\vartheta'|^2])^p+|y-y'|^{2p}\big),
\end{split}
\end{equation}
and analogously,
\begin{equation}\label{6.12+5}
\begin{split}
&E[\Big(\int_t^T\big|\widehat{E}[(\partial_\mu f)_2(\Pi_r^{t,x,\xi},P_{\Pi_r^{t,\xi}},\widehat{\Pi}_r^{t,\widehat{\vartheta},P_{\xi}})
\big(\int_K|\widehat{R}_r^{t,\widehat{\vartheta},P_{\xi}}(y,e)-\widehat{R}_r^{t,\widehat{\vartheta'},P_{\xi'}}(y',e)|l(e)\lambda(de)\big)]\big|dr\Big)^{2p}]\\
&\leq C_p(W_2(P_{\xi},P_{\xi'})^{2p}+(E[|\vartheta-\vartheta'|^2])^p+|y-y'|^{2p}).
\end{split}
\end{equation}
Consequently, with the help of Theorem \ref{proBSDE estimate}, recalling (\ref{6.12+1}), (\ref{6.12+2}), (\ref{6.12+4}) and (\ref{6.12+5}), we get for the solution of BSDE (\ref{equ 4.11.1} )
\begin{equation}\label{998}
\begin{split}
&E[\sup_{s\in[t,T]}|O_s^{t,x,P_{\xi}}(y)-O_s^{t,x',P_{\xi'}}(y')|^{2p}+(\int_t^T|Q_r^{t,x,P_{\xi}}(y)-Q_r^{t,x',P_{\xi'}}(y')|^2dr)^p\\
&\quad+(\int_t^T\int_K|R_r^{t,x,P_{\xi}}(y,e)-R_r^{t,x',P_{\xi'}}(y',e)|^2\lambda(de)dr)^p]\\
&\leq C_p(|x-x'|^{2p}+W_2(P_{\xi},P_{\xi'})^{2p}+|y-y'|^{2p}+(E[|\vartheta-\vartheta'|^2])^p),
\end{split}
\end{equation}
for all $\vartheta, \vartheta'\in L^2(\mathcal{F}_t)$ with $P_{\vartheta}=P_{\xi}$ and $P_{\vartheta'}=P_{\xi'}$.

Then from the definition of the 2-Wasserstein metric we get
\begin{equation}\label{999}
\begin{split}
&E[\sup_{s\in[t,T]}|O_s^{t,x,P_{\xi}}(y)-O_s^{t,x',P_{\xi'}}(y')|^{2p}+(\int_t^T|Q_r^{t,x,P_{\xi}}(y)-Q_r^{t,x',P_{\xi'}}(y')|^2dr)^p\\
&\quad+(\int_t^T\int_K|R_s^{t,x,P_{\xi}}(y,e)-R_s^{t,x',P_{\xi'}}(y',e)|^2\lambda(de)ds)^p]\\
&\leq C_p(|x-x'|^{2p}+W_2(P_{\xi},P_{\xi'})^{2p}+|y-y'|^{2p}).
\end{split}
\end{equation}
\end{proof}
\begin{lemma}\label{le 6.2}
Suppose (H5.1) and (H6.1) hold true. Then, for $0\leq t\leq s\leq T$ and $x\in\mathbb{R}$,
the lifted processes $
L^2(\mathcal{F}_t)\ni\xi\mapsto Y_s^{t,x,\xi}:=Y_s^{t,x,P_\xi}\in L^2(\mathcal{F}_s)$,
$L^2(\mathcal{F}_t)\ni\xi\mapsto Z_{\cdot}^{t,x,\xi}:=Z_{\cdot}^{t,x,P_\xi}\in \mathcal{H}^2_{\mathbb{F}}(t,T)$,
$L^2(\mathcal{F}_t)\ni\xi\mapsto H_{\cdot}^{t,x,\xi}:=H_{\cdot}^{t,x,P_\xi}\in \mathcal{K}^2_\lambda(t,T)
$
\noindent as functionals of $\xi$ are G\^{a}teaux differentiable, and the G\^{a}teaux derivatives in direction $\eta\in L^2({\cal F}_t)$ are just
$\mathcal{O}^{t,x,\xi}_s(\eta)$, $\mathcal{Q}^{t,x,\xi}_s(\eta)$ and $\mathcal{R}^{t,x,\xi}_s(\eta)$, respectively, i.e.,
$$
\begin{aligned}
&\partial_\xi Y^{t,x,\xi}_s(\eta)=\mathcal{O}^{t,x,\xi}_s(\eta)=\bar{E}[O^{t,x,P_\xi}_s(\bar{\xi})\cdot\bar{\eta}],\ \mbox{P-a.s.}, s\in[t,T],\\
&\partial_\xi Z^{t,x,\xi}_s(\eta)=\mathcal{Q}^{t,x,\xi}_s(\eta)=\bar{E}[Q^{t,x,P_\xi}_s(\bar{\xi})\cdot\bar{\eta}],\ \mbox{dsdP-a.e.},\\
&\partial_\xi H^{t,x,\xi}_s(\eta)=\mathcal{R}^{t,x,\xi}_s(\eta)=\bar{E}[R^{t,x,P_\xi}_s(\bar{\xi})\cdot\bar{\eta}],\ \mbox{dsd}\lambda\mbox{dP-a.e.} ,
\end{aligned}
$$
where $\mathcal{O}^{t,x,\xi}_s(\eta),\ \mathcal{Q}^{t,x,\xi}_s(\eta),\ \mathcal{R}^{t,x,\xi}_s(\eta),\
O^{t,x,P_\xi}_s(y),\ Q^{t,x,P_\xi}_s(y),\ R^{t,x,P_\xi}_s(y)$ are defined in Lemma \ref{le 6.1}.
\end{lemma}
\begin{proof}
The proof is split into two steps.\\
\indent $\mathbf{Step\ 1.}$ We prove that the directional derivatives of $Y^{t,x,\xi},\ Z^{t,x,\xi},\ H^{t,x,\xi}$
in all direction $\eta\in L^2(\mathcal{F}_t)$ exist, and
$$
\begin{aligned}
&\mathcal{O}^{t,x,\xi}_\cdot(\eta)-\frac{1}{h}(Y_\cdot^{t,x,\xi+h\eta}-Y_\cdot^{t,x,\xi})\xrightarrow[h\rightarrow0]{{\cal S}_{\mathbb{F}}^2}0,\quad
\mathcal{Q}^{t,x,\xi}_\cdot(\eta)-\frac{1}{h}(Z_\cdot^{t,x,\xi+h\eta}-Z_\cdot^{t,x,\xi})\xrightarrow[h\rightarrow0]{\mathcal{H}_\mathbb{F}^2}0,\\
&\mathcal{R}^{t,x,\xi}_\cdot(\eta)-\frac{1}{h}(H_\cdot^{t,x,\xi+h\eta}-H_\cdot^{t,x,\xi})\xrightarrow[h\rightarrow0]{\mathcal{K}^2_\lambda}0.
\end{aligned}
$$
In fact, for all $s\in[t,T],$
\begin{equation}\label{m6.24}
\begin{aligned}
&\frac{1}{h}(Y_s^{t,x,\xi+h\eta}-Y_s^{t,x,\xi})-\mathcal{O}^{t,x,\xi}_s(\eta)
 = \frac{1}{h}\big(
\Phi(X_T^{t,x,\xi+h\eta})-\Phi(X_T^{t,x,\xi})\big)-\partial_x\Phi(X_T^{t,x,\xi})\mathcal{U}^{t,x,\xi}_T(\eta)\\
&\ +\frac{1}{h}\int_s^T\big(f(\Pi_r^{t,x,\xi+h\eta},P_{\Pi_r^{t,\xi+h\eta}}\big)
-f(\Pi_r^{t,x,\xi},P_{\Pi_r^{t,\xi}})\big)dr\\
&\ -\bigg\{
\int_s^T\big[
\partial_z f(\Pi_r^{t,x,\xi},P_{\Pi_r^{t,\xi}})\mathcal{Q}_r^{t,x,\xi}(\eta)+\partial_h f(\Pi_r^{t,x,\xi},P_{\Pi_r^{t,\xi}})\int_K\mathcal{R}_r^{t,x,\xi}(\eta,e)l(e)\lambda(de)\big]dr\\
&\ +\int_s^T\widehat{E}\big[(\partial_\mu f)_1(\Pi_r^{t,x,\xi},P_{\Pi_r^{t,\xi}},\widehat{\Pi}_r^{t,\widehat{\xi}})\partial_x\widehat{Z}_r^{t,
\widehat{\xi},P_\xi}\widehat{\eta}+(\partial_\mu f)_1(\Pi_r^{t,x,\xi},P_{\Pi_r^{t,\xi}},\widehat{\Pi}_r^{t,\widehat{\xi}})\widehat{\mathcal{Q}}_r^{t,
\widehat{\xi}}(\widehat{\eta})\big]dr\\
&\ +\int_s^T\widehat{E}\big[(\partial_\mu f)_2(\Pi_r^{t,x,\xi},P_{\Pi_r^{t,\xi}},\widehat{\Pi}_r^{t,\widehat{\xi},P_\xi})
\int_K\partial_x\widehat{H}_r^{t,\widehat{\xi},P_\xi}(e)\widehat{\eta}l(e)\lambda(de)\\
&\ +(\partial_\mu f)_2(\Pi_r^{t,x,\xi},P_{\Pi_r^{t,\xi}},\widehat{\Pi}_r^{t,\widehat{\xi}})\cdot
\int_K\widehat{\mathcal{\mathcal{R}}}_r^{t,\widehat{\xi}}(\widehat{\eta},e)l(e)\lambda(de)
\big]dr\bigg\} -\int_s^T\!\!\!\!\Big(\frac{Z^{t,x,\xi+h\eta}-Z_r^{t,x,\xi}}{h}-\mathcal{Q}_r^{t,x,\xi}(\eta)\Big)dB_r\\
&\ -\int_s^T\int_K\Big(\frac{H_r^{t,x,\xi+h\eta}(e)-H_r^{t,x,\xi}(e)}{h}-\mathcal{R}^{t,x,\xi}_r(\eta,e)\Big)N_\lambda(dr,de)\\
& =\ I_1(x)+I_2(x)-\int_s^T\Big(\frac{Z_r^{t,x,\xi+h\eta}-Z_r^{t,x,\xi}}{h}-\mathcal{Q}_r^{t,x,\xi}(\eta) \Big)dB_r\\
 &\ -\int_s^T\int_K\Big(
\frac{H_r^{t,x,\xi+h\eta}(e)-H_r^{t,x,\xi}(e)}{h}-\mathcal{R}^{t,x,\xi}_r(\eta,e)\Big)N_\lambda(dr,de),\ \hskip2cm \
\end{aligned}
\end{equation}
where
\begin{equation}\nonumber
\begin{aligned}
I_{1}(x):=&\ \frac{1}{h}\big(\Phi(X_T^{t,x,\xi+h\eta})-\Phi(X_T^{t,x,\xi})\big)-
\partial_x\Phi(X_T^{t,x,\xi})\mathcal{U}^{t,x,\xi}_T(\eta)\\
=&\ \int_0^1\partial_x\Phi(X_T^{t,x,\xi}+\rho(X_T^{t,x,\xi+h\eta}-X_T^{t,x,\xi}))d\rho\frac{X_T^{t,x,\xi+h\eta}-X_T^{t,x,\xi}}{h}
-\partial_x\Phi(X_T^{t,x,\xi}){\cal U}_T^{t,x,\xi}(\eta),
\end{aligned}
\end{equation}
and $I_2(x)$ is then defined by (\ref{m6.24}). Notice that
\[
\begin{split}
&I_1(x)=\frac{1}{h}\big(\Phi(X_T^{t,x,\xi+h\eta})-\Phi(X_T^{t,x,\xi})\big)-(\partial_x\Phi)(X_T^{t,x,
\xi})\mathcal{U}_T^{t,x,\xi}(\eta)\\
 &=\int_0^1(\partial_x\Phi)\big(X_T^{t,x,\xi}+\rho(X_T^{t,x,\xi+h\eta}-X_T^{t,x,\xi})\big)d\rho\big(\frac{1}{h}(X_T^{t,x,\xi+h\eta}-X_T^{t,x,\xi})
   \big)-(\partial_x\Phi)(X_T^{t,x,\xi})\cdot\mathcal{U}_T^{t,x,\xi}(\eta)\\
  &=\int_0^1\Big((\partial_x\Phi)\big(X_T^{t,x,\xi}+\rho(X_T^{t,x,\xi+h\eta}-X_T^{t,x,\xi})\big)-(\partial_x\Phi)(X_T^{t,x,\xi})\Big)
   d\rho\big(\frac{1}{h}(X_T^{t,x,\xi+h\eta}-X_T^{t,x,\xi})\big)\\
   &\quad+(\partial_x\Phi)(X_T^{t,x,\xi})\big(\frac{1}{h}(X_T^{t,x,\xi+h\eta}-X_T^{t,x,\xi})-\mathcal{U}_T^{t,x,\xi}(\eta)\big).
\end{split}
\]
Consequently, as $\partial_x\Phi$ is Lipschitz and bounded, and as $I_1(x)$ is independent of $\mathcal{F}_t$,
\[
E[I_1(x)^2|\mathcal{F}_t]\!=\!E[I_1(x)^2]\leq C\frac{1}{h^2}E[|X_T^{t,x,\xi+h\eta}-X_T^{t,x,\xi}|^4]\!+\!CE[|\frac{1}{h}(X_T^{t,x,\xi+h\eta}-X_T^{t,x,\xi})\!-\!\mathcal{U}_T^{t,x,\xi}(\eta)|^2].
\]
From Lemma \ref{le 3.1} we have
\[
E[|X_T^{t,x,\xi+h\eta}-X_T^{t,x,\xi}|^4]\leq CW_2(P_{\xi+h\eta},P_{\xi})^4\leq Ch^4(E[\eta^2])^2.
\]

\noindent On the other hand, from Proposition \ref{pro 5.3}, as
\[
\frac{1}{h}(X_T^{t,x,\xi+h\eta}-X_T^{t,x,\xi})-\mathcal{U}_T^{t,x,\xi}(\eta)
=\widehat{E}\Big[\int_0^1\Big(\partial_\mu X_T^{t,x,P_{\xi+\rho h\eta}}(\widehat{\xi}+\rho h\widehat{\eta})-\partial_\mu X_T^{t,x,P_{\xi}}(\widehat{\xi})\Big)d\rho\cdot\widehat{\eta}\Big],
\]
we have $E[|\frac{1}{h}(X_T^{t,x,\xi+h\eta}-X_T^{t,x,\xi})-\mathcal{U}_T^{t,x,\xi}(\eta)|^2]$
\[
\begin{split}
&\leq E\Big[\Big(\widehat{E}\big[\int_0^1|\partial_\mu X_T^{t,x,P_{\xi+\rho h\eta}}(\widehat{\xi}+\rho h\widehat{\eta})-\partial_\mu X_T^{t,x,P_{\xi}}(\widehat{\xi})|^2d\rho\big]\cdot\widehat{E}[|\widehat{\eta}|^2]\Big)\Big]\\
&\leq E[|{\eta}|^2]\int_0^1\widehat{E}\Big[E\big[|\partial_\mu X_T^{t,x,P_{\xi+\rho h\eta}}(y)-\partial_{\mu}X_T^{t,x,P_{\xi}}(y')|^2\big]\big|_{y=\widehat{\xi}+\rho h\widehat{\eta},\ y'=\widehat{\xi}}\Big]d\rho\\
&\leq CE[|\eta|^2]\int_0^1\widehat{E}\Big[\big(W_2(P_{\xi+\rho h\eta},P_{\xi})^2+|y-y'|^2\big)\big|_{y=\widehat{\xi}+\rho h\widehat{\eta},\ y'=\widehat{\xi}}\Big]d\rho\leq Ch^2\big(E[|\eta|^2]\big)^2.
\end{split}
\]
This shows that
\begin{equation}\label{(*11)}
E[I_1(x)^2|\mathcal{F}_t]=E[I_1(x)^2]\leq Ch^2(E[|\eta|^2])^2.
\end{equation}
\noindent We now consider $I_2(x)=\int_s^TI_2(r)dr$ with $I_2(r)=I_{2,1}(r)-I_{2,2}(r)$, where
\begin{equation}\label{(*21)}
\begin{split}
&I_{2,1}(r)=\frac{1}{h}(f(\Pi_r^{t,x,\xi+h\eta},P_{\Pi_r^{t,\xi+h\eta}})-f(\Pi_r^{t,x,\xi},P_{\Pi_r^{t,\xi}}));\\
&I_{2,2}(r)=(\partial_zf)(\Pi_r^{t,x,\xi},P_{\Pi_r^{t,\xi}}){\cal Q}_r^{t,x,\xi}(\eta)+(\partial_hf)(\Pi_r^{t,x,\xi},P_{\Pi_r^{t,\xi}})
\int_K{\cal R}_r^{t,x,\xi}(\eta,e)l(e)\lambda(de)\\
           &\ \ \ +\widehat{E}[(\partial_\mu f)_1(\Pi_r^{t,x,\xi},P_{\Pi_r^{t,\xi}},\widehat{\Pi}_r^{t,\widehat{\xi}})(\partial_x\widehat{Z}_r^{t,\widehat{\xi},P_{\xi}}
           \cdot\widehat{\eta}+\widehat{\cal Q}_r^{t,\widehat{\xi}}(\widehat{\eta}))]\\
           &\ \ \ +\widehat{E}[(\partial_\mu f)_2(\Pi_r^{t,x,\xi},P_{\Pi_r^{t,\xi}},\widehat{\Pi}_r^{t,\widehat{\xi}})(\int_K\partial_x\widehat{H}_r^{t,\widehat{\xi},P_{\xi}}(e)
           \widehat{\eta}l(e)\lambda(de)+\int_K\widehat{\cal R}_r^{t,\widehat{\xi}}(\widehat{\eta},e)l(e)\lambda(de))].
\end{split}
\end{equation}
We put
\[
\Pi_r^{t,x,\xi}(\eta,\rho):=\Pi_r^{t,x,\xi}+\rho(\Pi_r^{t,x,\xi+h\eta}-\Pi_r^{t,x,\xi}),\ \
\Pi_r^{t,\xi}(\eta,\rho):=\Pi_r^{t,\xi}+\rho(\Pi_r^{t,\xi+h\eta}-\Pi_r^{t,\xi}).
\]
Then, using the fact that $f\in C_b^{1,1}(\mathbb{R}^2\times P_2(\mathbb{R}^2))$, we have
\[
\begin{split}
&I_{2,1}(r)=\frac{1}{h}\big(f(\Pi_r^{t,x,\xi+h\eta},P_{\Pi_r^{t,\xi+h\eta}})-f(\Pi_r^{t,x,\xi},P_{\Pi_r^{t,\xi}})\big)\\
 &=\frac{1}{h}\int_0^1\partial_\rho\big(f(\Pi_r^{t,x,\xi}(\eta,\rho),P_{\Pi_r^{t,\xi}(\eta,\rho)})\big)d\rho\\
          &=\int_0^1\Big\{(\partial_zf)\big(\Pi_r^{t,x,\xi}(\eta,\rho),P_{\Pi_r^{t,\xi}(\eta,\rho)}\big)\big(\frac{1}{h}(Z_r^{t,x,\xi+h\eta}
          -Z_r^{t,x,\xi})\big)\\
          &\quad+(\partial_hf)\big(\Pi_r^{t,x,\xi}(\eta,\rho),P_{\Pi_r^{t,\xi}(\eta,\rho)}\big)\Big(\int_K\frac{1}{h}(H_r^{t,x,\xi+h\eta}(e)-H_r^{t,x,\xi}(e))
          l(e)\lambda(de)\Big)\Big\}d\rho\\
           &\quad+\int_0^1\Big\{\widehat{E}\big[(\partial_\mu f)_1(\Pi_r^{t,x,\xi}(\eta,\rho),P_{\Pi_r^{t,\xi}(\eta,\rho)},\widehat{\Pi}_r^{t,\widehat{\xi}}(\widehat{\eta},\rho))
          (\frac{1}{h}(\widehat{Z}_r^{t,\widehat{\xi}+h\widehat{\eta}}-\widehat{Z}_r^{t,\widehat{\xi}}))\big]\\
          &\quad+\widehat{E}\Big[(\partial_\mu f)_2(\Pi_r^{t,x,\xi}(\eta,\rho),P_{\Pi_r^{t,\xi}(\eta,\rho)},\widehat{\Pi}_r^{t,\widehat{\xi}}(\widehat{\eta},\rho))
          (\int_K\frac{1}{h}\big(\widehat{H}_r^{t,\widehat{\xi}+h\widehat{\eta}}(e)-\widehat{H}_r^{t,\widehat{\xi}}(e)\big)l(e)\lambda(de))\Big]\Big\}d\rho.
\end{split}
\]
From the Lipschitz property of the derivative of $f$ we get
\begin{equation}\label{(*31)}
\begin{split}
&I_{2,1}(r)=(\partial_zf)(\Pi_r^{t,x,\xi},P_{\Pi_r^{t,\xi}})(\frac{1}{h}(Z_r^{t,x,\xi+h\eta}-Z_r^{t,x,\xi}))\\
          &\quad+(\partial_hf)(\Pi_r^{t,x,\xi},P_{\Pi_r^{t,\xi}})(\int_K\frac{1}{h}(H_r^{t,x,\xi+h\eta}(e)-H_r^{t,x,\xi}(e))l(e)\lambda(de))\\
          &\quad+\widehat{E}[(\partial_\mu f)_1(\Pi_r^{t,x,\xi},P_{\Pi_r^{t,\xi}},\widehat{\Pi}_r^{t,\widehat{\xi}})(\frac{1}{h}(\widehat{Z}_r^{t,\widehat{\xi}+h\widehat{\eta}}-\widehat{Z}_r^{t,\widehat{\xi}}))]\\
          &\quad+\widehat{E}[(\partial_\mu f)_2(\Pi_r^{t,x,\xi},P_{\Pi_r^{t,\xi}},\widehat{\Pi}_r^{t,\widehat{\xi}})(\int_K\frac{1}{h}(\widehat{H}_r^{t,\widehat{\xi}
          +h\widehat{\eta}}(e)-\widehat{H}_r^{t,\widehat{\xi}}(e))l(e)\lambda(de))]+R_1(x,h)(r),
\end{split}
\end{equation}
where $R_1(x,h)(r)$ is defined in an obvious way. Also recall that $\displaystyle\Pi_r^{t,x,\xi}=(Z_r^{t,x,\xi},\int_KH_r^{t,x,\xi}(e)l(e)\lambda(de))$, $\displaystyle\Pi_r^{t,\xi}=(Z_r^{t,\xi},\int_KH_r^{t,\xi}(e)l(e)\lambda(de))$. Let us put $\displaystyle R_s^{(1)}(x,h)=\int_s^TR_1(x,h)(r)dr$, and $\displaystyle \parallel R_s^{(1)}(x,h)\parallel:=\int_s^T|R_1(x,h)(r)|dr.$ Then
\[
\begin{split}
& E[\parallel R_t^{(1)}(x,h)\parallel^2|\mathcal{F}_t] \leq CE\Big[\Big(\int_t^T\int_0^1\big(|\Pi_r^{t,x,\xi}(\eta,\rho)-\Pi_r^{t,x,\xi}|+W_2(P_{\Pi_r^{t,\xi}(\eta,\rho)},P_{\Pi_r^{t,\xi}})\big)\\
&\ \ \ \cdot\big(|\frac{1}{h}(Z_r^{t,x,\xi+h\eta}-Z_r^{t,x,\xi})|+|\int_K\frac{1}{h}(H_r^{t,x,\xi+h\eta}(e)-H_r^{t,x,\xi}(e))l(e)\lambda(de)|\big)d\rho dr\Big)^2\big]\\
&+CE\Big[\Big(\int_t^T\int_0^1\widehat{E}\Big[\big(|\Pi_r^{t,x,\xi}(\eta,\rho)-\Pi_r^{t,x,\xi}|+W_2(P_{\Pi_r^{t,\xi}(\eta,\rho)},P_{\Pi_r^{t,\xi}})
+|\widehat{\Pi}_r^{t,\widehat{\xi}}(\widehat{\eta},\rho)-\widehat{\Pi}_r^{t,\widehat{\xi}}|\big)\\
&\ \ \ \big(|\frac{1}{h}(\widehat{Z}_r^{t,\widehat{\xi}
+h\widehat{\eta}}-\widehat{Z}_r^{t,\widehat{\xi}})|+|\int_K\frac{1}{h}(\widehat{H}_r^{t,\widehat{\xi}+h\widehat{\eta}}(e)-\widehat{H}_r^{t,\widehat{\xi}}(e))l(e)\lambda(de)|\big)\Big]d\rho dr\Big)^2\Big] = I_{3,1}+I_{3,2}, \\
\end{split}
\]
where $\displaystyle I_{3,1}:= CE[(\int_t^T\int_0^1(|\Pi_r^{t,x,\xi}(\eta,\rho)-\Pi_r^{t,x,\xi}|+W_2(P_{\Pi_r^{t,\xi}(\eta,\rho)},P_{\Pi_r^{t,\xi}}))
(|\frac{1}{h}(Z_r^{t,x,\xi+h\eta}-Z_r^{t,x,\xi})|$

\noindent$\displaystyle +|\int_K\frac{1}{h}(H_r^{t,x,\xi+h\eta}(e)-H_r^{t,x,\xi}(e))l(e)\lambda(de)|)d\rho dr)^2],$ \
and

$\displaystyle
I_{3,2}:= CE[(\int_t^T\int_0^1\widehat{E}[(|\Pi_r^{t,x,\xi}(\eta,\rho)-\Pi_r^{t,x,\xi}|+W_2(P_{\Pi_r^{t,\xi}(\eta,\rho)},P_{\Pi_r^{t,\xi}})
+|\widehat{\Pi}_r^{t,\widehat{\xi}}(\widehat{\eta},\rho)-\widehat{\Pi}_r^{t,\widehat{\xi}}|)$

\noindent$\displaystyle (|\frac{1}{h}(\widehat{Z}_r^{t,\widehat{\xi}
+h\widehat{\eta}}-\widehat{Z}_r^{t,\widehat{\xi}})|+|\int_K\frac{1}{h}(\widehat{H}_r^{t,\widehat{\xi}+
h\widehat{\eta}}(e)-\widehat{H}_r^{t,\widehat{\xi}}(e))l(e)\lambda(de)|)]d\rho dr)^2].$\\

\noindent Thanks to Proposition \ref{pro 4.3} we get that, here using the notation $|\Pi_r^{t,x',\xi'}-\Pi_r^{t,x,\xi}|:=|Z_r^{t,x',\xi'}-Z_r^{t,x,\xi}|+\int_K|H_r^{t,x',\xi'}(e)-H_r^{t,x,\xi}(e)|l(e)\lambda(de)$ (similar to $|\Pi_r^{t,\xi'}-\Pi_r^{t,\xi}|$),
\[
\begin{aligned}
&E[(\int_t^T|\Pi_r^{t,x',\xi+h\eta}-\Pi_r^{t,x,\xi}|^2dr)^p]\leq\ C_p(|x-x'|^{2p}+W_2(P_{\xi+h\eta},P_{\xi})^{2p})\\
\leq&\ C_p(|x-x'|^{2p}+(|h|^2E[\eta^2])^p),\ p\geq1,
\end{aligned}
\]
i.e.,\hskip1cm\ ${\rm i)}\ \ E[(\int_t^T|\Pi_r^{t,x,\xi+h\eta}-\Pi_r^{t,x,\xi}|^2dr)^2]\leq Ch^4(E[\eta^2])^2;\\$
\[
\begin{aligned}
{\rm ii)}&\ \int_t^TW_2(P_{\Pi_r^{t,\xi}(\eta,\rho)},P_{\Pi_r^{t,\xi}})^2dr\\
&\ \leq E\int_t^T|\Pi_r^{t,\xi+h\eta}-\Pi_r^{t,\xi}|^2dr =E[E[\int_t^T|\Pi_r^{t,x',\xi+h\eta}-\Pi_r^{t,x,\xi}|^2dr]|_{x'=\xi+h\eta,\ x=\xi}]\\
&\ \leq CE[(|x'-x|^2+W_2(P_{\xi+h\eta},P_{\xi})^2)|_{x'=\xi+h\eta,\ x=\xi}]\leq Ch^2E[\eta^2];
\end{aligned}
\]
and, thus,
\[
\begin{split}
I_{3,1}\leq&\frac{C}{h^2}E[(\int_t^T|\Pi_r^{t,x,\xi+h\eta}-\Pi_r^{t,x,\xi}|^2dr)^2]\\
       &+\frac{C}{h^2}E[(\int_t^T\int_0^1W_2(P_{\Pi_r^{t,\xi}(\eta,\rho)},P_{\Pi_r^{t,\xi}})d\rho\cdot|\Pi_r^{t,x,\xi+h\eta}-\Pi_r^{t,x,\xi}|dr)^2]
       \leq Ch^2(E[|\eta|^2])^2.
\end{split}
\]
On the other hand, we have
\[
\begin{split}
I_{3,2}&\leq\frac{C}{h^2}E[(\int_t^T\widehat{E}[(|\Pi_r^{t,x,\xi+h\eta}-\Pi_r^{t,x,\xi}|+h(E[\eta^2])^{\frac{1}{2}}+|\widehat{\Pi}_r^{t,\widehat{\xi}
+h\widehat{\eta}}-\widehat{\Pi}_r^{t,\widehat{\xi}}|)|\widehat{\Pi}_r^{t,\widehat{\xi}+h\widehat{\eta}}-\widehat{\Pi}_r^{t,\widehat{\xi}}|]dr)^2]\\
       &\leq\frac{C}{h^2}\Big(\Big(E[\int_t^T|\Pi_r^{t,x,\xi+h\eta}-\Pi_r^{t,x,\xi}|^2dr]\Big)^2+(h^2E[\eta^2])\widehat{E}[\int_t^T|\widehat{\Pi}_r^{t,\widehat{\xi}
       +h\widehat{\eta}}-\widehat{\Pi}_r^{t,\widehat{\xi}}|^2dr]\\
       &\ \ \ +(\widehat{E}[\int_t^T|\widehat{\Pi}_r^{t,\widehat{\xi}+h\widehat{\eta}}-\widehat{\Pi}_r^{t,\widehat{\xi}}|^2dr])^2\Big)\leq Ch^2(E[\eta^2])^2.
\end{split}
\]
From above we get that \begin{equation}\label{997}E[\parallel R_t^{(1)}(x,h)\parallel^2|\mathcal{F}_t]\leq Ch^2(E[|\eta|^2])^2.\end{equation}

\noindent We remark that from (\ref{equ 6.11})
\begin{equation}\label{996}
\begin{split}
&\int_t^T(\widehat{E}[|\frac{1}{h}(\widehat{Z}_r^{t,\widehat{\xi}+h\widehat{\eta}}-\widehat{Z}_r^{t,\widehat{\xi}})-(\partial_x\widehat{Z}_r^{t,\widehat{\xi},P_{\xi}}\cdot\widehat{\eta}+\frac{1}{h}(Z_r^{t,\widehat{\xi},P_{\xi+h\eta}}-Z_r^{t,\widehat{\xi},P_{\xi}}))|])^2dr\\
&=\int_t^T(\widehat{E}[|\int_0^1(\partial_x\widehat{Z}_r^{t,\widehat{\xi}+h\rho\widehat{\eta},P_{\xi+h\eta}}-\partial_x\widehat{Z}_r^{t,\widehat{\xi},P_{\xi}})d\rho\cdot\widehat{\eta}|])^2dr\\
&\leq\int_0^1\widehat{E}[\int_t^T|\partial_x\widehat{Z}_r^{t,\widehat{\xi}+h\rho\widehat{\eta},P_{\xi+h\eta}}-\partial_x\widehat{Z}_r^{t,\widehat{\xi},P_{\xi}}|^2dr]d\rho\cdot\widehat{E}[|\widehat{\eta}|^2]\\
&\leq CE[|\eta|^2](h^2\widehat{E}[|\widehat{\eta}|^2]+W_2(P_{\xi+h\eta},P_{\xi})^2)\leq C(E[|\eta|^2])^2\cdot h^2,
\end{split}
\end{equation}
and, analogously,
\begin{equation}\label{995}
\begin{split}
&\int_t^T\big(\widehat{E}\big[\big|\int_K\big\{\frac{1}{h}(\widehat{H}_r^{t,\widehat{\xi}+h\widehat{\eta},P_{\xi+h\eta}}(e)-\widehat{H}_r^{t,\widehat{\xi},P_{\xi}}(e))-(\partial_x\widehat{H}_r^{t,\widehat{\xi},P_{\xi}}(e))\cdot\widehat{\eta}\\
&\qquad+\frac{1}{h}(\widehat{H}_r^{t,\widehat{\xi},P_{\xi+h\eta}}(e)-\widehat{H}_r^{t,\widehat{\xi},P_{\xi}}(e))\big\}l(e)\lambda(de)\big|\big]\big)^2dr\leq C(E[|\eta|^2])^2\cdot h^2.
\end{split}
\end{equation}
\noindent Summarizing our above estimates we have from (\ref{m6.24}), (\ref{(*11)}), (\ref{(*21)}) and (\ref{(*31)})
\begin{equation*}
\begin{split}
&\!\!\!\! \ \ \ \ \ \frac{1}{h}(Y_s^{t,x,\xi+h\eta}-Y_s^{t,x,\xi})-{\cal O}_s^{t,x,\xi}(\eta)\quad\quad\quad\\
&\!\!\!\! \ =I_1(x)+I_2(x)-\int_s^T\Big(\frac{1}{h}(Z_r^{t,x,\xi+h\eta}-Z_r^{t,x,\xi})-{\cal Q}_r^{t,x,\xi}(\eta)\Big)dB_r\\
&  \  -\int_s^T\int_K\Big(\frac{1}{h}\big(H_r^{t,x,\xi+h\eta}(e)-H_r^{t,x,\xi}(e)\big)-{\cal R}_r^{t,x,\xi}(\eta,e)\Big)N_{\lambda}(ds,de)\quad\quad\quad\quad\quad\quad\quad\quad\quad\\
&\!\!\!\!  =I_1(x)+\int_s^T\Big\{(\partial_zf)(\Pi_r^{t,x,\xi},P_{\Pi_r^{t,\xi}})\big(\frac{1}{h}(Z_r^{t,x,\xi+h\eta}-Z_r^{t,x,\xi})-{\cal Q}_r^{t,x,\xi}(\eta)\big)\\
&\!\!\!\! \qquad+(\partial_hf)(\Pi_r^{t,x,\xi},P_{\Pi_r^{t,\xi}})\Big(\int_K\big(\frac{1}{h}(H_r^{t,x,\xi+h\eta}(e)-H_r^{t,x,\xi}(e))-{\cal R}_r^{t,x,\xi}(\eta,e)\big)
l(e)\lambda(de)\Big)\\
&\!\!\!\! \qquad+\widehat{E}\big[(\partial_\mu f)_1(\Pi_r^{t,x,\xi},P_{\Pi_r^{t,\xi}},\widehat{\Pi}_r^{t,\widehat{\xi}})\big(\frac{1}{h}(\widehat{Z}_r^{t,\widehat{\xi},P_{\xi+h\eta}}-
\widehat{Z}_r^{t,\widehat{\xi},P_{\xi}})-\widehat{\cal Q}_r^{t,\widehat{\xi}}(\widehat{\eta})\big)\big]\\
\end{split}
\end{equation*}
\begin{equation}\label{BSDE(1)}
\begin{split}
&\!\!\!\! +\widehat{E}\big[(\partial_\mu f)_2(\Pi_r^{t,x,\xi},P_{\Pi_r^{t,\xi}},\widehat{\Pi}_r^{t,\widehat{\xi}})\big(\int_K(\frac{1}{h}(\widehat{H}_r^{t,\widehat{\xi},P_{\xi+h\eta}}(e)
-\widehat{H}_r^{t,\widehat{\xi},P_{\xi}}(e))-\widehat{\cal R}_r^{t,\widehat{\xi}}(\widehat{\eta},e))l(e)\lambda(de)\big)\big]\Big\}dr\\
&\!\!\!\! +\int_s^TR^2_r(x,h)dr-\int_s^T\big(\frac{1}{h}(Z_r^{t,x,\xi+h\eta}-Z_r^{t,x,\xi})-{\cal Q}_r^{t,x,\xi}(\eta)\big)dB_r\\
&\!\!\!\!  -\int_s^T\int_K\big(\frac{1}{h}(H_r^{t,x,\xi+h\eta}(e)-H_r^{t,x,\xi}(e))-{\cal R}_r^{t,x,\xi}(\eta,e)\big)N_\lambda(dr,de),\ s\in[t,T].
\end{split}
\end{equation}
Substituting in (\ref{BSDE(1)}) for $x$ the variable $\xi$ we get
\begin{equation}\label{BSDE(2)}
\begin{split}
&\!\!\!\!\ \ \ \ \frac{1}{h}(Y_s^{t,\xi,P_{\xi+h\eta}}-Y_s^{t,\xi,P_{\xi}})-{\cal O}_s^{t,\xi}(\eta)\\
&\!\!\!\!=I_1(\xi)+\int_s^T\Big\{(\partial_zf)(\Pi_r^{t,\xi},P_{\Pi_r^{t,\xi}})\big(\frac{1}{h}(Z_r^{t,\xi,P_{\xi+h\eta}}-Z_r^{t,\xi,P_{\xi}})-{\cal Q}_r^{t,\xi}(\eta)\big)\\
&\!\!\!\! +(\partial_hf)(\Pi_r^{t,\xi},P_{\Pi_r^{t,\xi}})\Big(\int_K\big(\frac{1}{h}(H_r^{t,\xi,P_{\xi+h\eta}}(e)-H_r^{t,\xi,P_{\xi}}(e))
-{\cal R}_r^{t,\xi}(\eta,e)\big)l(e)\lambda(de)\Big)\\
&\!\!\!\! +\widehat{E}\big[(\partial_\mu f)_1(\Pi_r^{t,\xi},P_{\Pi_r^{t,\xi}},\widehat{\Pi}_r^{t,\widehat{\xi}})\big(\frac{1}{h}(\widehat{Z}_r^{t,\widehat{\xi},P_{\xi+h\eta}}
-\widehat{Z}_r^{t,\widehat{\xi},P_{\xi}})-\widehat{\cal Q}_r^{t,\widehat{\xi}}(\widehat{\eta})\big)\big]\\
&\!\!\!\! +\widehat{E}\Big[(\partial_\mu f)_2(\Pi_r^{t,\xi},P_{\Pi_r^{t,\xi}},\widehat{\Pi}_r^{t,\widehat{\xi}})\Big(\int_K\big(\frac{1}{h}\big(\widehat{H}_r^{t,\widehat{\xi},P_{\xi+h\eta}}(e)
-\widehat{H}_r^{t,\widehat{\xi},P_{\xi}}(e)\big)-\widehat{\cal R}_r^{t,\widehat{\xi}}(\widehat{\eta},e)\big)l(e)\lambda(de)\Big)\Big]\Big\}dr\\
&\!\!\!\!+\int_s^TR^2_r(\xi,h)dr-\int_s^T\big(\frac{1}{h}(Z_r^{t,\xi,P_{\xi+h\eta}}-Z_r^{t,\xi,P_{\xi}})-{\cal Q}_r^{t,\xi}(\eta)\big)dB_r\\
&\!\!\!\!-\int_s^T\int_K\big(\frac{1}{h}(H_r^{t,\xi,P_{\xi+h\eta}}(e)-H_r^{t,\xi,P_{\xi}}(e))-{\cal R}_r^{t,\xi}(\eta,e)\big)N_\lambda(dr,de),\  s\in[t,T].
\end{split}
\end{equation}
Notice that we have $\displaystyle E[(\int_t^T|R^2_r(\xi,h)|dr)^2]=E[E[(\int_t^T|R^2_r(x,h)|dr)^2]|_{x=\xi}]\leq Ch^2(E[|\eta|^2])^2$ from (\ref{997}), (\ref{996}), (\ref{995}); and $E[|I_1(\xi)|^2]\leq Ch^2(E[|\eta|^2])^2$ from (\ref{(*11)}).

Therefore, applying Corollary \ref{BSDE estimate} to BSDE (\ref{BSDE(2)}) we get that
\[
\begin{split}
&E[\sup_{s\in[t,T]}|\frac{1}{h}(Y_s^{t,\xi,P_{\xi+h\eta}}-Y_s^{t,\xi,P_{\xi}})-{\cal O}_s^{t,\xi}(\eta)|^2]
+E[\int_t^T\Big(|\frac{1}{h}(Z_r^{t,\xi,P_{\xi+h\eta}}-Z_r^{t,\xi,P_{\xi}})-{\cal Q}_r^{t,\xi}(\eta)|^2\\
&+\int_K|\frac{1}{h}(H_r^{t,\xi,P_{\xi+h\eta}}(e)-H_r^{t,\xi,P_{\xi}}(e))-{\cal R}_r^{t,\xi}(\eta,e)|^2\lambda(de)\Big)dr]
\leq C(E[|\eta|^2])^2\cdot h^2.
\end{split}
\]
This latter estimate now allows to deduce from (\ref{BSDE(1)}) by using Corollary \ref{BSDE estimate} that
\begin{equation}\label{994}
\begin{split}
&E[\sup_{s\in[t,T]}|\frac{1}{h}(Y_s^{t,x,P_{\xi+h\eta}}-Y_s^{t,x,P_{\xi}})-{\cal O}_s^{t,x,\xi}(\eta)|^2]
+E[\int_t^T(|\frac{1}{h}(Z_s^{t,x,P_{\xi+h\eta}}-Z_s^{t,x,P_{\xi}})-{\cal Q}_s^{t,x,\xi}(\eta)|^2\\
&+\int_K|\frac{1}{h}(H_s^{t,x,P_{\xi+h\eta}}(e)-H_s^{t,x,P_{\xi}}(e))-{\cal R}_s^{t,x,P_{\xi}}(\eta,e)|^2\lambda(de))ds]
\leq C(E[|\eta|^2])^2\cdot h^2.
\end{split}
\end{equation}

$\mathbf{Step\ 2.}$ In Step 1 we have proved that the directional derivatives of $Y^{t,x,\xi},\ Z^{t,x,\xi},\ H^{t,x,\xi}$
in all direction $\eta\in L^2(\mathcal{F}_t)$ exist and  the directional directives
$\partial_\xi Y^{t,x,\xi}_s(\eta),\  \partial_\xi Z^{t,x,\xi}_s(\eta),\   \partial_\xi H^{t,x,\xi}_s(\eta)$ coincide with
$\mathcal{O}_s^{t,x,\xi}(\eta),\ \mathcal{Q}_s^{t,x,\xi}(\eta),\ \mathcal{R}_s^{t,x,\xi}(\eta)$.
Recall that $\mathcal{O}_s^{t,x,\xi}(\cdot),\ \mathcal{Q}_s^{t,x,\xi}(\cdot),\ \mathcal{R}_s^{t,x,\xi}(\cdot)$ are
linear and continuous mappings. Consequently, $Y^{t,x,\xi},\ Z^{t,x,\xi},\ H^{t,x,\xi}$ as functionals of $\xi$
are G\^{a}teaux differentiable, and furthermore, from Lemma \ref{le 6.1} the G\^{a}teaux derivatives can be
characterized by

$$
\begin{aligned}
&\partial_\xi Y^{t,x,\xi}_s(\eta)=\mathcal{O}^{t,x,\xi}_s(\eta)=\bar{E}[O^{t,x,P_\xi}_s(\bar{\xi})\cdot\bar{\eta}],\ \mbox{P-a.s.},\ s\in[t,T],\\
&\partial_\xi Z^{t,x,\xi}_s(\eta)=\mathcal{Q}^{t,x,\xi}_s(\eta)=\bar{E}[Q^{t,x,P_\xi}_s(\bar{\xi})\cdot\bar{\eta}],\ \mbox{dsdP-a.e.},\\
&\partial_\xi H^{t,x,\xi}_s(\eta)=\mathcal{R}^{t,x,\xi}_s(\eta)=\bar{E}[R^{t,x,P_\xi}_s(\bar{\xi})\cdot\bar{\eta}],\ \mbox{dsd}\lambda\mbox{dP-a.e.}
\end{aligned}
$$
The proof is complete.
\end{proof}

In order to prove $Y^{t,x,\xi},\ Z^{t,x,\xi}, H^{t,x,\xi}$ are Fr\'{e}chet
differentiable, we want to show
$$
\begin{aligned}
&L^2(\mathcal{F}_t)\ni\xi\mapsto \mathcal{O}_s^{t,x,\xi}\in L(L^2(\mathcal{F}_t),L^2(\mathcal{F}_s)),\ \
 L^2(\mathcal{F}_t)\ni\xi\mapsto \mathcal{Q}^{t,x,\xi}=(\mathcal{Q}_s^{t,x,\xi})\in L(L^2(\mathcal{F}_t),{\cal H}_{\mathbb{F}}^2(t,T)),\\
&L^2(\mathcal{F}_t)\ni\xi\mapsto \mathcal{R}^{t,x,\xi}=(\mathcal{R}_s^{t,x,\xi})\in L(L^2(\mathcal{F}_t),{\cal K}_\lambda^2(t,T))
\end{aligned}
$$
are continuous.
\begin{lemma}\label{le 6.3}
Under the Assumptions (H5.1) and (H6.1), for all $t\in[0,T], \ x\in \mathbb{R},$ the mappings
$\partial_\xi Y_s^{t,x,\cdot}=\mathcal{O}_s^{t,x,\cdot}\in L(L^2({\cal F}_t), L^2({\cal F}_s)),\
(\partial_\xi Z_s^{t,x,\cdot})=(\mathcal{Q}_s^{t,x,\cdot})\in L(L^2({\cal F}_t), {\cal H}_{\mathbb{F}}^2(t, T)),$ and $
(\partial_\xi H_s^{t,x,\cdot})=(\mathcal{R}_s^{t,x,\cdot})\in L(L^2(\mathcal{F}_t),{\cal K}_\lambda^2(t,T))$
as the functional of $\xi$ are continuous.
\end{lemma}
\begin{proof}
We only prove that $\partial_\xi Y_s^{t,x,\xi}=\mathcal{O}_s^{t,x,\xi},\ s\in[t,T]$, is
continuous with respect to $\xi$. The continuity of $(\partial_\xi Z_s^{t,x,\cdot})=(\mathcal{Q}_s^{t,x,\cdot}),$ and $
(\partial_\xi H_s^{t,x,\cdot})=(\mathcal{R}_s^{t,x,\cdot})$ can be proved with a similar argument. From (\ref{equ 6.23}) we have
$$
\begin{aligned}
&|
\partial_\xi Y_s^{t,x,\xi}-\partial_\xi Y_s^{t,x,\xi'}
|^2_{L(L^2(\mathcal{F}_t),L^2(\mathcal{F}_s))}
=
\mathop{\rm sup}
\limits_{\eta\in L^2(\mathcal{F}_t),|\eta|_{L^2}\leq1}
|
\partial_\xi Y_s^{t,x,\xi}(\eta)-\partial_\xi Y_s^{t,x,\xi'}(\eta)|^2_{L^2}\\
&=
\mathop{\rm sup}
\limits_{\stackrel{\eta\in L^2(\mathcal{F}_t)}{|\eta|_{L^2}\leq1}}
E[
|\partial_\xi Y_s^{t,x,\xi}(\eta)-\partial_\xi Y_s^{t,x,\xi'}(\eta)|^2]
=
\mathop{\rm sup}
\limits_{\stackrel{\eta\in L^2(\mathcal{F}_t)}{|\eta|_{L^2}\leq1}}
E[
|\bar{E}[(O_s^{t,x,P_\xi}(\bar{\xi})-O_s^{t,x,P_{\xi'}}(\bar{\xi}'))\cdot\bar{\eta}]|^2
]\\
&\leq
\mathop{\rm sup}
\limits_{\stackrel{\eta\in L^2(\mathcal{F}_t)}{|\eta|_{L^2}\leq1}}
E[
(\bar{E}[|O_s^{t,x,P_\xi}(\bar{\xi})-O_s^{t,x,P_{\xi'}}(\bar{\xi}')|^2])\cdot
(\bar{E}|\bar{\eta}|^2)
]\\
&\leq
E[
\bar{E}[|O_s^{t,x,P_\xi}(\bar{\xi})-O_s^{t,x,P_{\xi'}}(\bar{\xi}')|^2]
]\leq
\bar{E}[E[
|O_s^{t,x,P_\xi}(y)-O_s^{t,x,P_{\xi'}}(y')|^2]
\big|_{y=\bar{\xi},\ y'=\bar{\xi}'}
]\\
&\leq C
(\bar{E}|\bar{\xi}-\bar{\xi}'|^2
+W_2(P_\xi,P_{\xi'})^2)\leq 2C E|\xi-\xi'|^2,\ t\leq s\leq T.
\end{aligned}
$$
\end{proof}

So far, combining the Lemmas \ref{le 6.1}, \ref{le 6.2} and \ref{le 6.3},
Theorem \ref{th 6.2} has been proved. As shown in Section 5, $(O^{t,x,P_\xi},Q^{t,x,P_\xi},R^{t,x,P_\xi})$
are the derivatives of $(Y^{t,x,P_\xi}, Z^{t,x,P_\xi}, H^{t,x,P_\xi})$
with respect to the  measure  $P_\xi$, i.e.,
$\partial_\mu Y_s^{t,x,P_\xi}(y):= O_s^{t,x,P_\xi}(y),\
\partial_\mu Z_s^{t,x,P_\xi}(y):=  Q_s^{t,x,P_\xi}(y),\ \partial_\mu H_s^{t,x,P_\xi}(y):=  R_s^{t,x,P_\xi}(y),\ s\in[t,T]. $
As a direct result of (\ref{1002}) and Proposition \ref{pro 6.1} we have
\begin{proposition}\label{pro 6.2}
For $p\geq2$, there exists a constant $C_p>0$ only depending on the bounds and Lipschitz constants of the coefficients, such that for all $t\in[0,T],\ x,\ \bar{x}\in \mathbb{R}^d,\ y,\ \bar{y}\in\mathbb{R}^d,$
$P_\xi,\ P_{\bar{\xi}}\in \mathcal{P}_2(\mathbb{R}^d),$
\begin{equation}\label{equ 6.26}
\begin{aligned}
&\mathrm{i)}\ E[\sup\limits_{s\in[t,T]}|\partial_\mu Y_s^{t,x,P_\xi}(y)|^p+(\int_t^T|\partial_\mu Z_s^{t,x,P_\xi}(y)|^2ds)^{p/2}
+(\int_t^T\int_K|\partial_\mu H_s^{t,x,P_\xi}(y)|^2\lambda(de)ds)^{p/2}]\leqslant C_p,\\
&\mathrm{ii)}\ E[\sup\limits_{s\in[t,T]}|\partial_\mu Y_s^{t,x,P_\xi}(y)-\partial_\mu Y_s^{t,\bar{x},P_{\bar{\xi}}}(\bar{y})|^p
+(\int_t^T|\partial_\mu Z_s^{t,x,P_\xi}(y)-\partial_\mu Z_s^{t,\bar{x},P_{\bar{\xi}}}(\bar{y})|^2ds)^{p/2}\\
&\quad+(\int_t^T\int_K|\partial_\mu H_s^{t,x,P_\xi}(y)-\partial_\mu H_s^{t,\bar{x},P_{\bar{\xi}}}(\bar{y})|^2\lambda(de)ds)^{p/2}]
 \leq C_p(|x-\bar{x}|^p+|y-\bar{y}|^p+W_2(P_\xi,P_{\bar{\xi}})^p).
\end{aligned}
\end{equation}
\end{proposition}
\section{{\protect \large {Second order derivatives of $X^{t,x,P_\xi}$}}}

In this section we investigate the second order derivatives of $X^{t,x,P_\xi}$. For
this we first give the following definition.
\begin{definition}\label{*def}
We say that $g\in C_b^{2,1}(\mathbb{R}^d\times{\cal P}_2(\mathbb{R}^d))$, if $g\in C_b^{1,1}(\mathbb{R}^d\times{\cal P}_2(\mathbb{R}^d))$ is such that \\
{\rm i)} For all $\mu\in {\cal P}_2(\mathbb{R}^d) $, $\partial_xg(\cdot,\mu)\in C^{1,1}(\mathbb{R}^d\rightarrow\mathbb{R}^d)$;\\
{\rm ii)} For all $x\in \mathbb{R}^d, \mu\in {\cal P}_2(\mathbb{R}^d) $, $\partial_\mu g(x,\mu,\cdot)\in C^1(\mathbb{R}^d\rightarrow\mathbb{R}^d)$;\\
{\rm iii)} The derivatives $\partial_x^2g:\mathbb{R}^d\times{\cal P}_2(\mathbb{R}^d)\rightarrow\mathbb{R}^{d\times d}$, $\partial_y(\partial_\mu g)$ $:\mathbb{R}^d\times{\cal P}_2(\mathbb{R}^d)\times\mathbb{R}^d\rightarrow\mathbb{R}^{d\times d}$ are bounded and Lipschitz.\\
{\rm(}Recall the notation $g\in C_b^{1,1}(\mathbb{R}^d\times{\cal P}_2(\mathbb{R}^d))$ contains that $\partial_xg$ and $\partial_\mu g$ are bounded and Lipschitz.{\rm)}
\end{definition}

\noindent $\mathbf{Assumption\ (H7.1)}$ $(b,\sigma)\in C_b^{2,1}(\mathbb{R}^d\times{\cal P}_2(\mathbb{R}^d)\rightarrow\mathbb{R}^d\times\mathbb{R}^{d\times d})$, $\beta(\cdot,e)\in C_b^{2,1}(\mathbb{R}^d\times{\cal P}_2(\mathbb{R}^d)\rightarrow\mathbb{R}^d)$ with bounds of the form $C(1\wedge|e|)$ for all its derivatives of first and second order, and with a Lipschitz constant of the form $C(1\wedge |e|)$ for $\partial_x^2\beta (\cdot, e)$ and $\partial_y(\partial_\mu\beta) (\cdot, e)$.

\begin{theorem}\label{pro 7.1}
Under the assumption (H7.1) the first order derivatives $\partial_{x_i}X_s^{t,x,P_\xi}$ and $\partial_\mu X_s^{t,x,P_\xi}(y)$
are in $L^2$-differentiable w.r.t. x and y, respectively, and interpreted as a functional of $\xi\in L^2({\cal F}_t;\mathbb{R}^d)$, and for
$$
M_{s,i,j}^{t,x,P_\xi}(y):=\big(\partial_{x_i x_j}^2X_s^{t,x,P_\xi},\partial_{y_i}(\partial_\mu X_s^{t,x,P_\xi}(y))\big),\ \  1\le i,\ j\le d,
$$
\noindent we have that, for all $p\ge 2,$ there exists a constant $C_p\in {\mathbb R}_+$
such that, for all $t\in [0,T],\ x,\ x',\ y,\ y'\in\mathbb{R}^d,$ and\ $\xi,\ \xi'\in L^2({\cal F}_t;\mathbb{R}^d),\ 1\le i,\ j\le d,$
\begin{equation}\label{equ 7.2}
\begin{array}{lll}
& \mathrm{i)}\ E[\sup_{s\in [t,T]}|M_{s,i,j}^{t,x,P_\xi}(y)|^p]\le C_p;\\
&\mathrm{ii)}\ E[\sup_{s\in [t,T]}|M_{s,i,j}^{t,x,P_\xi}(y)-M_{s,i,j}^{t,x',P_{\xi'}}(y')|^p] \le C_p(|x-x'|^p+|y-y'|^p+W_2(P_\xi,P_{\xi'})^p);\\
&\mathrm{iii)}\ E[\sup_{s\in [t,t+h]}|M_{s,i,j}^{t,x,P_\xi}(y)|^p] \le C_ph,\ 0\leq h\leq T-t.\\
\end{array}
\end{equation}
\end{theorem}

For the proof we can refer to \cite{HL3}; here the situation is even more simple since we don't need to consider the mixed derivatives $\partial_x\partial_\mu, \partial_\mu\partial_x$.

\section{{\protect \large {Second order derivatives of $(Y^{t,x,P_\xi}, Z^{t,x,P_\xi}, H^{t,x,P_\xi})$}}}

This section is devoted to the study of second order derivatives of $(Y^{t,x,P_\xi}, Z^{t,x,P_\xi}, H^{t,x,P_\xi})$.

\noindent $\mathbf{Assumption\ (H8.1)}$ Let $\Phi\in C_b^{2,1}({\mathbb R}^d\times{\cal P}_2({\mathbb R}^d))$,
$f\in C_b^{2,1}({\mathbb R}^d\times{\mathbb R}\times{\mathbb R}^d\times{\mathbb R}\times{\cal P}_2({\mathbb R}^{d+1+d+1}))$ (Recall Definition 7.1).
\begin{theorem}\label{Pro8.1}
Assuming (H7.1) and (H8.1) we have, for all $t\in[0,T],$ $x,y\in\mathbb{R}^d,\xi\in L^2(\mathcal{F}_t;\mathbb{R}^d)$,\\
{\rm i)} The differentiability (in $L^2$) of the mappings\\
$\mathbb{R}^d\ni x\rightarrow(\partial_xY^{t,x,P_\xi},\partial_xZ^{t,x,P_\xi},\partial_xH^{t,x,P_\xi})\in\mathcal{S}_\mathbb{F}^2(t,T;\mathbb{R}^{d})\times\mathcal{H}_\mathbb{F}^2(t,T;\mathbb{R}^{d\times d})\times\mathcal{K}_\lambda^2(t,T;\mathbb{R}^d),$ and\\
$\mathbb{R}^d\ni y\rightarrow(\partial_\mu Y^{t,x,P_\xi}(y),\partial_\mu
Z^{t,x,P_\xi}(y),\partial_\mu H^{t,x,P_\xi}(y))\in\mathcal{S}_\mathbb{F}^2(t,T;\mathbb{R}^d)\times\mathcal{H}_\mathbb{F}^2(t,T;\mathbb{R}^{d\times d})\times \mathcal{K}_\lambda^2(t,T;\mathbb{R}^d)$. \\
{\rm ii)} Moreover, for all $p\geq2$, there is some constant $C_p>0$ only depending on the bounds and the Lipschitz constants of the coefficients $\sigma,b,f,\Phi$ and their first and second order derivatives, such that, for both $(\zeta_s^{t,x,P_\xi}(y),\delta_s^{t,x,P_\xi}(y),\theta_s^{t,x,P_\xi}(y,\cdot))\in\{ (\partial_x^2Y_s^{t,x,P_\xi},\partial_x^2Z_s^{t,x,P_\xi},\partial_x^2H_s^{t,x,P_\xi}(\cdot)),
(\partial_y\partial_\mu Y_s^{t,x,P_\xi}(y),\\ \partial_y\partial_\mu Z_s^{t,x,P_\xi}(y),\partial_y\partial_\mu H_s^{t,x,P_\xi}(y,\cdot))\},
$

$$
\begin{aligned}
{\rm a)}&\ E[\sup_{t\leq s\leq T}|\zeta_s^{t,x,P_\xi}(y)|^p+(\int_t^T|\delta_s^{t,x,P_\xi}(y)|^2ds)^{\frac{p}{2}}+(\int_t^T\int_K|\theta_s^{t,x,P_\xi}(y,e)|^2\lambda(de)
ds)^\frac{p}{2}]\leq C_p;\\
{\rm b)}&\ E\Big[\sup_{t\leq s\leq T}|\zeta_s^{t,x,P_\xi}(y)-\zeta_s^{t,x',P_{\xi'}}(y')|^p+(\int_t^T|\delta_s^{t,x,P_{\xi}}(y)-\delta_s^{t,x',P_{\xi'}}(y')|^2ds)^{\frac{p}{2}}\\
&+(\int_t^T\int_K|\theta_s^{t,x,P_\xi}(e)-\theta_s^{t,x',P_{\xi'}}(e)|^2\lambda(de)ds)^{\frac{p}{2}}\Big]\\
&\leq C_pM^p(|x-x'|^p+|y-y'|^p+W_2(P_\xi,P_{\xi'})^p)+\rho_{M,p}(t,y,P_{\xi}),
\end{aligned}
$$
for all $t\in[0,T]$, $x,x'\in\mathbb{R}^d$, $y,y'\in\mathbb{R}^d$, $\xi,\xi'\in L^2(\mathcal{F};\mathbb{R}^d)$, $M>0$, with $\rho_{M,p}(t,y,P_{\xi})\mathop{\rightarrow}\limits_{M\rightarrow\infty}0$, and $E[\rho_{M,p}(t,\xi,P_{\xi})]\rightarrow0$, as $M\rightarrow\infty$.
\end{theorem}
\begin{proof} Similar to Theorem 6.1, as the $L^2$-derivative of $(\partial_xY^{t,x,P_\xi},\partial_xZ^{t,x,P_\xi},\partial_xH^{t,x,P_\xi})$ with respect to $x$ concerns only $\Pi^{t,x,P_\xi}$ but not the law $P_{\Pi^{t,\xi}}$ of the coefficients of BSDE (\ref{equ 6.1}), the proof is standard, the reader may refer to, for instance, \cite{PP}. Then applying Lemma 10.1 to BSDE satisfied by $(\partial^2_xY^{t,x,P_\xi},\partial^2_xZ^{t,x,P_\xi},\partial^2_xH^{t,x,P_\xi})$ we get directly the estimate a) for $(\partial^2_xY^{t,x,P_\xi},\partial^2_xZ^{t,x,P_\xi},\partial^2_xH^{t,x,P_\xi})$, similar to Proposition 6.1 we get the estimate b) for $(\partial^2_xY^{t,x,P_\xi},\partial^2_xZ^{t,x,P_\xi},\partial^2_xH^{t,x,P_\xi})$ (the reader may also refer to the following proof of the estimate b) for $(\partial_y\partial_\mu Y_s^{t,x,P_\xi}(y), \partial_y\partial_\mu Z_s^{t,x,P_\xi}(y),\partial_y\partial_\mu H_s^{t,x,P_\xi}(y))$).

Now let us prove i) for $(\partial_\mu Y^{t,x,P_\xi}(y),\partial_\mu Z^{t,x,P_\xi}(y),\partial_\mu H^{t,x,P_\xi}(y))$ and the associated estimates in ii). Large parts of the proof are standard or similar to the proofs of Theorem \ref{th 6.2} and Proposition \ref{pro 6.1}. As before in order to point out the main difficulties here but w.l.o.g, let us study the case of dimension $d=1$, with
$ \Phi(X_T^{t,x,P_\xi},P_{X_T^{t,P_\xi}})=\Phi(X_T^{t,x,P_\xi})$, and $f(\Pi_r^{t,x,P_\xi},P_{\Pi_r^{t,\xi}})$ with $ \Pi_r^{t,x,P_\xi}=\int_KH_r^{t,x,P_\xi}(e)l(e)\lambda(de),$ \  and $\Pi_r^{t,\xi}=\Pi_r^{t,\xi,P_\xi}(=\Pi_r^{t,x,P_\xi}|_{x=\xi}).$  From (\ref{equ 6.17}) in our case now $(\partial_\mu Y^{t,x,P_\xi}(y),\partial_\mu Z^{t,x,P_\xi}(y),\partial_\mu H^{t,x,P_\xi}(y))$ is a solution of the following BSDE:
$$
\begin{aligned}
\partial_\mu Y_s^{t,x,P_\xi}(y)&=(\partial_x\Phi)(X_T^{t,x,P_\xi})\partial_\mu X_T^{t,x,P_\xi}(y)+\int_s^T\Big((\partial_hf)(\Pi_r^{t,x,P_\xi},P_{\Pi_r^{t,\xi}})\partial_\mu \Pi_r^{t,x,P_\xi}(y)\\
&+\widehat{E}\big[(\partial_\mu f)(\Pi_r^{t,x,P_\xi},P_{\Pi_r^{t,\xi}},\widehat{\Pi}_r^{t,y,P_\xi})\partial_x\widehat{\Pi}_r^{t,y,P_\xi}+(\partial_\mu f)(\Pi_r^{t,x,P_\xi},P_{\Pi_r^{t,\xi}},\widehat{\Pi}_r^{t,\widehat{\xi}})\partial_\mu\widehat{\Pi}_r^{t,\widehat{\xi},P_\xi}(y)\big]\Big)dr\\
&-\int_s^T\partial_\mu Z_r^{t,x,P_\xi}(y)dB_r-\int_s^T\int_K\partial_\mu H_r^{t,x,P_\xi}(y,e)N_\lambda(dr,de),\ s\in[t,T],
\end{aligned}
$$
for $\xi\in L^2(\mathcal{F}_t)$, $x,\ y\in\mathbb{R}$ and $t\in[0,T]$.

Differentiating formally the above BSDE with respect to $y$, we get the following BSDE:
\begin{equation}\label{222}
\begin{aligned}
&\partial_y\big(\partial_\mu Y_s^{t,x,P_\xi}(y)\big)=(\partial_x\Phi)(X_T^{t,x,P_\xi})\partial_y\big(\partial_\mu X_T^{t,x,P_\xi}(y)\big) \!+\!\int_s^T\!\!(\partial_hf)(\Pi_r^{t,x,P_\xi},P_{\Pi_r^{t,\xi}})\partial_y\big(\partial_\mu\Pi_r^{t,x,P_\xi}(y)\big)dr\\
&+\int_s^T\widehat{E}\big[(\partial_\mu f)(\Pi_r^{t,x,P_\xi},P_{\Pi_r^{t,\xi}},\widehat{\Pi}_r^{t,y,P_\xi})\partial^2_x\widehat{\Pi}_r^{t,y,P_\xi}+\partial_y(\partial_\mu f)(\Pi_r^{t,x,P_\xi},P_{\Pi_r^{t,\xi}},\widehat{\Pi}_r^{t,y,P_\xi})(\partial_x\widehat{\Pi}_r^{t,y,P_\xi})^2\big]dr\\
&+\int_s^T\widehat{E}\big[(\partial_\mu f)(\Pi_r^{t,x,P_\xi},P_{\Pi_r^{t,\xi}},\widehat{\Pi}_r^{t,\widehat{\xi}})\partial_y(\partial_\mu\widehat{\Pi}_r^{t,\widehat{\xi},P_\xi}(y))\big]dr-\int_s^T\partial_y\big(\partial_\mu Z_r^{t,x,P_\xi}(y)\big)dB_r\\
&-\int_s^T\int_K\partial_y\big(\partial_\mu H_r^{t,x,P_\xi}(y,e)\big)N_\lambda(dr,de),\ s\in[t,T].
\end{aligned}
\end{equation}
Notice that all second order derivatives of $\Phi$ and $f$ are bounded and
$\widehat{E}[(\int_t^T|\partial_x \widehat{\Pi}_r^{t,y,P_\xi}|^2dr)^p]\leq C_p\quad$ (see (\ref{equ 6.11})). We first consider the above equation (\ref{222}) with $x$ replaced by $\xi$, then from Theorem 10.1 this equation has a unique solution $(\partial_y(\partial_\mu Y_s^{t,\xi,P_\xi}(y)),\partial_y(\partial_\mu Z_s^{t,\xi,P_\xi}(y)),\partial_y(\partial_\mu H_s^{t,\xi,P_\xi}(y)))\in \mathcal{S}_\mathbb{F}^2(t,T)\times\mathcal{H}_\mathbb{F}^2(t,T)\times\mathcal{K}_\lambda^2(t,T)$, and, furthermore, from Theorem 10.3, for all $p\geq 2$, there is some $C_p>0$ only depending on the bounds and the Lipschitz constants of the coefficients and its derivatives of order 1 and 2 such that
\begin{equation}\label{(*500)}
\begin{aligned}
&E\Big[\mathop{\rm sup}_{s\in[t,T]}|\partial_y(\partial_\mu Y_s^{t,\xi,P_\xi}(y))|^p+\big(\int_t^T|\partial_y(\partial_\mu Z_s^{t,\xi,P_\xi}(y))|^2ds\big)^{\frac{p}{2}}\\
&\quad+\big(\int_t^T\int_K|\partial_y(\partial_\mu H_s^{t,\xi,P_\xi}(y,e))|^2\lambda(de)ds\big)^{\frac{p}{2}}\Big]\leq C_p,\ \mbox{for all}\  t\in[0,T],\ y\in\mathbb{R},\ \xi\in L^2(\mathcal{F}_t).
\end{aligned}
\end{equation}
Then return to the equation (\ref{222}) again from Theorem 10.1 it has a unique solution $(\partial_y(\partial_\mu Y_s^{t,x,P_\xi}(y)),$ $\partial_y(\partial_\mu Z_s^{t,x,P_\xi}(y)),\partial_y(\partial_\mu H_s^{t,x,P_\xi}(y)))\in \mathcal{S}_\mathbb{F}^2(t,T)\times\mathcal{H}_\mathbb{F}^2(t,T)\times\mathcal{K}_\lambda^2(t,T)$, and from Theorem 10.3,
\begin{equation}\label{(*51)}
\begin{aligned}
&E\Big[\mathop{\rm sup}_{s\in[t,T]}|\partial_y(\partial_\mu Y_s^{t,x,P_\xi}(y))|^p+\big(\int_t^T|\partial_y(\partial_\mu Z_s^{t,x,P_\xi}(y))|^2ds\big)^{\frac{p}{2}}\\
&\quad+\big(\int_t^T\int_K|\partial_y(\partial_\mu H_s^{t,x,P_\xi}(y,e))|^2\lambda(de)ds\big)^{\frac{p}{2}}\Big]\leq C_p,\ \mbox{for all}\  t\in[0,T],\ y\in\mathbb{R},\ \xi\in L^2(\mathcal{F}_t).
\end{aligned}
\end{equation}
Let $\xi,\xi',\vartheta,\vartheta'\in L^2(\mathcal{F}_t)$ be such that $P_{\vartheta}=P_{\xi}$, $P_{\vartheta'}=P_{\xi'}$. Notice that $\Pi_s^{t,x,P_{\xi}}$ and  $(O_s^{t,x,P_{\xi}}(y),$ $ Q_s^{t,x,P_{\xi}}(y), R_s^{t,x,P_{\xi}}(y)),\ t\leq s\leq T, $ are independent of ${\cal F}_t$. Hence, from (\ref{222}) we get
the following BSDE:
\begin{equation}\label{223}
\begin{aligned}
&\partial_y\big(\partial_\mu Y_s^{t,x,P_\xi}(y)\big)-\partial_y\big(\partial_\mu Y_s^{t,x',P_{\xi'}}(y')\big)=I_1(x,y,P_\xi)-I_1(x',y',P_{\xi'})+\int_s^TR(r,x,x')dr\\
&+\int_s^T(\partial_hf)(\Pi_r^{t,x,P_\xi},P_{\Pi_r^{t,\xi}})
\big(\partial_y\big(\partial_\mu\Pi_r^{t,x,P_\xi}(y)\big)-\partial_y\big(\partial_\mu\Pi_r^{t,x',P_{\xi'}}(y')\big)\big)dr\\
&+\int_s^T\widehat{E}\big[(\partial_\mu f)(\Pi_r^{t,x,P_\xi},P_{\Pi_r^{t,\xi}},\widehat{\Pi}_r^{t,\widehat{\vartheta},P_\xi})
\big(\partial_y(\partial_\mu\widehat{\Pi}_r^{t,\widehat{\vartheta},P_\xi}(y))-\partial_y(\partial_\mu\widehat{\Pi}_r^{t,\widehat{\vartheta'},P_{\xi'}}(y'))\big)\big]dr\\
&-\int_s^T\big(\partial_y\big(\partial_\mu Z_r^{t,x,P_\xi}(y)\big)-\partial_y\big(\partial_\mu Z_r^{t,x',P_{\xi'}}(y')\big)\big)dB_r\\
&-\int_s^T\int_K\big(\partial_y\big(\partial_\mu H_r^{t,x,P_\xi}(y,e)\big)-\partial_y\big(\partial_\mu H_r^{t,x',P_{\xi'}}(y',e)\big)\big)N_\lambda(dr,de),\\
\end{aligned}
\end{equation}
where $I_1(x,y,P_\xi):= (\partial_x\Phi)(X_T^{t,x,P_\xi})\partial_y\big(\partial_\mu X_T^{t,x,P_\xi}(y)\big);$
\[
\begin{split}
&R(r,x,x'):=I_2(r,x,y,P_\xi)-I_2(r,x',y',P_{\xi'})+I_3(r,x,y,P_\xi)-I_3(r,x',y',P_{\xi'})\\
&\quad\quad\quad\quad\quad\quad+R_1(r,x,y,P_\xi;x',y',P_{\xi'})+R_2(r,x,y,P_\xi;x',y',P_{\xi'});\\
&I_2(r,x,y,P_\xi):= \widehat{E}[(\partial_\mu f)(\Pi_r^{t,x,P_\xi},P_{\Pi_r^{t,\xi}}, \widehat{\Pi}_r^{t,y,P_\xi})\partial_x^2 \widehat{\Pi}_r^{t,y,P_\xi}];\\
&I_3(r,x,y,P_\xi):=\widehat{E}[\partial_y(\partial_\mu f)(\Pi_r^{t,x,P_\xi},P_{\Pi_r^{t,\xi}},\widehat{\Pi}_r^{t,y,P_\xi})(\partial_x\widehat{\Pi}_r^{t,y,P_\xi})^2];\\
&R_1(r,x,y,P_\xi;x',y',P_{\xi'}):=\big((\partial_hf)(\Pi_r^{t,x,P_\xi},P_{\Pi_r^{t,\xi}})-(\partial_hf)(\Pi_r^{t,x',P_{\xi'}},P_{\Pi_r^{t,\xi'}})\big)\partial_y(\partial_\mu \Pi_r^{t,x',P_{\xi'}}(y'));\\
&R_2(r,x,y,P_\xi;x',y',P_{\xi'}):=\widehat{E}[\big((\partial_\mu f)(\Pi_r^{t,x,P_\xi},P_{\Pi_r^{t,\xi}},\widehat{\Pi}_r^{t,\widehat{\vartheta},P_{\xi}})\\
&\quad\quad\quad\quad\quad\quad\quad\quad\quad\quad\quad\quad\quad\quad -(\partial_\mu f)(\Pi_r^{t,x',P_{\xi'}},P_{\Pi_r^{t,\xi'}},\widehat{\Pi}_r^{t,\widehat{\vartheta'},P_{\xi'}})\big)\partial_y(\partial_\mu
\widehat{\Pi}_r^{t,\widehat{\vartheta'},P_{\xi'}}(y'))].
\end{split}
\]

\noindent Observe that

\noindent i) It follows from Lemma \ref{le 3.1} and Theorem \ref{pro 7.1} that
$$E[|I_1(x,y,P_\xi)-I_1(x',y',P_{\xi'})|^p]\leq C_p(|x-x'|^p+|y-y'|^p+W_2(P_\xi,P_{\xi'})^p).$$

\noindent ii) As $\partial_hf$ is bounded and Lipschitz, we have from (\ref{(*51)}) and Proposition \ref{pro 4.3},
\begin{equation*}
\begin{aligned}
&E[\Big(\int_t^T|R_1(r,x,y,P_\xi;x',y',P_{\xi'})|dr\Big)^p]\\
&\leq C_p\Big(E[\Big(\int_t^T\Big(|\Pi_r^{t,x,P_\xi}-\Pi_r^{t,x',P_{\xi'}}|^{2}+W_2(P_{\Pi_r^{t,\xi}}, P_{\Pi_r^{t,{\xi'}}})^{2}\Big)dr\Big)^p]\Big)^{\frac{1}{2}}\\
&\leq C_p\Big(|x-x'|^p+W_2(P_\xi, P_{\xi'})^p\Big).
\end{aligned}
\end{equation*}

\noindent  iii) Similar arguments to ii) from (\ref{(*500)}) and Proposition \ref{pro 4.3}, we see that also
$$E[\Big(\int_t^T|R_2(r,x,y,P_\xi;x',y',P_{\xi'})|dr\Big)^p]\leq C_p\Big(|x-x'|^p+W_2(P_\xi, P_{\xi'})^p+\big(E[|\vartheta-\vartheta'|^2]\big)^{\frac{p}{2}}\Big).$$

\noindent  iv) Similar arguments to ii) we get first:
$$\displaystyle E\Big[\Big(\int_t^T\!\!|I_3(r,x,y,P_\xi)-I_3(r,x',y',P_{\xi'})|dr\Big)^p\Big]\!\leq \! I_{3,1}(y,P_\xi;y',P_{\xi'})+I_{3,2}(x,y,P_\xi;x',y',P_{\xi'}),$$
with $\displaystyle I_{3,1}(y,P_\xi;y',P_{\xi'}):=C_pE\Big[\Big(\int_t^T|(\partial_x\Pi_r^{t,y,P_\xi})^2-(\partial_x\Pi_r^{t,y',P_{\xi'}})^2|dr\Big)^p\Big]$\  and
\begin{equation}
\begin{aligned}
&\!\!\!\!\!\!\!\!\!\!\!\!\!\!\!\!\!\!\!\!\!\!\!\!I_{3,2}(x,y,P_\xi;x',y',P_{\xi'}):=C_pE\Big[\Big(\widehat{E}[\int_t^T|\partial_x\widehat{\Pi}_r^{t,y,P_\xi}|^2\\
&\ \ \ \ \cdot
\min\{C, |\Pi_r^{t,x,P_\xi}-\Pi_r^{t,x',P_{\xi'}}|+W_2(P_{\Pi_r^{t,\xi}}, P_{\Pi_r^{t,\xi'}})+|\widehat{\Pi}_r^{t,y,P_{\xi'}}-\widehat{\Pi}_r^{t,y',P_{\xi'}}|\}dr]\Big)^p\Big]\end{aligned}
\end{equation}
(Recall that $\partial_y(\partial_\mu f)$ is bounded and Lipschitz). Obviously, from (\ref{equ 6.11}),
\begin{equation*}
\begin{aligned}
&\!\!\!\!\!\!\!\!\!\!\!\!\!\!\!\!\!\!\!\!\!\!\!\!I_{3,1}(y,P_\xi;y',P_{\xi'}) \leq C_p \Big(E\Big[\Big(\int_t^T|\partial_x\Pi_r^{t,y,P_\xi}-\partial_x\Pi_r^{t,y',P_{\xi'}}|^2dr\Big)^p\Big]\Big)^{\frac{1}{2}}\\
&\quad\quad\quad\ \leq C_p \Big(|y-y'|^p+W_2(P_\xi,P_{\xi'})^p\Big).
\end{aligned}
\end{equation*}
On the other hand, from Proposition \ref{pro 4.3},
\begin{equation}\label{333}
\begin{aligned}
&I_{3,2}(x,y,P_\xi;x',y',P_{\xi'})\\
&\leq C_p M^p\Big( E\Big[\Big(\int_t^T(|\Pi_r^{t,x,P_\xi}-\Pi_r^{t,x',P_{\xi'}}|^2+W_2(P_{\Pi_r^{t,\xi}},P_{\Pi_r^{t,\xi'}})^2
+|\Pi_r^{t,y,P_\xi}-\Pi_r^{t,y',P_{\xi'}}|^2)dr\Big)^p\Big]\Big)^{\frac{1}{2}}\\
&\ \ \ \ +C_pE\Big[\Big(\int_t^T|\partial_x\Pi_r^{t,y,P_\xi}|^2I_{\{|\partial_x\Pi_r^{t,y,P_\xi}|^2\geq M\}}dr\Big)^p\Big]\\
&\leq C_pM^p\Big(|x-x'|^p+W_2(P_\xi, P_{\xi'})^p+|y-y'|^p\Big)+\rho_{M,p}(t,y,P_\xi),
\end{aligned}
\end{equation}
where $\rho_{M,p}(t,y,P_\xi)\rightarrow 0\ (M\rightarrow \infty)$ and $E[\rho_{M,p}(t,\xi,P_\xi)]\rightarrow 0\ (M\rightarrow \infty)$ thanks to the Dominated Convergence Theorem (indeed, $\displaystyle \sup_{y\in \mathbb{R}}E\Big[\Big(\int_t^T|\partial_x\Pi_r^{t,y,P_\xi}|^2dr\Big)^p\Big]\leq C_p<\infty$).

\noindent v) Remarking that, in analogy to (\ref{(*51)}), from Lemma 10.1 with more classical arguments not involving the derivative with respect to the measure, we can show on one hand $\displaystyle E\Big[\Big(\int_t^T|\partial_x^2\Pi_r^{t,x,P_\xi}|^2dr\Big)^{\frac{p}{2}}\Big]\leq C_p,\ (t,x,P_\xi)\in [0, T]\times \mathbb{R} \times {\cal P}_2(\mathbb{R})$, and on the other hand (meeting for the estimate the same difficulty as in iv)-the difficulty is already inherent to the classical case (see, Pardoux and Peng \cite{PP}) although not developed there)
\begin{equation*}
\begin{aligned}
&E\Big[\Big(\int_t^T|\partial_x^2\Pi_r^{t,x,P_\xi}-\partial_x^2\Pi_r^{t,x',P_{\xi'}}|^2dr\Big)^{\frac{p}{2}}\Big]\\
&\leq C_pM^p\Big(|x-x'|^p+W_2(P_\xi, P_{\xi'})^p\Big)+\rho_{M,p}(t,x,P_\xi),
\end{aligned}
\end{equation*}
where $\rho_{M,p}(t,x,P_\xi)\rightarrow 0$ and $E[\rho_{M,p}(t,\xi,P_\xi)]\rightarrow 0\ (M\rightarrow \infty)$. Therefore,
\begin{equation*}
\begin{aligned}
&E\Big[\Big(\int_t^T|I_2(r,x,y,P_\xi)-I_2(r,x',y',P_{\xi'})|dr\Big)^p\Big]\\
&\leq C_p E\Big[\widehat{E}\Big[\Big(\int_t^T(|\Pi_r^{t,x,P_\xi}-\Pi_r^{t,x',P_{\xi'}}|^2+W_2(P_{\Pi_r^{t,\xi}},P_{\Pi_r^{t,\xi'}})^2
+|\widehat{\Pi}_r^{t,y,P_\xi}-\widehat{\Pi}_r^{t,y',P_{\xi'}}|^2)dr\Big)^\frac{p}{2}\Big]\\
&\ \ \ \ +C_pE\Big[\Big(\int_t^T|\partial_x^2\Pi_r^{t,y,P_\xi}-\partial_x^2\Pi_r^{t,y',P_{\xi'}}|^2dr\Big)^\frac{p}{2}\Big]\\
&\leq C_pM^p\Big(|x-x'|^p+W_2(P_\xi, P_{\xi'})^p+|y-y'|^p\Big)+\rho_{M,p}(t,y,P_\xi),
\end{aligned}
\end{equation*}
where $\rho_{M,p}(t,y,P_\xi)\rightarrow 0$ and $E[\rho_{M,p}(t,\xi,P_\xi)]\rightarrow 0$, as $M\rightarrow \infty$.

\noindent Consequently, we have
$$
E[(\int_t^T\!\!\!|R(r,x,x')|dr)^p]\leq\rho_{M,p}(t,y,P_\xi)+C_pM^p(|x-x'|^p+|y-y'|^p+W_2(P_\xi,P_{\xi'})^p+\big(E[|\vartheta-\vartheta'|^2]\big)^{\frac{p}{2}}).
$$

\noindent Now substituting in BSDE (\ref{223}) $x=\vartheta$ and $x'=\vartheta'$, similar to (\ref{6.12+3}) we get
\begin{equation}\label{new(111-1)}
\begin{aligned}
&E\Big[\sup_{s\in[t,T]}|\partial_y(\partial_\mu Y_s^{t,\vartheta,P_\xi}(y))-\partial_y(\partial_\mu Y_s^{t,\vartheta',P_{\xi'}}(y'))|^2+\int_t^T\!\!|\partial_y(\partial_\mu Z_r^{t,\vartheta,P_\xi}(y))-\partial_y(\partial_\mu Z_r^{t,\vartheta',P_{\xi'}}(y'))|^2dr\\
&+\int_t^T\int_K|\partial_y(\partial_\mu H_r^{t,\vartheta,P_\xi}(y,e))-\partial_y(\partial_\mu H_r^{t,\vartheta',P_{\xi'}}(y',e))|^2\lambda(de)dr\Big]\\
&\leq \rho_{M}(t,y,P_\xi)+CM^2(|y-y'|^2+W_2(P_\xi,P_{\xi'})^2+E[|\vartheta-\vartheta'|^2]).
\end{aligned}
\end{equation}

This estimate allows to study BSDE (\ref{223}), following the same arguments for (\ref{6.12+5}), (\ref{998}) and (\ref{999}), we obtain that, for
all $t\in[0,T]$, $x,x',y,y'\in\mathbb{R}$, $\xi,\xi'\in L^2(\mathcal{F}_t;\mathbb{R})$,
\begin{equation}\label{new(11)}
\begin{aligned}
&E\Big[\sup_{s\in[t,T]}|\partial_y(\partial_\mu Y_s^{t,x,P_\xi}(y))-\partial_y(\partial_\mu Y_s^{t,x',P_{\xi'}}(y'))|^p+(\int_t^T|\partial_y(\partial_\mu Z_r^{t,x,P_\xi}(y))-\partial_y(\partial_\mu Z_r^{t,x',P_{\xi'}}(y'))|^2dr)^{\frac{p}{2}}\\
&+(\int_t^T\int_K|\partial_y(\partial_\mu H_r^{t,x,P_\xi}(y,e))-\partial_y(\partial_\mu H_r^{t,x',P_{\xi'}}(y',e))|^2\lambda(de)dr)^{\frac{p}{2}}\Big]\\
&\leq C_pM^p(|x-x'|^p+|y-y'|^p+W_2(P_\xi,P_{\xi'})^p)+\rho_{M,p}(t,y,P_\xi),
\end{aligned}
\end{equation}
where $\rho_{M,p}(t,y,P_\xi)\rightarrow0$, as $M\rightarrow\infty$, and $\hat{E}[\rho_{M,p}(t,\hat{\xi},P_\xi)]\rightarrow0$, as $M\rightarrow\infty$.\\

In order to show that the formal derivative $(\partial_y(\partial_\mu Y^{t,x,P_\xi}(y)),\partial_y(\partial_\mu Z^{t,x,P_\xi}(y)),\partial_y(\partial_\mu H^{t,x,P_\xi}(y)))$ is really the $L^2$-derivative of $(\partial_\mu Y^{t,x,P_\xi}(y),\partial_\mu Z^{t,x,P_\xi}(y),\partial_\mu H^{t,x,P_\xi}(y))$, we estimate
\begin{equation}\nonumber
\begin{aligned}
&{\rm i)}\ \frac{1}{h}(\partial_\mu Y^{t,x,P_\xi}(y+h)-\partial_\mu Y^{t,x,P_\xi}(y))-\partial_y(\partial_\mu Y^{t,x,P_\xi}(y)),\\
&{\rm ii)}\ \frac{1}{h}(\partial_\mu Z^{t,x,P_\xi}(y+h)-\partial_\mu Z^{t,x,P_\xi}(y))-\partial_y(\partial_\mu Z^{t,x,P_\xi}(y)),\\
&{\rm iii)}\ \frac{1}{h}(\partial_\mu H^{t,x,P_\xi}(y+h,\cdot)-\partial_\mu H^{t,x,P_\xi}(y,\cdot))-\partial_y(\partial_\mu H^{t,x,P_\xi}(y)),
\end{aligned}
\end{equation}
with the same tools as for (\ref{new(11)}). Indeed, for $h\in\mathbb{R}\backslash\{0\}$, one has to study the BSDE satisfied by
$$
\begin{aligned}
\Xi^{t,x,P_\xi}(y)(h):=&\frac{1}{h}\big\{(\partial_\mu Y^{t,x,P_\xi}(y+h),\partial_\mu Z^{t,x,P_\xi}(y+h),\partial_\mu H^{t,x,P_\xi}(y+h))-(\partial_\mu Y^{t,x,P_\xi}(y),\partial_\mu Z^{t,x,P_\xi}(y),\\
&\partial_\mu H^{t,x,P_\xi}(y))\big\}-(\partial_y(\partial_\mu Y^{t,x,P_\xi}(y)),\partial_y(\partial_\mu Z^{t,x,P_\xi}(y)),\partial_y(\partial_\mu H^{t,x,P_\xi}(y))).
\end{aligned}
$$
In analogy to the proof of the estimate (\ref{new(11)}) we meet, for example, the term (see iv))
\begin{equation*}\label{new(1)}
\begin{aligned}
&I:=E\Big[\widehat{E}\Big[\Big(\int_t^T|\partial_x \widehat{\Pi}_r^{t,y,P_\xi}\Big\{\frac{1}{h}((\partial_\mu f)(\Pi_r^{t,x,P_\xi},P_{\Pi_r^{t,\xi}},\widehat{\Pi}_r^{t,y+h,P_\xi})-(\partial_\mu f)(\Pi_r^{t,x,P_\xi},P_{\Pi_r^{t,\xi}},\widehat{\Pi}_r^{t,y,P_\xi}))\\
&\ \  \ \ \ \ \ \ \ \ \ -\partial_y(\partial_\mu f)(\Pi_r^{t,x,P_\xi},P_{\Pi_r^{t,\xi}},\widehat{\Pi}_r^{t,y,P_\xi})\partial_x \widehat{\Pi}_r^{t,y,P_\xi}\Big\}|dr\Big)^2\Big]\Big]\\
&=E\Big[\widehat{E}\Big[\Big(\int_t^T|\partial_x \widehat{\Pi}_r^{t,y,P_\xi}\Big\{\int_0^1\partial_y(\partial_\mu f)(\Pi_r^{t,x,P_\xi},P_{\Pi_r^{t,\xi}},\widehat{\Pi}_r^{t,y+\lambda h,P_\xi})\partial_x \widehat{\Pi}_r^{t,y+\lambda h,P_\xi}d\lambda\\
&\ \  \ \ \ \ \ \ \ \ \ -\partial_y(\partial_\mu f)(\Pi_r^{t,x,P_\xi},P_{\Pi_r^{t,\xi}},\widehat{\Pi}_r^{t,y,P_\xi})\partial_x \widehat{\Pi}_r^{t,y,P_\xi}\Big\}|dr\Big)^2\Big]\Big].
\end{aligned}
\end{equation*}
Using that $\partial_y(\partial_\mu f)$ is bounded and Lipschitz, this yields
\begin{equation*}\label{new(1)}
\begin{aligned}
&I\leq C E\Big[\widehat{E}\Big[\Big(\int_t^T|\partial_x \widehat{\Pi}_r^{t,y,P_\xi}|\int_0^1|\partial_x \widehat{\Pi}_r^{t,y+\lambda h,P_\xi}-\partial_x \widehat{\Pi}_r^{t,y,P_\xi}|d\lambda dr\Big)^2\Big]\Big]\\
&\ \  \ +C E\Big[\widehat{E}\Big[\Big(\int_t^T|\partial_x {\Pi}_r^{t,y,P_\xi}|^2\int_0^1\mbox{min}\{C,|\widehat{\Pi}_r^{t,y+\lambda h,P_\xi}-\widehat{\Pi}_r^{t,y,P_\xi}|\}d\lambda dr\Big)^2\Big]\Big],\\
\end{aligned}
\end{equation*}
and the argument developed to prove (\ref{333}) allows to see that $I\leq C M^2|h|^2+\rho_{M}(t,y,P_\xi)$, with $\rho_{M,p}(t,y,P_\xi)\rightarrow 0$, $E[\rho_{M,p}(t,\xi,P_\xi)]\rightarrow 0$,  as $M\rightarrow \infty$. On the other hand, it is easy to show that
$$E[|\big(\frac{1}{h}(\partial_\mu X_T^{t,x,P_\xi}(y+h)-\partial_\mu X_T^{t,x,P_\xi}(y))-\partial_y(\partial_\mu X_T^{t,x,P_\xi}(y))\big)\partial_x \Phi(X_T^{t,x,P_\xi})|^2]\leq Ch^2.$$

\noindent Furthermore, when applying Theorem \ref{proBSDE estimate} in the Appendix it yields by using a similar discussion for other terms corresponding to i), iii) and iv)
\[
\begin{split}
&E\Big[\sup\limits_{s\in[t,T]}|\frac{1}{h}(\partial_{\mu}Y_s^{t,x,P_\xi}(y+h)-\partial_{\mu}Y_s^{t,x,P_\xi}(y))-\partial_y(\partial_{\mu}Y_s^{t,x,P_\xi}(y))|^2\\
&\quad+\int_t^T|\frac{1}{h}(\partial_{\mu}Z_s^{t,x,P_\xi}(y+h)-\partial_{\mu}Z_s^{t,x,P_\xi}(y))-\partial_y(\partial_{\mu}Z_s^{t,x,P_\xi}(y))|^2ds\\
&\quad+\int_t^T\int_K|\frac{1}{h}(\partial_{\mu}H_s^{t,x,P_\xi}(y+h,e)-\partial_{\mu}H_s^{t,x,P_\xi}(y,e))-\partial_y(\partial_{\mu}H_s^{t,x,P_\xi}(y,e))|^2\lambda (de)ds\Big]\\
&\leq CM^2|h|^2+\rho_{M}(t,y,P_\xi),
\end{split}
\]
with $\rho_{M}(t,y,P_\xi)\rightarrow 0$ as $M\rightarrow \infty$.

It follows the wished $L^2$-differentiability in $y$ of $(\partial_{\mu}Y^{t,x,P_\xi}(y),\partial_{\mu}Z^{t,x,P_\xi}(y),\partial_{\mu}H^{t,x,P_\xi}(y))$. \end{proof}

\section{Related integral-PDEs of mean-field type}

The objective of this section is to study the related  integral-PDEs of mean-field type. We will prove that $V(t,x,P_\xi)$ defined by (\ref{4.6-2}) is the unique classical solution of the following new nonlocal quasi-linear integral PDE of mean-field type:
\begin{equation}\label{equ 1.4}
\begin{aligned}
&\partial_tV(t,x,P_\xi)=
-\bigg\{\sum\limits_{i=1}^d
\partial_{x_i}V(t,x,P_\xi)b_i(x,P_\xi)
+\frac{1}{2}\sum\limits_{i,j,k=1}^d
\partial_{x_ix_j}^2V(t,x,P_\xi)
(\sigma_{i,k}\sigma_{j,k})(x,P_\xi)\\
&+\int_K\!\!\Big(
V(t,x+\beta(x,P_\xi,e),P_\xi)-V(t,x,P_\xi)
-\sum\limits_{i=1}^d
\partial_{x_i}V(t,x,P_\xi)
\beta_i(x,P_\xi,e)\Big)\lambda(de)+f\Big(x,V(t,x,P_\xi),\\
&\sum\limits_{i=1}^d\partial_{x_i}V(t,x,P_\xi)\sigma_i(x,P_\xi),\int_K\!\!(V(t,x+\beta(x,P_\xi,e),P_\xi)-V(t,x,P_\xi))l(e)\lambda(de),P_{(\xi,\psi(t,\xi,P_{\xi}))}\Big)\\
 &+E\bigg[
\sum\limits_{i=1}^d(\partial_\mu V)_i(t,x,P_\xi,\xi)
b_i(\xi,P_\xi)
+\frac{1}{2}\sum\limits_{i,j,k=1}^d\partial_{y_i}(\partial_\mu V)_j(t,x,P_\xi,\xi)
(\sigma_{i,k}\sigma_{j,k})(\xi,P_\xi)\\
&+
\int^1_0\int_K\sum\limits_{i=1}^d
[(\partial_\mu V)_i(t,x,P_\xi,\xi+\rho\beta(\xi,P_{\xi},e))
-
(\partial_\mu V)_i(t,x,P_\xi,\xi)]
\cdot
\beta_i(\xi,P_{\xi},e)\lambda(de)d\rho
\bigg]\bigg\},\\
& \qquad\qquad\qquad \qquad\qquad\qquad \qquad\qquad\qquad \quad\qquad
 (t,x,\xi)\in[0,T]\times\mathbb{R}^d\times L^2(\mathcal{F}_t;\mathbb{R}^d),\\
&\noindent V(T,x,P_\xi)=\Phi(x,P_\xi),\ (x,P_\xi)\in\mathbb{R}^d\times \mathcal{P}_2(\mathbb{R}^d), \\
\end{aligned}\end{equation}
where $\psi(t,x,P_{\xi}):=$

$
\begin{aligned}
(V(t,x,P_{\xi}),\sum\limits_{i=1}^d\partial_{x_i}V(t,x,P_{\xi})\sigma_i(x,P_{\xi}),\int_K(V(t,x+\beta(x,P_{\xi},e),P_{\xi})-V(t,x,P_{\xi}))
l(e)\lambda(de)).
\end{aligned}
$

The following two propositions
study the regularity properties of the value function $V(t,x,P_\xi)$.
\begin{proposition}\label{pro 9.1}
Under the assumptions (H7.1) and (H8.1) the value function $V$ has the following properties:\\
{\rm i)} $V\in C^{\frac{1}{2},2,2}([0,T]\times {\mathbb R}^d\times{ {\cal P}_2}({\mathbb R}^d))$, which means,\\
\mbox{ }\ {\rm{a)}} $V(t,\cdot,\mu)\in C^2({\mathbb R}^d)$, for all $(t,\mu)\in[0,T]\times{ {\cal P}_2}({\mathbb R}^d)$;\\
\mbox{ }\ {\rm{b)}} $V(t,x,\cdot)\in C_b^2({ {\cal P}_2}(\mathbb{R}^d))$, for all $(t,x)\in [0,T]\times {\mathbb R}^d$;\\
\mbox{ }\ {\rm{c)}} The derivatives $\partial_xV, \partial_x^2V$ are continuous on $[0,T]\times {\mathbb R}^d\times{ {\cal P}_2}({\mathbb R}^d)$, $\partial_\mu V,\partial_y(\partial_\mu V)$ are continuous\\
\mbox{ } \ \ \ and bounded on $[0,T]\times {\mathbb R}^d\times{ {\cal P}_2}({\mathbb R}^d)\times {\mathbb R}^d$;\\
\mbox{ }\ {\rm{d)}} $V(\cdot,x,\mu)$ is $\frac{1}{2}$-H\"{o}lder continuous in $t$, uniformly with respect to $(x,\mu)\in {\mathbb R}^d\times{ {\cal P}_2}({\mathbb R}^d)$.\\
{\rm ii)} For all $\varphi\in\{\partial_xV,\partial_x^2V,\partial_\mu V,\partial_y(\partial_\mu V)\}$, there exists a constant $C>0$,
$$|\varphi(t,x,P_\xi,y)-\varphi(t',x,P_\xi,y)|\leq C|t-t'|^{\frac{1}{8}},\ t,\ t'\in[0,T],\ x,\ y\in {\mathbb R}^d,\ \xi\in L^2({\cal F}).$$
\end{proposition}
\begin{proof}
{\rm i)} follows directly from the preceding results-Proposition 4.3, Theorems 6.1 and 6.2, Proposition 6.2, Theorem 8.1 on $Y^{t,x,P_\xi}$ and its derivatives;\\
{\rm ii)} follows from Lemma \ref{newlem2} in the Appendix. The proof is long, we give it in the appendix.
\end{proof}
From (\ref{4.6-3}) and (\ref{(*102)}) in the proof of Lemma \ref{newlem2}, we get the following results.
\begin{corollary}(Representation Formulas) Under the assumptions (H7.1) and (H8.1) we have the following representation formulas:
\begin{equation}\label{991}\begin{array}{lll}
&\!\!\!\!\!\! Y_s^{t,x,P_{\xi}}= V(s,X_s^{t,x,P_{\xi}},P_{X_s^{t,\xi}}), \mbox{\rm{P-a.s.}},\ s\in [t, T];\\
&\!\!\!\!\!\!  Z_s^{t,x,P_{\xi}}=\partial_xV(s,X_s^{t,x,P_{\xi}},P_{X_s^{t,\xi}})\sigma(X_s^{t,x,P_{\xi}},P_{X_s^{t,\xi}}), {\mbox{\rm dsdP-a.e.}};\\
&\!\!\!\!\!\!  H_s^{t,x,P_{\xi}}(e)=V(s,X_{s-}^{t,x,P_{\xi}}\!+\!\beta(X_{s-}^{t,x,P_{\xi}},P_{X_s^{t,\xi}},e),P_{X_s^{t,\xi}})\!-\!V(s,X_{s-}^{t,x,P_{\xi}},P_{X_s^{t,\xi}}), \mbox{\rm{dsd}}\lambda\mbox{\rm{dP-a.e.}}
\end{array}
\end{equation}
\end{corollary}
\begin{remark} From (\ref{4.6-111}) the solution of BSDE (\ref{equ 4.1}) has the following representation formulas:
\begin{equation}\label{9911}\begin{array}{lll}
&\!\!\!\!\!\! Y_s^{t,P_{\xi}}= V(s,X_s^{t,{\xi}},P_{X_s^{t,\xi}}), \mbox{\rm{P-a.s.}},\ s\in [t, T];\\
&\!\!\!\!\!\!  Z_s^{t,P_{\xi}}=\partial_xV(s,X_s^{t,{\xi}},P_{X_s^{t,\xi}})\sigma(X_s^{t,{\xi}},P_{X_s^{t,\xi}}), {\mbox{\rm dsdP-a.e.}};\\
&\!\!\!\!\!\!  H_s^{t,x,P_{\xi}}(e)=V(s,X_{s-}^{t,{\xi}}\!+\!\beta(X_{s-}^{t,{\xi}},P_{X_s^{t,\xi}},e),P_{X_s^{t,\xi}})\!-\!V(s,X_{s-}^{t,{\xi}},P_{X_s^{t,\xi}}), \mbox{\rm{dsd}}\lambda\mbox{\rm{dP-a.e.}}
\end{array}
\end{equation}

\end{remark}

\begin{theorem}\label{newthm1}
The value function $V\in C^{1,2,2}([0,T]\times {\mathbb R}^d\times{ {\cal P}_2}({\mathbb R}^d))$.

\end{theorem}

\begin{proof} From Proposition 9.1 we see that we only need to prove continuous differentiability of $V$ with respect to $t$. For simplicity of notations we still give the proof when $d=1$. Using the notations in Lemma \ref{newlem2} and (\ref{C2-2}) in the Appendix, we have
\begin{equation}\label{990}
V(t,x,P_\xi)-V(t+h,x,P_\xi)=E[\int_t^{t+h}(\theta(t,t+h,s)+\delta(t,t+h,s)+f(\Pi_s^{t,x,P_{\xi}},P_{\Pi_s^{t,\xi}}))ds];
\end{equation}
and
\[
\begin{split}
&E[\theta(t,t+h,s)]\\
&=E[\partial_xV(s,X_s^{t,x,P_{\xi}},P_{X_s^{t,\xi}})b(X_s^{t,x,P_\xi},P_{X_s^{t,\xi}})
+\frac{1}{2}(\partial_x^2V)(s,X_s^{t,x,P_{\xi}},P_{X_s^{t,\xi}})\sigma(X_s^{t,x,P_\xi},P_{X_s^{t,\xi}})^2\\
&+\int_K\Big(V(s,X_s^{t,x,P_{\xi}}+\beta(X_s^{t,x,P_{\xi}},P_{X_s^{t,\xi}},e),P_{X_s^{t,\xi}})-V(s,X_s^{t,x,P_{\xi}},P_{X_s^{t,\xi}})-
\partial_xV(s,X_s^{t,x,P_{\xi}},P_{X_s^{t,\xi}})\\
&\beta(X_s^{t,x,P_{\xi}},P_{X_s^{t,\xi}},e)\Big)\lambda(de)]+R_1(t,t+h)(s),\\
\end{split}
\]
with $|R_1(t,t+h)(s)|\leq Ch^\frac{1}{8} $ which follows from Proposition 9.1. Furthermore, from the Bounded Convergence Theorem we get
\[
\begin{split}
&E[\theta(t,t+h,s)]\rightarrow \partial_xV(t,x,P_{\xi})b(x,P_{\xi})+\frac{1}{2}(\partial_x^2V)(t,x,P_{\xi})\sigma(x,P_{\xi})^2\\
&+\int_K(V(t,x+\beta(x,P_{\xi},e),P_{\xi})-V(t,x,P_{\xi})-(\partial_xV)(t,x,P_{\xi})\beta(x,P_{\xi},e))\lambda(de), \ \mbox{as}\ s\rightarrow t,\ h\downarrow 0,
\end{split}
\]
(bounded by some K only depending on $\partial_xV,\partial_x^2V$). Similarly, we also have
\[
\begin{split}
&E[\delta(t,t+h,s)]\\
=&E[\widehat{E}[(\partial_\mu V)(s,X_s^{t,x,P_{\xi}},P_{X_s^{t,\xi}},\widehat{X}_s^{t,\widehat{\xi}})b(\widehat{X}_s^{t,\widehat{\xi}},P_{X_s^{t,\xi}})
+\frac{1}{2}\partial_y(\partial_\mu V)(s,X_s^{t,x,P_{\xi}},P_{X_s^{t,\xi}},\widehat{X}_s^{t,\widehat{\xi}})\\
&\sigma(\widehat{X}_s^{t,\widehat{\xi}},P_{X_s^{t,\xi}})^2+\int_K\int_0^1((\partial_\mu V)(s,X_s^{t,x,P_{\xi}},P_{X_s^{t,\xi}},\widehat{X}_s^{t,\widehat{\xi}}+\rho\beta(\widehat{X}_s^{t,\widehat{\xi}},
P_{X_s^{t,\xi}},e))\\
&-(\partial_\mu V)(s,X_s^{t,x,P_{\xi}},P_{X_s^{t,\xi}},\widehat{X}_s^{t,\widehat{\xi}}))\beta(\widehat{X}_s^{t,\widehat{\xi}},P_{X_s^{t,\xi}},e)d\rho\lambda(de)]]+R_2(t,t+h)(s),
\end{split}
\]
with $|R_2(t,t+h)(s)|\leq Ch^\frac{1}{8}$ from Proposition 9.1, and from the Bounded Convergence Theorem we have
\[
\begin{split}
&E[\delta(t,t+h,s)]\rightarrow \widehat{E}[(\partial_\mu V)(t,x,P_{\xi},\widehat{\xi})b(\widehat{\xi},P_{\xi})
+\frac{1}{2}\partial_y(\partial_\mu V)(t,x,P_{\xi},\widehat{\xi})\sigma(\widehat{\xi},P_{\xi})^2\\
&+\int_K\int_0^1((\partial_\mu V)(t,x,P_{\xi},\widehat{\xi}+\rho\beta(\widehat{\xi},P_{\xi},e))-(\partial_\mu V)(t,x,P_{\xi},\widehat{\xi}))\beta(\widehat{\xi},P_{\xi},e)d\rho\lambda(de)],
\end{split}
\]
as $s\rightarrow t$, $h\downarrow 0$. Moreover, recall that\\
\smallskip
 $\quad
E[f(\Pi_r^{t,x,P_{\xi}},P_{\Pi_r^{t,\xi}})] $

 $\quad
=E[f(X_r^{t,x,P_{\xi}},Y_r^{t,x,P_{\xi}},Z_r^{t,x,P_{\xi}},\int_KH_r^{t,x,
P_{\xi}}(e)l(e)\lambda(de),P_{(X_r^{t,\xi},Y_r^{t,\xi},Z_r^{t,\xi},\int_KH_r^{t,\xi}(e)\lambda(de))})],$

\noindent and using the representation formulas (\ref{991})
we see that we have also the convergence
\[
\begin{split}
E[f(\Pi_s^{t,x,P_{\xi}},P_{\Pi_s^{t,\xi}})]\rightarrow& f(x,V(t,x,P_{\xi}),\partial_xV(t,x,P_{\xi})\sigma(x,P_{\xi}),\int_K(V(t,x+\beta(x,P_{\xi},e),P_{\xi})\\
&-V(t,x,P_{\xi}))l(e)\lambda(de),P_{(\xi,\psi(t,\xi,P_{\xi}))}),\ s\downarrow t,
\end{split}
\]
where $\psi(t,x,P_{\xi})$

$:=(V(t,x,P_{\xi}),\partial_xV(t,x,P_{\xi})\sigma(x,P_{\xi}),\int_K(V(t,x+\beta(x,P_{\xi},e),P_{\xi})-V(t,x,P_{\xi}))
l(e)\lambda(de)).$

\noindent Consequently, from (\ref{990}) we can obtain that $V(t,x,P_{\xi})$ is differentiable in $t$, and
\begin{equation}\label{C1bis-1}
\begin{split}
&-\partial_tV(t,x,P_{\xi})=(\partial_xV)(t,x,P_{\xi})b(x,P_{\xi})+\frac{1}{2}(\partial_x^2V)(t,x,P_{\xi})\sigma(x,P_{\xi})^2\\
&+\int_K(V(t,x+\beta(x,P_{\xi},e),P_{\xi})-V(t,x,P_{\xi})-(\partial_xV)(t,x,P_{\xi})\beta(x,P_{\xi},e))\lambda(de)\\
&+E\Big[(\partial_\mu V)(t,x,P_{\xi},\xi)b(\xi,P_{\xi})+\frac{1}{2}\partial_y(\partial_\mu V)(t,x,P_\xi,\xi)\sigma(\xi,P_{\xi})^2\\
&+\int_K\int_0^1\big((\partial_\mu V)(t,x,P_{\xi},\xi+\rho\beta(\xi,P_{\xi},e))-(\partial_\mu V)(t,x,P_{\xi},\xi)\big)\beta(\xi,P_{\xi},e)d\rho\lambda(de)\Big]\\
&+f(x,V(t,x,P_{\xi}),\partial_xV(t,x,P_{\xi})\sigma(x,P_{\xi}),\!\int_K\!\!(V(t,x\!+\!\beta(x,P_{\xi},e),P_{\xi})\!-\!V(t,x,P_{\xi}))l(e)\lambda(de),P_{\eta}),
\end{split}
\end{equation}
where $\eta:=(\xi,\psi(t,\xi,P_{\xi}))$. As the whole right-hand side of (\ref{C1bis-1}) is continuous in $(t,x,P_{\xi})$, this proves $V\in C^{1,2,2}([0,T]\times {\mathbb R}^d\times{ {\cal P}_2}({\mathbb R}^d))$.
\end{proof}

Moreover, it also shows the following main result.
\begin{theorem}\label{thm9.2}
$V\in C^{1,2,2}([0,T]\times {\mathbb R}^d\times{ {\cal P}_2}({\mathbb R}^d))$ is the classical solution of PDE (\ref{equ 1.4}), and it is unique in $C^{1,2,2}([0,T]\times {\mathbb R}^d\times{ {\cal P}_2}({\mathbb R}^d))$.
\end{theorem}
\begin{proof}
As before we still assume $d=1$. From (\ref{C1bis-1}) and the definition of $V$ we see immediately that $V\in C^{1,2,2}([0,T]\times {\mathbb R}^d\times{ {\cal P}_2}({\mathbb R}^d))$ is the classical solution of PDE (\ref{equ 1.4}). Now we only need to prove the uniqueness of solution of PDE (\ref{equ 1.4}) in $C^{1,2,2}([0,T]\times {\mathbb R}^d\times{ {\cal P}_2}({\mathbb R}^d))$.

Suppose $U(t,x,P_\xi)\in  C^{1,2,2}([0,T]\times\mathbb{R}\times \mathcal{P}_2(\mathbb{R}))$
is another solution of the integral-PDE of mean-field type ({\ref{equ 1.4}}). Then, applying It\^{o}'s
formula to $U(s,X_s^{t,x,P_\xi},P_{X_s^{t,\xi}})$ (Recall Theorem 2.1, now with $U_s=X_s^{t,x,P_\xi}$ and $X_s=X_s^{t,\xi}$), we have

\begin{equation}
\begin{aligned}
& dU(s,X_s^{t,x,P_\xi},P_{X_s^{t,\xi}})
=\Big\{\partial_sU(s,X_s^{t,x,P_\xi},P_{X_s^{t,\xi}})
+\partial_x U(s,X_{s}^{t,x,P_\xi},P_{X_s^{t,\xi}})b(X_s^{t,x,P_\xi},P_{X_s^{t,\xi}})\\
&+\frac{1}{2}\partial_x^2 U(s,X_s^{t,x,P_\xi},P_{X_s^{t,\xi}})\cdot\sigma(X_s^{t,x,P_\xi},P_{X_s^{t,\xi}})^2 +\int_K \Big(U(s, X_{s}^{t,x,P_\xi}+\beta(X_{s}^{t,x,P_\xi}, P_{X_s^{t,\xi}},e),P_{X_s^{t,\xi}})\\
&\qquad-U(s,X_{s}^{t,x,P_\xi},P_{X_s^{t,\xi}})-\partial_x U(s,X_{s}^{t,x,P_\xi},P_{X_s^{t,\xi}})\beta(X_{s}^{t,x,P_\xi},P_{X_s^{t,\xi}},e)\Big)\lambda(de)\\
&+\widehat{E}\Big[\partial_\mu U(s,X_s^{t,x,P_\xi},P_{X_s^{t,\xi}},\widehat{X_s^{t,\xi}})b(\widehat{X_s^{t,\xi}},P_{X_s^{t,\xi}})+\frac{1}{2}\partial_y(\partial_\mu U)(s,X_s^{t,x,P_\xi},P_{X_s^{t,\xi}},\widehat{X_s^{t,\xi}})\sigma(\widehat{X_s^{t,\xi}},P_{X_s^{t,\xi}})^2\\
&+\int_K\int_0^1\Big(
\partial_\mu U(s, X_{s}^{t,x,P_\xi}, P_{X_s^{t,\xi}}, \widehat{X_s^{t,\xi}}
+\rho\beta(\widehat{X_s^{t,\xi}},P_{X_s^{t,\xi}},e))\\
&\qquad-\partial_\mu U(s, X_{s}^{t,x,P_\xi}, P_{X_s^{t,\xi}}, \widehat{X_s^{t,\xi}})\Big)
\beta(\widehat{X_s^{t,\xi}},P_{X_s^{t,\xi}},e)d\rho\lambda(de)\Big]\Big\}ds\\
&+\partial_xU(s,X_s^{t,x,P_\xi},P_{X_s^{t,\xi}})\sigma(X_s^{t,x,P_\xi},P_{X_s^{t,\xi}})dB_s\\
&+\int_K\Big(U(s, X_{s-}^{t,x,P_\xi}+\beta(X_{s-}^{t,x,P_\xi}, P_{X_s^{t,\xi}},e),P_{X_s^{t,\xi}})-U(s,X_{s-}^{t,x,P_\xi},P_{X_s^{t,\xi}})\Big)N_\lambda(ds,de).\\
\end{aligned}
\end{equation}
But, as $U(t,x,P_\xi)$ satisfies equation ({\ref{equ 1.4}}), this yields
\begin{equation}\label{888}
\begin{aligned}
&dU(s,X_s^{t,x,P_\xi},P_{X_s^{t,\xi}})=-f\Big(X_s^{t,x,P_\xi},U(s, X_{s}^{t,x,P_\xi}, P_{X_s^{t,\xi}}),
\partial_xU(s, X_{s}^{t,x,P_\xi}, P_{X_s^{t,\xi}})\cdot\sigma(X_s^{t,x,P_\xi},P_{X_s^{t,\xi}}),\\
& \int_K\big(U(s,X_{s}^{t,x,P_\xi}+\beta(X_{s}^{t,x,P_\xi},P_{X_s^{t,\xi}},e),P_{X_s^{t,\xi}})
-U(s,X_{s}^{t,x,P_\xi},P_{X_s^{t,\xi}})\big)l(e)\lambda(de), P_\eta\Big)ds\\
&+\partial_x U(s,X_s^{t,x,P_\xi},P_{X_s^{t,\xi}})\sigma(X_s^{t,x,P_\xi},P_{X_s^{t,\xi}})dB_s\\
&+\int_K\Big( U(s, X_{s-}^{t,x,P_\xi}+\beta(X_{s-}^{t,x,P_\xi}, P_{X_s^{t,\xi}},e),P_{X_s^{t,\xi}})-U(s,X_{s-}^{t,x,P_\xi},P_{X_s^{t,\xi}})\Big)
N_\lambda(ds,de),\\
&U(T,X_T^{t,x,P_\xi},P_{X_T^{t,\xi}})=\Phi(X_T^{t,x,P_\xi},P_{X_T^{t,\xi}}),\\
\end{aligned}
\end{equation}
where $\displaystyle \eta:=\Big(X_s^{t,\xi},U(s,X_s^{t,\xi},P_{X_s^{t,\xi}}),
\partial_xU(s,X_s^{t,\xi},P_{X_s^{t,\xi}})\cdot\sigma(X_s^{t,\xi},P_{X_s^{t,\xi}}),
\int_K\big(U(s,X_{s}^{t,\xi}+\beta(X_{s}^{t,\xi},P_{X_s^{t,\xi}},e),$ $P_{X_s^{t,\xi}})
-U(s,X_{s}^{t,\xi},P_{X_s^{t,\xi}})\big)l(e)\lambda(de)\Big).$

Now we replace $x$ by $\xi$ in the above (\ref{888}) (recall that $X_s^{t,\xi}=X_s^{t,x,P_\xi}|_{x=\xi}$, or applying It\^{o}'s
formula directly to $U(s,X_s^{t,\xi},P_{X_s^{t,\xi}})$) we get
\begin{equation}\label{887}
\begin{aligned}
&dU(s,X_s^{t,\xi},P_{X_s^{t,\xi}})=-f\Big(X_s^{t,\xi},U(s, X_{s}^{t,\xi}, P_{X_s^{t,\xi}}),
\partial_xU(s, X_{s}^{t,\xi}, P_{X_s^{t,\xi}})\cdot\sigma(X_s^{t,\xi},P_{X_s^{t,\xi}}),\\
& \int_K\big(U(s,X_{s}^{t,\xi}+\beta(X_{s}^{t,\xi},P_{X_s^{t,\xi}},e),P_{X_s^{t,\xi}})
-U(s,X_{s}^{t,\xi},P_{X_s^{t,\xi}})\big)l(e)\lambda(de), P_\eta\Big)ds\\
&+\partial_x U(s,X_s^{t,\xi},P_{X_s^{t,\xi}})\sigma(X_s^{t,\xi},P_{X_s^{t,\xi}})dB_s\\
&+\int_K\Big( U(s, X_{s-}^{t,\xi}+\beta(X_{s-}^{t,\xi}, P_{X_s^{t,\xi}},e),P_{X_s^{t,\xi}})-U(s,X_{s-}^{t,\xi},P_{X_s^{t,\xi}})\Big)
N_\lambda(ds,de),\\
&U(T,X_T^{t,\xi},P_{X_T^{t,\xi}})=\Phi(X_T^{t,\xi},P_{X_T^{t,\xi}}).\\
\end{aligned}
\end{equation}

\noindent From the uniqueness of the solution of mean-field BSDEs with jumps (\ref{equ 4.1}) we can get that:
$$Y_s^{t,{\xi}}= U(s,X_s^{t,{\xi}},P_{X_s^{t,\xi}}), \mbox{\rm{P-a.s.}},\  s\in [t, T];\ \
 Z_s^{t,{\xi}}=\partial_xU(s,X_s^{t,{\xi}},P_{X_s^{t,\xi}})\sigma(X_s^{t,{\xi}},P_{X_s^{t,\xi}}), {\mbox{\rm dsdP-a.e.}}; $$
\begin{equation}\label{886}\begin{array}{lll}
& H_s^{t,{\xi}}(e)=U(s,X_{s-}^{t,{\xi}}\!+\!\beta(X_{s-}^{t,{\xi}},P_{X_s^{t,\xi}},e),P_{X_s^{t,\xi}})\!-\!U(s,X_{s-}^{t,{\xi}},P_{X_s^{t,\xi}}), \mbox{\rm{dsd}}\lambda\mbox{\rm{dP-a.e.}}
\end{array}
\end{equation}

\noindent Furthermore, with the help of (\ref{888}) and (\ref{886}) it follows from the uniqueness of the solution of BSDEs with jumps (\ref{equ 4.2}) we can conclude that
\begin{equation*}\begin{array}{lll}
&\!\!\!\!\!\! Y_s^{t,x,P_{\xi}}= U(s,X_s^{t,x,P_{\xi}},P_{X_s^{t,\xi}}), \mbox{\rm{P-a.s.}},\ s\in [t, T];\\
&\!\!\!\!\!\!  Z_s^{t,x,P_{\xi}}=\partial_xU(s,X_s^{t,x,P_{\xi}},P_{X_s^{t,\xi}})\sigma(X_s^{t,x,P_{\xi}},P_{X_s^{t,\xi}}), {\mbox{\rm dsdP-a.e.}};\\
&\!\!\!\!\!\!  H_s^{t,x,P_{\xi}}(e)=U(s,X_{s-}^{t,x,P_{\xi}}\!+\!\beta(X_{s-}^{t,x,P_{\xi}},P_{X_s^{t,\xi}},e),P_{X_s^{t,\xi}})\!-\!U(s,X_{s-}^{t,x,P_{\xi}},P_{X_s^{t,\xi}}), \mbox{\rm{dsd}}\lambda\mbox{\rm{dP-a.e.}}
\end{array}
\end{equation*}
In particular, as $s=t$,  $V(t,x,P_\xi)=Y^{t,x,P_\xi}_t=U(t,x,P_\xi).$
The proof is complete.

\end{proof}

\section{Appendix}

\subsection{The proof of Theorem \ref{minThm2.1}}

For simplicity we just consider the case of $d=1$; using the same argument, the results can be easily extended to the case $d>1$.

Now we give the proof of Theorem \ref{minThm2.1}.

\begin{proof}
Let us begin to consider the special case $F(s,x,\mu)=f(\mu)$, $(s,x,\mu)\in[0,T]\times\mathbb{R}\times{\cal P}_2(\mathbb{R})$.\\
{\it Step 1}. We first consider $\displaystyle X_t= X_0+\int_0^t b_sds+\int_0^t\sigma_sdB_s+\int_0^t\int_K\beta_s(e)N_\lambda(ds,de)$, $t\in[0,T]$, $X_0\in L^2({\cal F}_0)$, where $b\in L^\infty_{\mathbb{F}}(0,T)$, $\sigma\in L^\infty_{\mathbb{F}}(0,T)$, and $\beta\in {\cal K}^2_\lambda(0,T)$ are bounded step processes such that, there exists a partition $\pi=\{0=t_0<t_1<\cdots<t_N=T\}$ with:

  i) $\sigma_s=\sigma_{t_k}$, $b_s=b_{t_k}$, $\beta_s(e)=\beta_{t_k}(e)$, $s\in(t_k,t_{k+1}]$, $0\leq k\leq N-1$;

 ii) $|\sigma_{t_k}|,|b_{t_k}|\leq C$, $|\beta_{t_k}(e)|\leq C(1\wedge|e|^2)$, $e\in K$, $0\leq k\leq N-1$.\\
 Recall that $f\in C_b^2({\cal P}_2(\mathbb{R}))$ with $\partial_\mu$, $\partial_y(\partial_\mu f)$ are bounded and continuous (but not necessarily Lipschitz continuous).

 For $0\leq k\leq N-1$, $t_k\leq t<t+h\leq t_{k+1}$, since $f\in C_b^2({\cal P}_2(\mathbb{R}))$ is continuously differentiable with a bounded derivative $\partial_\mu f$, $|\partial_\mu f(\mu,y)|\leq K$, for all $(\mu,y)\in{\cal P}_2(\mathbb{R})\times\mathbb{R}$, we have
 \begin{equation}\label{(1)}f(P_{X_{t+h}})-f(P_{X_t})=\int_0^1E[(\partial_\mu f)(P_{X_t+\rho(X_{t+h}-X_t)},X_t+\rho(X_{t+h}-X_t))(X_{t+h}-X_t)]d\rho.\end{equation}
 Now we define
 \begin{equation}\label{(2)}
 \begin{aligned}
 X_t(\rho,h):=&\ X_t+\rho(X_{t+h}-X_t)\\
 =&\ X_t+\rho(b_{t_k}h+\sigma_{t_k}(B_{t+h}-B_t)+\int_t^{t+h}\int_K\beta_{t_k}(e)N_\lambda(ds,de)).
\end{aligned}
\end{equation}
Then,
\begin{equation}\label{(3)}
\begin{aligned}
&\ E[(\partial_\mu f)(P_{X_t(\rho,h)},X_t(\rho,h))(X_{t+h}-X_t)]\\
=&\ E[(\partial_\mu f)(P_{X_t(\rho,h)},X_t(\rho,h))b_{t_k}h]+E[(\partial_\mu f)(P_{X_t(\rho,h)},X_t(\rho,h))\sigma_{t_k}(B_{t+h}-B_t)]\\
&+ E[(\partial_\mu f)(P_{X_t(\rho,h)},X_t(\rho,h))\int_t^{t+h}\int_K\beta_{t_k}(e)N_\lambda(ds,de)]\\
=& I_1(t,\rho,h)+I_2(t,\rho,h)+I_3(t,\rho,h),
\end{aligned}
\end{equation}
where
$$\begin{array}{lll}
&\displaystyle I_1(t,\rho,h):=E[(\partial_\mu f)(P_{X_t(\rho,h)},X_t(\rho,h))b_{t_k}h];\\
&\displaystyle I_2(t,\rho,h):=E[(\partial_\mu f)(P_{X_t(\rho,h)},X_t(\rho,h))\sigma_{t_k}(B_{t+h}-B_t)];\\
&\displaystyle I_3(t,\rho,h):=E[(\partial_\mu f)(P_{X_t(\rho,h)},X_t(\rho,h))\int_t^{t+h}\int_K\beta_{t_k}(e)N_\lambda(ds,de)].\\
\end{array}
$$
Now we deal with $I_2$ and $I_3$.\\
a) Recall that for $\varphi\in C^1_b(\mathbb{R})$, and $\zeta$ standard normal random variable, we have by partial integration
$$E[\varphi(\zeta)\zeta]=\frac{1}{\sqrt{2\pi}}\int_{-\infty}^{+\infty}\varphi(x)xe^{-\frac{x^2}{2}}dx=
\frac{1}{\sqrt{2\pi}}\int_{-\infty}^{+\infty}\varphi'(x)e^{-\frac{x^2}{2}}dx=E[\varphi'(\zeta)].$$
Hence, as $(B_{t+h}-B_t)$ is independent of $b_{t_k}$, $\sigma_{t_k}$ and $\displaystyle\int_t^{t+h}\int_K\beta_{t_k}(e)N_\lambda (ds,de)$, we have
\begin{equation}\label{(4)}I_2(t,\rho,h)=E[\partial_y(\partial_\mu f)(P_{X_t(\rho,h)},X_t(\rho,h))|\sigma_{t_k}|^2\rho h].\end{equation}
\noindent b) We remark that, as $|\beta_{t_k}(e)|\leq C(1\wedge|e|^2)$ we have
$$E[\int_t^{t+h}\int_K|\beta_{t_k}(e)|N(dsde)]=E[\int_t^{t+h}\int_K|\beta_{t_k}(e)|\lambda(de)ds]\leq C\int_K(1\wedge|e|^2)\lambda(de)h,$$
i.e., we can consider the decomposition
\begin{equation}\label{(5)}\int_t^s\int_K\beta_{t_k}(e)N_\lambda(ds,de)=\int_t^s\int_K\beta_{t_k}(e)N(drde)-(\int_K\beta_{t_k}(e)\lambda(de))(s-t),\ s\in[t,t+h].\end{equation}
We define $\displaystyle \zeta(t,\rho,h)(s):= X_t+\rho(b_{t_k}-\int_K\beta_{t_k}(e)\lambda(de))h+\rho\sigma_{t_k}(B_{t+h}-B_t)+\rho\int_t^s\int_K\beta_{t_k}(e)N(drde)$, $s\in[t,t+h]$. Then we have $\zeta(t,\rho,h)(t+h)=X_t(\rho,h)$, and
\begin{equation}\label{(6)}
\begin{aligned}
&\ (\partial_\mu f)(P_{X_t(\rho,h)},X_t(\rho,h))-(\partial_\mu f)(P_{X_t(\rho,h)},\zeta(t,\rho,h)(t))\\
=&\ \sum_{t<s\leq t+h}\Big((\partial_\mu f)(P_{X_t(\rho,h)},\zeta(t,\rho,h)(s))-(\partial_\mu f)(P_{X_t(\rho,h)},\zeta(t,\rho,h)(s-))\Big)\\
=&\ \int_t^{t+h}\!\!\!\int_K\Big\{(\partial_\mu f)(P_{X_t(\rho,h)},\zeta(t,\rho,h)(s-)+\rho\beta_{t_k}(e))-(\partial_\mu f)(P_{X_t(\rho,h)},\zeta(t,\rho,h)(s-))\Big\}N(dsde)\\
=&\ \int_t^{t+h}\!\!\!\int_K\Big\{ (\partial_\mu f)(P_{X_t(\rho,h)},\zeta(t,\rho,h)(s-)+\rho\beta_{t_k}(e))-(\partial_\mu f)(P_{X_t(\rho,h)},\zeta(t,\rho,h)(s-)) \Big\}N_\lambda(ds,de)\\
&\ \  +R_1(t,\rho,h),
\end{aligned}
\end{equation}
\noindent where
$$R_1(t,\rho,h)=\int_t^{t+h}\!\!\!\int_K\{(\partial_\mu f)(P_{X_t(\rho,h)},\zeta(t,\rho,h)(s-)+\rho\beta_{t_k}(e))-(\partial_\mu f)(P_{X_t(\rho,h)},\zeta(t,\rho,h)(s-))\}\lambda(de)ds,$$
 and as $\partial_y(\partial_\mu f)$ is bounded:
$$|R_1(t,\rho,h)|\leq C\int_t^{t+h}\!\int_K\!\!|\beta_{t_k}(e)|\lambda(de)ds\leq C\int_K(1\wedge|e|^2)\lambda(de)h\leq Ch.$$
Noting that $(\partial_\mu f)(P_{X_t(\rho,h)},\zeta(t,\rho,h)(t))$ and $N(dsde)$, $s\in[t,t+h]$, are independent, we have
\begin{equation}\label{(7)}
\begin{aligned}
&\ I_3(t,\rho,h)\\
=&\ E\Big[\Big((\partial_\mu f)(P_{X_t(\rho,h)},X_t(\rho,h))-(\partial_\mu f)(P_{X_t(\rho,h)},\zeta(t,\rho,h)(t))\Big)\int_t^{t+h}\int_K\beta_{t_k}(e)N_\lambda(ds,de)\Big]\\
=&\ E\Big[\Big(\int_t^{t+h}\int_K \Big\{(\partial_\mu f)(P_{X_t(\rho,h)},\zeta(t,\rho,h)(s-)+\rho\beta_{t_k}(e))-(\partial_\mu f)(P_{X_t(\rho,h)},\zeta(t,\rho,h)(s-))\Big\}\\
&\  N_\lambda(ds,de) \Big) \Big(\int_t^{t+h}\int_K\beta_{t_k}(e)N_\lambda(ds,de)\Big)\Big]
+ E\Big[R_1(t,\rho,h)\Big(\int_t^{t+h}\int_K\beta_{t_k}(e)N_\lambda(ds,de)\Big)\Big]\\
=&\ E\Big[\Big(\int_t^{t+h}\int_K\Big\{(\partial_\mu f)(P_{X_t(\rho,h)},\zeta(t,\rho,h)(s)+\rho\beta_{t_k}(e))-(\partial_\mu f)(P_{X_t(\rho,h)},\zeta(t,\rho,h)(s))   \Big\}\\
&\ \beta_{t_k}(e)\lambda(de)ds\Big)\Big]+R_2(t,\rho,h),
\end{aligned}
\end{equation}
where $\displaystyle R_2(t,\rho,h):=E\Big[R_1(t,\rho,h)\Big(\int_t^{t+h}\int_K\beta_{t_k}(e)N_\lambda(ds,de)\Big)\Big]$.

\noindent Notice that
$$\begin{array}{lll}
|R_2(t,\rho,h)|&\displaystyle\leq E[|R_1(t,\rho,h)\int_t^{t+h}\int_K\beta_{t_k}(e)N_\lambda(ds,de)|]\\
&\displaystyle\leq Ch(E[\int_t^{t+h}\int_K|\beta_{t_k}(e)|^2\lambda(de)ds])^{\frac{1}{2}}\leq Ch^{\frac{3}{2}}.
\end{array}$$

\noindent Consequently, from (\ref{(1)}), (\ref{(3)}), (\ref{(4)}), (\ref{(6)}) and (\ref{(7)}) we get
\begin{equation}\label{(8)}
\begin{split}
&\displaystyle f(P_{X_{t+h}})-f(P_{X_t})\\
&=\int_t^{t+h}\int_0^1E\Big[(\partial_\mu f)\big(P_{X_t(\rho,h)},X_t(\rho,h)\big)b_s+\rho\partial_y(\partial_\mu f)\big(P_{X_t(\rho,h)},X_t(\rho,h)\big)|\sigma_s|^2\\
&\displaystyle+\int_K\big\{(\partial_\mu f)\big(P_{X_t(\rho,h)},\zeta(t,\rho,h)(s)+\rho\beta_s(e)\big)-(\partial_\mu f)\big(P_{X_t(\rho,h)},\zeta(t,\rho,h)(s)\big)\big\}\beta_s(e)\lambda(de)\Big]d\rho ds\\
&\displaystyle+\int_0^1R_2(t,\rho,h)d\rho,
\end{split}
\end{equation}
for $t_k\leq t< t+h\leq t_{k+1}$\ $(0\leq k\leq N-1)$. Let now $n\geq 1$, $t_i^n:= t+\frac{ih}{n}$, $0\leq i\leq n$, then\\
\begin{equation}\label{equli 1}
\begin{split}
f(P_{X_{t+h}})-f(P_{X_t})&=\sum_{i=0}^{n-1}\big(f(P_{X_{t_{i+1}^n}})-f(P_{X_{t_i^n}})\big)\\
&=\int_t^{t+h}\int_0^1E\Big[\sum_{i=0}^{n-1}
I_{[t_i^n,t^n_{i+1}]}(s)\Xi_i^{(n)}(s,\rho)\Big]d\rho ds
+R_3^{(n)}(t,t+h),
\end{split}
\end{equation}
with $\displaystyle R_3^{(n)}(t,t+h)=\sum_{i=0}^{n-1}\int_0^1R_2(t_i^n,\rho,\frac{h}{n})d\rho$, $|R_3^{(n)}(t,t+h)|\leq C {h^{\frac{3}{2}}}\cdot{n^{-\frac{1}{2}}}\rightarrow 0$, as $n\rightarrow\infty$, and
\begin{equation}\label{(9)}
\begin{split}
&\Xi_i^{(n)}(s,\rho)=(\partial_\mu f)\big(P_{X_{t_i^n}(\rho,\frac{h}{n})},X_{t_i^n}(\rho,\frac{h}{n})\big)b_s+\rho\big(\partial_y(\partial_\mu f)\big)\big(P_{X_{t_i^n}(\rho,\frac{h}{n})},X_{t_i^n}(\rho,\frac{h}{n})\big)|\sigma_s|^2\\
&+\int_K\Big\{(\partial_\mu f)\big(P_{X_{t_i^n}(\rho,\frac{h}{n})},\zeta(t_i^n,\rho,\frac{h}{n})(s)+\rho\beta_s(e)\big)-(\partial_\mu f)\big(P_{X_{t_i^n}(\rho,\frac{h}{n})},\zeta(t_i^n,\rho,\frac{h}{n})(s)\big)\Big\}\beta_s(e)\lambda(de),
\end{split}
\end{equation}
where $s\in[t_i^n,t^n_{i+1}],$ $0\leq i\leq n-1.$ As $\partial_\mu f$ and $\partial_y(\partial_\mu f)$ are bounded,
$$|\Xi_i^{(n)}(s,\rho)|\leq C\big(|b_s|+|\sigma_s|^2+\int_K|\beta_s(e)|^2\lambda(de)\big)\leq C,\ s\in[t_i^n,t^n_{i+1}],\ 0\leq i\leq n-1,\ \rho\in[0,1],$$
and, as
\begin{equation}\label{(10)} E\Big[\sup_{s\in[t_i^n,t_{i+1}^n]}|X_{t_i^n}(\rho,\frac{h}{n})-X_s|^2\Big]\leq CE\Big[\sup_{s\in[t_i^n,t_{i+1}^n]}|X_s-X_{t_i^n}|^2\Big]\leq C\frac{h}{n},\end{equation}
we also have
\begin{equation}\label{(11)}\begin{array}{lll}
&{\rm (i)}\ \sup_{s\in[t_i^n,t_{i+1}^n]}W_2\big(P_{X_{t_i^n}(\rho,\frac{h}{n})},P_{X_s}\big)\leq C(\frac{h}{n})^{\frac{1}{2}},\\
&{\rm (ii)}\ E\Big[\sup_{s\in[t_i^n,t_{i+1}^n]}|\zeta(t_i^n,\rho,\frac{h}{n})(s)-X_s|^2\Big]\leq C(\frac{h}{n}).\end{array}\end{equation}
It follows from the continuity of $\partial_\mu f$ and $\partial_y(\partial_\mu f)$ on $\mathcal{P}_2(\mathbb{R})\times\mathbb{R}$ and the Dominated Convergence Theorem that taking limit in (\ref{equli 1}) as $n\rightarrow\infty$ it yields
\begin{equation}\label{(12)}
\begin{split}
&f(P_{X_{t+h}})-f(P_{X_t})\\
=&\int_t^{t+h}\int_0^1E\Big[(\partial_\mu f)(P_{X_s},X_s)b_s+\rho\partial_y(\partial_\mu f)(P_{X_s},X_s)|\sigma_s|^2\\
&+\int_K\big\{(\partial_\mu f)\big(P_{X_s},X_s+\rho\beta_s(e)\big)-(\partial_\mu f)(P_{X_s},X_s)\big\}\beta_s(e)\lambda(de)\Big]d\rho ds\\
=&\int_t^{t+h}E\Big[(\partial_\mu f)(P_{X_s},X_s)b_s+\frac{1}{2}\partial_y(\partial_\mu f)(P_{X_s},X_s)|\sigma_s|^2\\
&+\int_K\int_0^1\big\{(\partial_\mu f)\big(P_{X_s},X_s+\rho\beta_s(e)\big)-(\partial_\mu f)(P_{X_s},X_s)\big\}\beta_s(e)d\rho\lambda(de)\Big]ds.
\end{split}
\end{equation}
As this holds for all $t_k\leq t<t+h\leq t_{k+1}$, $0\leq k\leq N$, it holds for all $0\leq t<t+h\leq T$.\\
{\it Step 2}. Let now $b,\sigma\in {\cal H}^2_\mathbb{F}(0,T)$, $\beta\in {\cal K}^2_\lambda(0,T)$, $X_0\in L^2(\mathcal{F}_0)$. Then we can approximate $b,\sigma,\beta$ by processes $b^n,\sigma^n,\beta^n$ which satisfy, for each $n$, the assumptions made in {\it Step 1} (with bounds depending on $n$ and a partition depending on $n$): $b^n\rightarrow b$, $\sigma^n\rightarrow\sigma$ in ${\cal H}^2_{\mathbb{F}}(0,T)$, $\beta^n\rightarrow\beta$ in ${\cal K}_\lambda^2(0,T)$.\\
Let
\begin{equation}\label{(13)}X_t=X_0+\int_0^tb_sds+\int_0^t\sigma_sdB_s+\int_0^t\int_K\beta_s(e)N_\lambda(ds,de),\ t\in[0,T],\end{equation}
\begin{equation}\label{(14)}X_t^n=X_0+\int_0^tb_s^nds+\int_0^t\sigma_s^ndB_s+\int_0^t\int_K\beta_s^n(e)N_\lambda(ds,de),\ t\in[0,T],\ n\geq 1.\end{equation}
Then,\ \ $ E\Big[\sup_{t\in[0,T]}|X_t-X_t^n|^2\Big]$\\
\begin{equation}\nonumber
\begin{split}
 &\leq C\Bigg(E\Big[\int_0^T|b_s-b_s^n|^2ds\Big]+E\Big[\int_0^T|\sigma_s-\sigma_s^n|^2ds\Big]+E\Big[\int_0^T\int_K|\beta_s(e)-\beta_s^n(e)|^2\lambda(de)ds\Big]\Bigg)
\\
&\rightarrow 0,\ \text{as}\ n\rightarrow\infty.
\end{split}
\end{equation}
From {\it Step 1} we know that for all $0\leq t\leq t+h\leq T$,
\begin{equation}\label{equli 2}
\begin{split}
&f(P_{X^n_{t+h}})-f(P_{X^n_t})
=\int_t^{t+h}E\Big[(\partial_\mu f)(P_{X^n_s},X^n_s)b^n_s+\frac{1}{2}\partial_y(\partial_\mu f)(P_{X^n_s},X^n_s)|\sigma^n_s|^2\\
&+\int_K\int_0^1\Big\{(\partial_\mu f)\big(P_{X^n_s},X^n_s+\rho\beta^n_s(e)\big)-(\partial_\mu f)(P_{X^n_s},X^n_s)\Big\}\beta^n_s(e)d\rho\lambda(de)\Big]ds.
\end{split}
\end{equation}
We observe that, for
\begin{equation}\nonumber
\begin{split}
V_s^n:=& (\partial_\mu f)(P_{X^n_s},X^n_s)b^n_s+\frac{1}{2}\partial_y(\partial_\mu f)(P_{X^n_s},X^n_s)|\sigma^n_s|^2+\int_K\int_0^1\Big\{(\partial_\mu f)\big(P_{X^n_s},X^n_s+\rho\beta^n_s(e)\big)\\
&-(\partial_\mu f)(P_{X^n_s},X^n_s)\Big\}\beta^n_s(e)d\rho\lambda(de),
\end{split}
\end{equation}
i) $V^n\rightarrow V$ in measure $dsdP$, where $V$ is defined in the same way as $V^n$, but with $(X,b,\sigma,\beta)$

\noindent\ \ \ instead of $(X^n,b^n,\sigma^n,\beta^n)$.\\
ii) As $\partial_\mu f$ and $\partial_y(\partial_\mu f)$ are bounded,
$$|V_s^n|\leq C\Big(|b_s^n|+|\sigma_s^n|^2+\int_K|\beta_s^n(e)|^2\lambda(de)\Big),\ s\in[0,T],\ n\geq 1.$$
But, as $b^n\rightarrow b$, $\sigma^n\rightarrow\sigma$ in ${\cal H}^2_{\mathbb{F}}(0,T)$ and $\beta^n\rightarrow\beta$ in ${\cal K}_\lambda^2(0,T)$, the right hand side is a uniformly integrable sequence over $[0,T]\times\Omega$. Consequently, we can apply Lebesgue's convergence Theorem to (\ref{equli 2}) and we get
\begin{equation}\label{equli 3}
\begin{split}
&\!\!\!\!f(P_{X_{t+h}})-f(P_{X_t})
=\int_t^{t+h}E\Big[(\partial_\mu f)(P_{X_s},X_s)b_s+\frac{1}{2}\partial_y(\partial_\mu f)(P_{X_s},X_s)|\sigma_s|^2\\
&\!\!\!\!+\int_K\!\int_0^1\!\big\{(\partial_\mu f)\big(P_{X_s},X_s\!+\!\rho\beta_s(e)\big)-(\partial_\mu f)(P_{X_s},X_s)\big\}\beta_s(e)d\rho\lambda(de)\Big]ds,\  0\leq t\leq t+h\leq T.
\end{split}
\end{equation}
{\it Step 3}. Let now $b,\sigma,\beta$ be as in {\it Step 2}. For $u\in L^0_{\mathbb{F}}\big(\Omega,L^1(0,T)\big)$, $v\in L^0_{\mathbb{F}}\big(\Omega,L^2(0,T)\big)$, $\gamma\in {\cal K}^0_\lambda(0,T)$ with $|\gamma_s(e)|\leq \zeta (1\wedge |e|)$, P-a.s., $(s,e)\in [0, T]\times K$, for some real-valued random variable $\zeta\geq0$, $P$-a.s., and $U_0\in L^0(\mathcal{F}_0)$, we consider the It\^{o} process
$$U_t= U_0+\int_0^tu_sds+\int_0^tv_sdB_s+\int_0^t\int_K\gamma_s(e)N_{\lambda}(ds,de),\ t\in [0,T].$$
Let $F\in C^{1,2,2}\big([0,T]\times\mathbb{R}\times\mathcal{P}_2(\mathbb{R})\big)$, we emphasize that we do not need the existence of the second order mixed derivatives $\partial_x\partial_\mu F$, $\partial_\mu\partial_x F$, nor that of $\partial_\mu(\partial_\mu F)$, unlike \cite{BLPR} and \cite{HL3}. Under the above assumptions then we have the It\^o formula.

Indeed, let us first suppose that $|b_s|+|\sigma_s|\leq C$, $|\beta_s(e)|\leq C(1\wedge|e|)$ and $b_s,\sigma_s,\beta_s(e)$ are continuous with respect to $s$. Then we see from (\ref{equli 3}) that
\begin{equation}\nonumber
\begin{split}
&\partial_t\big[F(t,x,P_{X_t})\big]=(\partial_tF)(t,x,P_{X_t})+E\big[(\partial_\mu F)(t,x,P_{X_t},X_t)b_t+\frac{1}{2}\partial_y(\partial_\mu F)(t,x,P_{X_t},X_t)|\sigma_t|^2\\
&+\int_K\int_0^1\big\{(\partial_\mu F)\big(t,x,P_{X_t},X_t+\rho\beta_t(e)\big)-(\partial_\mu F)(t,x,P_{X_t},X_t)\big\}\beta_t(e)d\rho\lambda(de)\big]
\end{split}
\end{equation}
is continuous in $(t, x)$, i.e., $G(t,x):= F(t,x,P_{X_t})$, $(t,x)\in[0,T]\times\mathbb{R}$, belongs to $C^{1,2}([0,T]\times\mathbb{R})$. But this means that we can apply to $F(t,U_t,P_{X_t})=G(t,U_t)$ the classical It\^{o} formula. This yields
\begin{equation*}
\begin{split}
&dF(t,U_t,P_{X_t})=dG(t,U_t)\\
=&\Big\{(\partial_tG)(t,U_t)+(\partial_xG)(t,U_t)u_t
+\frac{1}{2}(\partial_x^2G)(t,U_t)v_t^2+\int_K\Big(G(t,U_t+\gamma_t(e))-G(t,U_t)-(\partial_xG)(t,U_t)\\
&\ \gamma_t(e)\Big)\lambda(de)\Big\}dt
+(\partial_xG)(t,U_t)v_tdB_t+\int_K\Big(G(t,U_{t-}+\gamma_t(e))-G(t,U_{t-})\Big)N_\lambda(dt,de)\\
\end{split}
\end{equation*}
\begin{equation}\label{(*100)}
\begin{split}
=&\Big\{(\partial_tF)(t,U_t,P_{X_t})+(\partial_xF)(t,U_t,P_{X_t})u_t+\frac{1}{2}(\partial_x^2F)(t,U_t,P_{X_t})v_t^2
+\int_K\big(F(t,U_t+\gamma_t(e),P_{X_t})\\
&-F(t,U_t,P_{X_t})-(\partial_xF)(t,U_t,P_{X_t})\gamma_t(e)\big)\lambda(de)\Big\}dt
+\hat{E}\Big[(\partial_\mu F)(t,U_t,P_{X_t},\hat{X}_t)\hat{b}_t+\frac{1}{2}\partial_y(\partial_\mu F)(t,U_t,\\
&P_{X_t},\hat{X}_t)|\hat{\sigma}_t|^2
+\int_K\int_0^1\Big\{(\partial_\mu F)\big(t,U_t,P_{X_t},\hat{X}_t+\rho\hat{\beta}_t(e)\big)-(\partial_\mu F)(t,U_t,P_{X_t},\hat{X}_t)\Big\}\hat{\beta}_t(e)d\rho\lambda(de)\Big]dt\\
&+(\partial_xF)(t,U_t,P_{X_t})v_tdB_t
+\int_K\Big(F(t,U_{t-}+\gamma_t(e),P_{X_t})-F(t,U_{t-},P_{X_t})\Big)N_\lambda(dt,de),\ t\in[0,T].
\end{split}
\end{equation}

\noindent {\it Step 4}. For the general case let now $b, \sigma\in {\cal H}^2_{\mathbb{F}}(0,T)$ and $\beta\in {\cal K}^2_\lambda(0,T)$, and suppose that $b^n, \sigma^n, \beta^n$ satisfy the assumptions made on $b, \sigma, \beta$ in {\it Step 3} and $b^n\rightarrow b,\ \sigma^n\rightarrow \sigma$ in ${\cal H}^2_{\mathbb{F}}(0,T)$ and $\beta^n\rightarrow\beta$ in ${\cal K}^2_\lambda(0,T)$. From {\it Step 3} we have the It\^{o} formula (\ref{(*100)}) for $b^n, \sigma^n, \beta^n$ and $X^n$ (defined by (\ref{(14)})) instead of $b, \sigma, \beta$ and $X$, and we have to take the limit as $n\rightarrow \infty$.

For this notice that, as $E[\sup_{t\in [0,T]}|X_t^n-X_t|^2]\rightarrow 0, n\rightarrow 0$ (see {\it Step 2}), we have in particular that
$$W_2(P_{X^n}, P_X)^2\leq E[\sup_{t\in [0,T]}|X_t^n-X_t|^2]\rightarrow 0, n\rightarrow 0.$$
Hence, $\{P_{X^n}, n\geq 1\}$ is relatively compact in ${\cal P}_2(D([0,T]))$. $D([0,T])$ is the space of c\`{a}dl\`{a}g functions over $[0, T]$, endowed with the supermum norm. Combining this with the pathwise c\`{a}dl\`{a}g property of the process $U$ and the continuity of $\partial_x^2F: [0, T]\times {\mathbb R}\times {\cal P}_2({\mathbb R})\rightarrow {\mathbb R}$, we see that $\sup_{n\geq 1, t\in [0,T]}|\partial_x^2F(t,U_t,P_{X_t^n})|<\infty$, P-a.s., and

\noindent i) $\displaystyle \int_0^t\partial_x^2F(s,U_s,P_{X_s^n})v^2_sds\rightarrow \int_0^t\partial_x^2F(s,U_s,P_{X_s})v^2_sds $, P-a.s., from the Dominated Convergence Theorem.

\noindent Using the same arguments and the fact that $|\gamma_t(e)|\leq \zeta (1\wedge |e|)$, we see that
\begin{equation*}
\begin{split}
& \sup_{n\geq 1, t\in [0,T]}|F(t, U_{t-}+\gamma_t(e), P_{X_t^n})-F(t, U_{t-}, P_{X_t^n})|^2\\
&\leq \sup_{n\geq 1, t\in [0,T], \rho\in [0,1]}|\partial_xF(t, U_{t-}+\rho\gamma_t(e), P_{X_t^n})|^2|\zeta|^2 (1\wedge |e|^2),
\end{split}
\end{equation*}
and
\begin{equation*}
\begin{split}
& \sup_{n\geq 1, t\in [0,T]}|F(t, U_{t}+\gamma_t(e), P_{X_t^n})-F(t, U_{t}, P_{X_t^n})-\partial_xF(t, U_{t}, P_{X_t^n})\gamma_t(e)|^2\\
&\leq \sup_{n\geq 1, t\in [0,T], \rho\in [0,1]}|\partial^2_xF(t, U_{t}+\rho\gamma_t(e), P_{X_t^n})|^2|\zeta|^2 (1\wedge |e|^2),
\end{split}
\end{equation*}
with $\displaystyle \sup_{n\geq 1, t\in [0,T], \rho\in [0,1]}|\partial_xF(t, U_{t}+\rho\gamma_t(e), P_{X_t^n})|< +\infty$, and $\displaystyle \sup_{n\geq 1, t\in [0,T], \rho\in [0,1]}|\partial^2_xF(t, U_{t}+\rho\gamma_t(e), P_{X_t^n})|< +\infty$, P-a.s. This allows to show that, for all $t\in [0, T]$, P-a.s., as $n\rightarrow \infty$,

\noindent ii) $\displaystyle\int_0^t\int_K\Big(F(s, U_{s-}+\gamma_s(e), P_{X_s^n})-F(s, U_{s-}, P_{X_s^n})\Big)N_\lambda(ds,de)\rightarrow \int_0^t\int_K\Big(F(s, U_{s-}+\gamma_s(e), P_{X_s})-F(s, U_{s-}, P_{X_s})\Big)N_\lambda(ds,de),$

\noindent and

\noindent iii) $\displaystyle \int_0^t\int_K\Big(F(s, U_{s}+\gamma_s(e), P_{X_s^n})-F(s, U_{s}, P_{X_s^n})-\partial_xF(s, U_{s}, P_{X_s^n})\gamma_s(e)\Big)\lambda(de)ds\rightarrow \int_0^t\int_K\Big(F(s,$ $ U_{s}+\gamma_s(e), P_{X_s})-F(s, U_{s}, P_{X_s})-\partial_xF(s, U_{s}, P_{X_s})\gamma_s(e)\Big)\lambda(de)ds.$

On the other hand, from the boundedness of $\partial_y(\partial_\mu F)$ (with some bound $C>0$) we see that

\noindent$\displaystyle |(\partial_\mu F)(t, U_t, P_{X_t^n}, \widehat{X}_t^n+\rho\widehat{\beta}_t^n(e))-(\partial_\mu F)(t, U_t, P_{X_t^n}, \widehat{X}_t^n)||\widehat{\beta}_t^n(e)|\leq C|\widehat{\beta}_t^n(e)|^2,\ (t, e)\in [0, T]\times E$, P-a.s. This allows to conclude from Lebesgue's Convergence Theorem that

\noindent iv) $\displaystyle \widehat{E}\Big[\int_0^t\int_K\int_0^1\Big((\partial_\mu F)(s, U_s, P_{X_s^n}, \widehat{X}_s^n+\rho\widehat{\beta}_s^n(e))-(\partial_\mu F)(s, U_s, P_{X_s^n}, \widehat{X}_s^n) \Big)\widehat{\beta}_s^n(e)d\rho\lambda(de)ds\Big]\rightarrow $\\
\noindent$\displaystyle \widehat{E}\Big[\int_0^t\int_K\int_0^1\Big((\partial_\mu F)(s, U_s, P_{X_s}, \widehat{X}_s+\rho\widehat{\beta}_s(e))-(\partial_\mu F)(s, U_s, P_{X_s}, \widehat{X}_s) \Big)\widehat{\beta}_s(e)d\rho\lambda(de)ds\Big]$,

\noindent for all $t\in [0, T]$, P-a.s., $n\rightarrow \infty$.

An analogous discussion for the remaining terms in (\ref{(*100)}) shows that the we have also for them the convergence. The proof is complete.
\end{proof}

\subsection{Mean field BSDEs with jumps}
We first give two classical estimates for the solutions of BSDEs with jumps, the proof is standard, the readers may refer to, e.g., \cite{BBP}, \cite{LW1} and  \cite{LW2}.
\begin{lemma}\label{leli 4.1}
Suppose $(Y^i,Z^i,H^i)$ is the unique solution of the following BSDE with data $(g_i,\theta^i)$,
\begin{equation}\label{equ 4.0}
\left\{
\begin{array}{l}
dY_s^i=-g_i(s,Y_s^i,Z_s^i,H_s^i)ds+Z_s^idB_s+\int_KH_s^i(e)N_{\lambda}(ds,de),\\
Y_T^i=\theta^i,
\end{array}
\right.
\end{equation}
where $\theta^i\in L^2(\mathcal{F}_T)$, and $g_i:[0,T]\times\Omega\times\mathbb{R}\times\mathbb{R}^d\times L^2(K, {\cal{B}}(K),\lambda; \mathbb{R})\rightarrow\mathbb{R}$, $i=1,2$, respectively, are ${\mathbb F}$-predictable and satisfy:\\
\noindent$\mathbf{Assumption\ (H10.0)}\ ${\rm i}$)$ $g_i(\cdot,\cdot,0,0,0)\in {\cal H}^2_{\mathbb{F}}(0,T)$;\\
{\rm ii}$)$ There exists a constant $c^*>0$ such that, P-a.s., for any $t\in [0,T]$, $y_1, y_2\in \mathbb{R}$, $z_1, z_2\in \mathbb{R}^d$, $h_1, h_2\in L^2(K, {\cal B}(K), \lambda)$, $|g_i(t,y_1,z_1,h_1)-g_i(t,y_2,z_2,h_2)|\leq c^*(|y_1-y_2|+|z_1-z_2|+|h_1-h_2|)$.

\smallskip

For $(\overline{Y},\overline{Z},\overline{H}):= (Y^1,Z^1,H^1)-(Y^2,Z^2,H^2)$, $\overline{g}:= g_1-g_2$, $\overline{\theta}:=\theta^1-\theta^2$, we have the following estimates:\\
1$)$ For all $\delta>0$, there exists a suitable $\beta(\geq \frac{1}{2}+2c^*(1+2c^*+\frac{1}{2\delta}))$ such that
\begin{equation}\label{100-1}
\begin{split}
&|\overline{Y}_t|^2+\frac{1}{2}E\Big[\int_t^Te^{\beta(s-t)}\big(|\overline{Y}_s|^2
+|\overline{Z}_s|^2+\int_K|\overline{H}_s(e)|^2\lambda(de)\big)ds|\mathcal{F}_t\Big]\\
\leq& E\big[e^{\beta(T-t)}|\overline{\theta}|^2|\mathcal{F}_t\big]+c^*\delta E\Big[\int_t^Te^{\beta(s-t)}|\overline{g}(s,Y_s^1,Z_s^1,H_s^1)|^2ds|
\mathcal{F}_t\Big],\ P\text{-a.s.},\ t\in[0,T].
\end{split}
\end{equation}
2$)$ For all $p\geq 2$, there exists $C_p>0$ (only depending on $p$ and the Lipschitz constants) such that
\begin{equation}\label{99}
\begin{split}
&E\Big[\sup_{s\in[t,T]}|\overline{Y}_s|^p+(\int_t^T|\overline{Z}_s|^2ds)^{\frac{p}{2}}
+\big(\int_t^T\int_K|\overline{H}_s(e)|^2\lambda(de)ds\big)^{\frac{p}{2}}|\mathcal{F}_t\Big]\\
\leq& C_pE\big[|\overline{\theta}|^p+\big(\int_t^T|\overline{g}(s,Y_s^1,Z_s^1,H_s^1)|ds\big)^{p}|
\mathcal{F}_t\big],\ P\text{-a.s.},\ t\in[0,T].
\end{split}
\end{equation}
\end{lemma}

We now consider a more general case of BSDE (\ref{equ 4.0}). Let
$f:[0,T]\times\Omega\times\mathbb{R}\times\mathbb{R}^{d}\times L^2(K, {\cal B}(K), \lambda)\times\mathcal{P}_{2}(\mathbb{R}^{d}\times\mathbb{R}\times\mathbb{R}^{d}\times L^2(K, {\cal B}(K), \lambda))\rightarrow\mathbb{R}$ be ${\mathbb F}$-predictable and satisfy the following assumptions.\\
\noindent$\mathbf{Assumption\ (H10.1)}$ {\rm i)} $f(\cdot,\cdot,0,0,0,\delta_{\textbf{0}})\in {\cal H}^2_{\mathbb{F}}(0,T)$;\\
{\rm ii)} $f$ is Lipschitz with respect to $(y,z,k,\mu)\in\mathbb{R}\times\mathbb{R}^d\times L^2(K, {\cal B}(K), \lambda)\times\mathcal{P}_2(\mathbb{R}^{d+1+d}\times L^2(K, {\cal B}(K), \lambda))$,\\
\mbox{ } \ \  uniformly with respect to $(s,\omega)$;\\
{\rm iii)} $\xi\in L^{2}(\mathcal{F}_{T})$;\\
{\rm iv)} $X\in S^{2}_{\mathbb{F}}(0,T;\mathbb{R}^{d})$.
\begin{theorem}\label{thli 4.12.1}
Under the Assumption (H10.1), the following mean-field BSDE with jumps
\begin{equation}\label{eq:zx4}
\bigg\{
\begin{tabular}{l}
$dY_{s}=-f(s,Y_{s},Z_{s},H_{s},P_{(X_{s},Y_{s},Z_{s},H_{s})})ds+Z_{s}dB_{s}+\int_{K}H_{s}(e)N_{\lambda}(ds,de),$ \\
$Y_{T}=\xi,$  \\
\end{tabular}
\end{equation}
has a unique solution $(Y, Z, H)\in S^{2}_{\mathbb{F}}(0,T;\mathbb{R})\times \mathcal{H}^{2}_{\mathbb{F}}(0,T;\mathbb{R}^{d})\times \mathcal{K}^{2}_{\lambda}(0,T;\mathbb{R})$.
\end{theorem}
\begin{proof}
Let $(U,V,W)\in {\cal M}:= \mathcal{H}^{2}_{\mathbb{F}}(0,T;\mathbb{R})\times \mathcal{H}^{2}_{\mathbb{F}}(0,T;\mathbb{R}^{d})\times \mathcal{K}^{2}_{\lambda}(0,T;\mathbb{R})$,
then there is a unique

\noindent solution $(Y,Z,H)\in S^{2}_{\mathbb{F}}(0,T;\mathbb{R})\times \mathcal{H}^{2}_{\mathbb{F}}(0,T;\mathbb{R}^{d})\times \mathcal{K}^{2}_{\lambda}(0,T;\mathbb{R})$
of the BSDE
\begin{equation*}
\left\{
\begin{array}{l}
dY_{s}=-f(s,Y_{s},Z_{s},H_{s},P_{(X_{s},U_{s},V_{s},W_{s})})ds+Z_{s}dB_{s}+\int_{K}H_{s}(e)N_{\lambda}(ds,de),\\
Y_{T}=\xi.
\end{array}
\right.
\end{equation*}
We define the mapping $\Phi(U,V,W):=(Y,Z,H): {\cal M}\rightarrow {\cal M}$. Let $(U^{i},V^{i},W^{i})\in {\cal M},$ $(Y^{i},Z^{i},H^{i}):=\Phi(U^{i},V^{i},W^{i}),i=1,2.$ Let $(\overline{Y},\overline{Z},\overline{H}):=(Y^{1},Z^{1},H^{1})-(Y^{2},Z^{2},H^{2})$,
$(\overline{U},\overline{V},\overline{W}):=(U^{1},V^{1},W^{1})-(U^{2},V^{2},W^{2})$. Then
from (\ref{100-1}), for $C$ is the Lipschitz constant of $f$, we have
\begin{equation*}
\begin{split}
\|(\overline{Y},\overline{Z},\overline{H})\|_{\beta}^{2}
&:=\frac{1}{2}E[\int_{0}^{T}e^{\beta s}(|\overline{Y}_{s}|^{2}+|\overline{Z}_{s}|^{2}+\int_{K}|\overline{H}_{s}(e)|^{2}\lambda(de))ds]\\
&\leq\delta C\int_{0}^{T}e^{\beta s}W_{2}(P_{(X_{s},U_{s}^{1},V_{s}^{1},W_{s}^{1})},P_{(X_{s},U_{s}^{2},V_{s}^{2},W_{s}^{2})})^{2}ds\\
&\leq\delta CE[\int_{0}^{T}e^{\beta s}(|\overline{U}_{s}|^{2}+|\overline{V}_{s}|^{2}+\int_{K}|\overline{W}_{s}(e)|^{2}\lambda(de))ds]\\
&=2\delta C\|(\overline{U},\overline{V},\overline{W})\|_{\beta}^{2}.
\end{split}
\end{equation*}
Consequently, for choosing $\delta>0$ such that $2\delta C\leq\frac{1}{4}$, then $\Phi:({\cal M},\|\cdot\|_{\beta})\rightarrow({\cal M},\|\cdot\|_{\beta})$ is a contraction mapping,
i.e., there exists a unique fixed point $(Y,Z,H)\in {\cal M}$ such that $(Y,Z,H)=\Phi(Y,Z,H)$, which means $(Y,Z,H)$ is the solution of (\ref{eq:zx4}).
\end{proof}
Similar to Lemma 10.1, we also have the following estimate for mean-field BSDE with jumps.
\begin{theorem}\label{newproposition}
Let $(Y^i,Z^i,H^i)$ be the unique solution of the following BSDE with data $(f_i,\xi^i)$,
\begin{equation}
\bigg\{
\begin{tabular}{l}
$dY^i_{s}=-f^i(s,Y^i_{s},Z^i_{s},H^i_{s},P_{(X_{s},Y^i_{s},Z^i_{s},H^i_{s})})ds+Z^i_{s}dB_{s}+\int_{K}H^i_{s}(e)N_{\lambda}(ds,de),$ \\
$Y^i_{T}=\xi^i,$  \\
\end{tabular}
\end{equation}
where $(X, f_i,\xi^i)$ satisfy {\rm{(H10.1)}}, $i=1, 2$, respectively.

\noindent We denote $(\overline{Y},\overline{Z},\overline{H}):= (Y^1,Z^1,H^1)- (Y^2,Z^2,H^2)$, $\overline{f}:= f^1-f^2$, $\overline{\xi}:=\xi^1-\xi^2$. Then there exists a constant $C>0$ such that, P-a.s., $t\in[0,T],$
\begin{equation*}
\begin{split}
E\Big[\sup_{s\in[t,T]}|\overline{Y}_s|^2\!+\!\int_t^T\!\!\Big(|\overline{Z}_s|^2ds\!+\!\int_K\!\!|\overline{H}_s(e)|^2\lambda(de)\Big)ds|\mathcal{F}_t\Big]\\
\!\leq\! C E\big[|\overline{\xi}|^2\!+\!\big(\int_t^T\!\!|\overline{f}(s,Y_s^1,Z_s^1,H_s^1,P_{(X_{s},Y^1_{s},Z^1_{s},H^1_{s})})|ds\big)^{2}|\mathcal{F}_t\big].
\end{split}
\end{equation*}
In particular, if $f^i(s,y,z,h,P_{(Y_s,Z_s,H_s)})=u^i(s,y,z,h)+\widehat{E}[v^i(s,\widehat{Y}_s,\widehat{Z}_s,\widehat{H}_s)]$, where $u^i, v^i:[0, T]\times \Omega\times\mathbb{R}\times\mathbb{R}^d\times L^2(K, {\cal B}(K),\lambda)$ both satisfy the assumption {\rm{(H10.0)}}, and $(Y, Z, H)\in S^{2}_{\mathbb{F}}(0,T;\mathbb{R})\times \mathcal{H}^{2}_{\mathbb{F}}(0,T;\mathbb{R}^{d})\times \mathcal{K}^{2}_{\lambda}(0,T;\mathbb{R})$, $i=1, 2$, respectively. Then, we have P-a.s., $t\in[0,T],$
\begin{equation*}
\begin{split}
&E\Big[\sup_{s\in[t,T]}|\overline{Y}_s|^2\!+\!\int_t^T\!\!\Big(|\overline{Z}_s|^2ds\!+\!\int_K\!\!|\overline{H}_s(e)|^2\lambda(de)\Big)ds|\mathcal{F}_t\Big]\\
&\!\leq\! C E\big[|\overline{\xi}|^2\!+\!\big(\int_t^T\!\!|(u^1-u^2)(s,Y_s^1,Z_s^1,H_s^1)|ds\big)^{2}+\!\big(\int_t^T\!\!|\widehat{E}[(v^1-v^2)(s,\widehat{Y}_s^1,
\widehat{Z}_s^1,\widehat{H}_s^1)]|ds\big)^{2}|\mathcal{F}_t\big].
\end{split}
\end{equation*}

\end{theorem}

Now we give the estimates for a special type of mean field BSDEs with jumps, which are used frequently in our work. We suppose that\\
\noindent$\mathbf{Assumption\ (H10.2)}$ (i) $\xi\in L^2(\mathcal{F}_T)$;\\
(ii) $\alpha=(\alpha_s)$,$\gamma=(\gamma_s)$ are bounded $\mathbb{F}$-progressively measurable processes, and $\beta=(\beta_s(e))$ is $\mathbb{F}$-progressively measurable such that $|\beta_s(e)|\leq C(1\wedge|e|)$;\\
(iii) $\zeta=(\zeta_s),\theta=(\theta_s)$ are bounded $\mathbb{G}$-progressively measurable processes with $\mathcal{G}_t=\mathcal{F}_t\otimes\hat{\mathcal{F}_T}$, where $(\widehat{\Omega},\widehat{\mathcal{F}},\widehat{\mathbb{F}},\widehat{P})$ is a copy of $({\Omega},{\mathcal{F}},{\mathbb{F}},{P})$, and $\delta=(\delta_s(e))$ is bounded $\mathbb{G}$-progressively measurable such that $|\delta_s(e)|\leq C(1\wedge|e|)$;\\
(iv) $R=(R(s))$ is $\mathbb{F}$-progressively measurable with $E[(\int_0^T|R(r)|dr)^2]<+\infty$.

From Theorem 10.2 we get the following corollary directly.
\begin{corollary}\label{BSDE estimate} Suppose Assumption (H10.2) holds. Let $(Y,Z,H)\in \mathcal{S}^2_\mathbb{F}(0, T)\times\mathcal{H}^2_{\mathbb{F}}(0, T; \mathbb{R}^d)\times\mathcal{K}^2_\lambda(0, T)$ and $(\widehat{Y},\widehat{Z},\widehat{H})$ be a copy of $(Y,Z,H)$ on $(\widehat{\Omega},\widehat{\mathcal{F}},\widehat{P})$ (i.e., $P_{(Y,Z,H)}=\widehat{P}_{(\widehat{Y},\widehat{Z},\widehat{H})}$) such that
\begin{equation}\label{100}
\begin{aligned}
Y_t&=\xi+\int_s^T\big(R(r)+\alpha_rY_r+\widehat{E}[\zeta_r\widehat{Y}_r]+\gamma_rZ_r+\hat{E}[\theta_r\widehat{Z}_r]+
\int_K\beta_r(e)H_r(e)\lambda(de)\\
&\quad+\widehat{E}\big[\int_K\delta_r(e)\widehat{H}_r(e)\lambda(de)\big]\big)dr-\int_s^TZ_rdB_r-\int_s^T\int_KH_r(e)N_\lambda(drde),\ s\in[t,T].
\end{aligned}
\end{equation}
Then there exists $C \in \mathbb{R}_+$ only depending on the bounds of the coefficients such that
\[
E\big[\mathop{\rm sup}_{t\in[0,T]}|Y_t|^2+\int_0^T\big(|Z_s|^2+\int_K|H_s(e)|^2\lambda(de)\big)ds\big]\leq C\big(E[|\xi|^2]+E\big[\big(\int_0^T|R(s)|ds\big)^2\big]\big).
\]
\end{corollary}

\begin{theorem}\label{proBSDE estimate} Suppose the Assumption (H10.2) holds. For all $p>1$, there exists $C_p\in\mathbb{R}_+$ only depending on the bounds of the coefficients, such that
$$
E[\sup_{t\in[0,T]}|Y_t|^{2p}+(\int_0^T(|Z_s|^2+\int_K|H_s(e)|^2\lambda(de))ds)^p]\leq C_pE[|\xi|^{2p}]+C_pE[(\int_0^T|R(s)|ds)^{2p}],
$$
where $(Y, Z, H)$ is the solution of BSDE (\ref{100}).
\end{theorem}
\begin{proof} For some $\beta>0$ (to be specified later) and $\delta>0$ we have
\begin{equation*}
\begin{aligned}
&e^{\beta t}|Y_t|^2+E\big[\int_t^Te^{\beta s}(\beta|Y_s|^2+|Z_s|^2+\int_K|H_s(e)|^2\lambda(de))ds\big|\mathcal{F}_t\big]\\
&\leq E[e^{\beta T}|\xi|^2|\mathcal{F}_t]+CE\Big[\int_t^Te^{\beta s}\big\{|Y_s|^2+|Y_s|\big(|Z_s|+\widehat{E}[|\widehat{Y}_s|+|\widehat{Z}_s|]+\int_K|H_s(e)|(1\wedge|e|)\lambda(de)\\
&+\widehat{E}\big[\int_K|\widehat{H}_s(e)|(1\wedge|e|)\lambda(de)\big]\big)\big\}ds\big|\mathcal{F}_t\Big]+CE\big[\int_t^Te^{\beta s}|Y_s||R(s)|ds\big|\mathcal{F}_t\big]\\
\end{aligned}
\end{equation*}
\begin{equation}\label{(*)}
\begin{aligned}
&\leq E[e^{\beta T}|\xi|^2|\mathcal{F}_t]+C'E\big[\int_t^T e^{\beta s}|Y_s|^2ds\big|\mathcal{F}_t\big]+\frac{1}{2}E\big[\int_t^Te^{\beta s}(|Z_s|^2+\int_K|H_s(e)|^2\lambda(de))ds\big|\mathcal{F}_t\big]\\
&+ \delta E[\sup_{s\in[t,T]}|Y_s|^2|\mathcal{F}_t]+C_\delta E\big[\big(\int_t^Te^{\beta s}|R(s)|ds\big)^2\big|\mathcal{F}_t\big],\ t\in[0,T].
\end{aligned}
\end{equation}
From (\ref{(*)}) ($\beta>C'+\frac{1}{2}$) we have
$$
\begin{aligned}
&\quad|Y_t|^2+\frac{1}{2}E[\int_t^T(|Y_s|^2+|Z_s^2|+\int_K|H_s(e)|^2\lambda(de))ds |\mathcal{F}_t]\\
&\leq CE[\xi^2|\mathcal{F}_t]+\delta E[\mathop{\rm sup}_{s\in[0,T]}|Y_s|^2|\mathcal{F}_t]+C_\delta E[(\int_0^T|R(s)|ds)^2|\mathcal{F}_t],\ t\in[0,T].
\end{aligned}
$$
Consequently, from Doob's martingale inequality,
$$
E[\sup_{t\in[0,T]}|Y_t|^{2p}]\leq C_pE[|\xi|^{2p}]+2^{p-1}\delta^p(\frac{p}{p-1})^pE[\mathop{\rm sup}_{t\in[0,T]}|Y_t|^{2p}]+C_pE[(\int_0^T|R(s)|ds)^{2p}].
$$
Let $\delta>0$ be small enough such that $(2\delta(\frac{p}{p-1}))^p\leq\frac{1}{2}$, then we have
\begin{equation}\label{(*2)}
E[\sup_{t\in[0,T]}|Y_t|^{2p}]\leq C_pE[|\xi|^{2p}]+C_pE[(\int_0^T|R(s)|ds)^{2p}].
\end{equation}
On the other hand, as for any right-continuous increasing process $A$ with continuous dual predictable projection $^p\!A$, from Burkholder-Davis-Gundy inequality we have $E[(^p\!A_{t+h}-^p\!\!A_{t})^p]\leq C_pE[(A_{t+h}-A_{t})^p]$, for some constant $C_p$ depending only on $p$, and
$$
\begin{aligned}
&E[(\int_t^{t+h}(|Z_s|^2+\int_K|H_s(e)|^2\lambda(de))ds)^p]\leq C_pE[(\int_t^{t+h}|Z_s|^2ds+\int_t^{t+h}\int_K|H_s(e)|^2N(ds,de))^p]\\
&\leq C_pE[\sup_{s\in[t,t+h]}|\int_t^sZ_rdB_r+\int_t^s\int_KH_r(e)N_\lambda(dr,de)|^{2p}]\leq C_pE[\sup_{s\in[0,T]}|Y_s|^{2p}]+C_pE\Big[\Big(\int_t^{t+h}\big(|Z_r|\\
&\quad+|R(r)|+\widehat{E}[|\widehat{Z}_r|]+\int_K|H_r(e)|(1\wedge|e|)\lambda(de)+\widehat{E}[\int_K|\widehat{H}_r(e)|(1\wedge|e|)\lambda(de)]
\big)dr\Big)^{2p}\Big]\\
&\leq C_pE[\sup_{s\in[0,T]}|Y_s|^{2p}]\!+\!C_pE[\Big(\int_t^{t+h}\!\!|R(r)|dr\Big)^{2p}]+C_ph^p E[\Big(\int_t^{t+h}\!\!(|Z_r|^2+E[|Z_r|^2]\!+\!\int_K\!\!|H_s(e)|^2\lambda(de)\\
&\quad\quad\quad+E[\int_K|H_s(e)|^2\lambda(de)])ds\Big)^p].
\end{aligned}
$$
Thus, from (\ref{(*2)}),

$$
\begin{aligned}
&E[(\int_t^{t+h}(|Z_s|^2+\int_K|H_s(e)|^2\lambda(de))ds)^p]\\
&\leq C_pE[|\xi|^{2p}]+C_pE[(\int_0^T|R(s)|ds)^{2p}]+C_ph^pE[(\int_t^{t+h}(|Z_s|^2+\int_K|H_s(e)|^2\lambda(de))ds)^p].
\end{aligned}
$$
Now letting $h>0$ be small enough such that $C_ph^p\leq\frac{1}{2}$, we get
$$
E[(\int_t^{t+h}(|Z_s|^2+\int_K|H_s(e)|^2\lambda(de))ds)^p]\leq C_pE[|\xi|^{2p}]+C_pE[(\int_0^T|R(s)|ds)^{2p}].
$$
Finally, let $0=t_1^{n}<t_2^n<\cdots t_n^n=T$ be the sequence of partitions of the interval $[0,T]$ of the form $t_i^n:=\frac{i}{n}T$, $0\leq i\leq n$, where  $n\geq n_0$ is such that $C_p(\frac{T}{n_0})^p\leq\frac{1}{2}$. Then we have
$$
E[(\int_{t_i^n}^{t_{i+1}^n}(|Z_s|^2+\int_K|H_s(e)|^2\lambda(de))ds)^p]\leq C_pE[|\xi|^{2p}]+C_pE[(\int_0^T|R(s)|ds)^{2p}].
$$
It follows that
$$
E[(\int_0^T(|Z_s|^2+\int_K|H_s(e)|^2\lambda(de))ds)^p]\leq C_pE[|\xi|^{2p}]+C_pE[(\int_0^T|R(s)|ds)^{2p}].
$$
\end{proof}
\subsection{Lemma for the proof of Proposition 9.1.}

Under the assumptions made for Proposition 9.1 we have

\begin{lemma}\label{newlem2}
There exists a constant $C>0$, such that, for all $t,t'\in[0,T],\ x,y\in\mathbb{R},\ \xi\in L^2(\Omega, {\cal F}_t, P)$, it holds
\begin{equation*}
\begin{split}
&{\rm i)}\ \  |V(t,x,P_{\xi})-V(t',x,P_{\xi})|\leq C|t'-t|,\\
&{\rm ii)}\ \  |\partial_{x}V(t,x,P_{\xi})-\partial_{x}V(t',x,P_{\xi})|\leq C|t'-t|^{\frac{1}{8}},\\
&{\rm iii)}\ \ |\partial_{x}^{2}V(t,x,P_{\xi})-\partial_{x}^{2}V(t',x,P_{\xi})|\leq C|t'-t|^{\frac{1}{8}},\\
&{\rm iv)}\ \  |\partial_{\mu}V(t,x,P_{\xi},y)-\partial_{\mu}V(t',x,P_{\xi},y)|\leq C|t'-t|^{\frac{1}{8}},\\
&{\rm v)}\ \  |\partial_{y}(\partial_{\mu}V)(t,x,P_{\xi},y)-\partial_{y}(\partial_{\mu}V)(t',x,P_{\xi},y)|\leq C|t'-t|^{\frac{1}{8}}.\\
\end{split}
\end{equation*}
\end{lemma}
\begin{proof} Let us prove Lemma \ref{newlem2} in four steps.

\noindent {\it Step 1}. To prove i) of Lemma 10.2.

For any $0\leq t\leq t+h\leq T,$ from (\ref{4.6-2}) and (\ref{4.6-3}) we have
\begin{equation}\label{(*101)}
\begin{split}
&V(t,x,P_{\xi})-V(t+h,x,P_{\xi})=Y_t^{t,x,P_{\xi}}-Y_{t+h}^{t+h,x,P_{\xi}}\\
&=(Y_t^{t,x,P_{\xi}}-Y_{t+h}^{t,x,P_{\xi}})+
(Y_{t+h}^{t+h,X_{t+h}^{t,x,P_{\xi}},P_{X_{t+h}^{t,\xi}}}-Y_{t+h}^{t+h,x,P_{\xi}})\\
&=(Y_t^{t,x,P_{\xi}}-Y_{t+h}^{t,x,P_{\xi}})+(V(t+h,X_{t+h}^{t,x,P_{\xi}},P_{X_{t+h}^{t,\xi}})-V(t+h,x,P_{\xi})).\\
\end{split}
\end{equation}
As $V(t+h,\cdot,\cdot)\in C^{2,2}({\mathbb R}\times{ {\cal P}_2}({\mathbb R}))$ with bounded continuous derivatives of 1st and 2nd order which are uniformly with respect to $t$, it follows from the It\^{o} formula-Theorem 2.1 that
\begin{equation}\nonumber
\begin{aligned}
&V(t+h,X_{t+h}^{t,x,P_{\xi}},P_{X_{t+h}^{t,\xi}})-V(t+h,x,P_{\xi})\\
&=\int_t^{t+h}\Big\{(\partial_xV)(t+h,X_s^{t,x,P_{\xi}},P_{X_s^{t,\xi}})b(X_s^{t,x,P_{\xi}},P_{X_s^{t,\xi}})\\
&\ \ \ +\frac{1}{2}(\partial_x^2V)
(t+h,X_s^{t,x,P_{\xi}},P_{X_s^{t,\xi}})\sigma(X_s^{t,x,P_{\xi}},P_{X_s^{t,\xi}})^2\\
&\ \ \ +\int_K\big(V(t+h,X_s^{t,x,P_{\xi}}+\beta(X_s^{t,x,P_{\xi}},P_{X_s^{t,\xi}},e),P_{X_s^{t,\xi}})-V(t+h,X_s^{t,x,P_{\xi}},P_{X_s^{t,\xi}})\\
&\ \ \ \ \ \ \ \ \ \ \ \ \ \ -(\partial_xV)(t+h,X_s^{t,x,P_{\xi}},P_{X_s^{t,\xi}})\beta(X_s^{t,x,P_{\xi}},P_{X_s^{t,\xi}},e)\big)\lambda(de)\Big\}ds\\
&+\int_t^{t+h}\widehat{E}\Big[(\partial_\mu V)(t+h,X_s^{t,x,P_{\xi}},P_{X_s^{t,\xi}},\widehat{X}_s^{t,\widehat{\xi}})b(\widehat{X}_s^{t,\widehat{\xi}},P_{X_s^{t,\xi}})
+\frac{1}{2}\partial_y(\partial_\mu V)(t+h,X_s^{t,x,P_{\xi}},P_{X_s^{t,\xi}},\widehat{X}_s^{t,\widehat{\xi}})\\
\end{aligned}
\end{equation}
\begin{equation}\label{C2-1}\begin{aligned}
&\ \ \ \cdot\sigma(\widehat{X}_s^{t,\widehat{\xi}},P_{X_s^{t,\xi}})^2+\int_K\int_0^1\big(\partial_\mu V(t+h,X_s^{t,x,P_{\xi}},P_{X_s^{t,\xi}},\widehat{X}_s^{t,\widehat{\xi}}+\rho\beta(\widehat{X}_s^{t,\widehat{\xi}},P_{X_s^{t,\xi}},e))\\
&\ \ \ -\partial_\mu V(t+h,X_s^{t,x,P_{\xi}},P_{X_s^{t,\xi}},\widehat{X}_s^{t,\widehat{\xi}})\big)\beta(\widehat{X}_s^{t,\widehat{\xi}},P_{X_s^{t,\xi}},e)d\rho\lambda(de)
\Big]ds\\
&+\int_t^{t+h}(\partial_xV)(t+h,X_s^{t,x,P_{\xi}},P_{X_s^{t,\xi}})\sigma(X_s^{t,x,P_{\xi}},P_{X_s^{t,\xi}})dB_s\\
&+\int_t^{t+h}\int_K(V(t+h,X_{s-}^{t,x,P_{\xi}}+\beta(X_{s-}^{t,x,P_{\xi}},P_{X_s^{t,\xi}},e),P_{X_s^{t,\xi}})-V(t+h,X_{s-}^{t,x,P_{\xi}},
P_{X_s^{t,\xi}}))N_\lambda(ds,de)\\
&= \int_t^{t+h}\theta(t,t+h,s)ds+\int_t^{t+h}\delta(t,t+h,s)ds\\
&+\int_t^{t+h}(\partial_xV)(t+h,X_s^{t,x,P_{\xi}},P_{X_s^{t,\xi}})\sigma(X_s^{t,x,P_{\xi}},P_{X_s^{t,\xi}})dB_s\\
&+\int_t^{t+h}\int_K(V(t+h,X_{s-}^{t,x,P_{\xi}}+\beta(X_{s-}^{t,x,P_{\xi}},P_{X_s^{t,\xi}},e),P_{X_s^{t,\xi}})-V(t+h,X_{s-}^{t,x,P_{\xi}},
P_{X_s^{t,\xi}}))N_\lambda(ds,de),
\end{aligned}
\end{equation}
where
\[
\begin{split}
&\theta(t,t+h,s)\\
:=&(\partial_xV)(t+h,X_s^{t,x,P_{\xi}},P_{X_s^{t,\xi}})b(X_s^{t,x,P_{\xi}},P_{X_s^{t,\xi}})+\frac{1}{2}(\partial_x^2V)
(t+h,X_s^{t,x,P_{\xi}},P_{X_s^{t,\xi}})\sigma(X_s^{t,x,P_{\xi}},P_{X_s^{t,\xi}})^2\\
& +\int_K\Big(V(t+h,X_s^{t,x,P_{\xi}}+\beta(X_s^{t,x,P_{\xi}},P_{X_s^{t,\xi}},e),P_{X_s^{t,\xi}})-V(t+h,X_s^{t,x,P_{\xi}},P_{X_s^{t,\xi}})\\
& \ \ \ -(\partial_xV)(t+h,X_s^{t,x,P_{\xi}},P_{X_s^{t,\xi}})\beta(X_s^{t,x,P_{\xi}},P_{X_s^{t,\xi}},e)\Big)\lambda(de),
\end{split}
\]
and
\[
\begin{split}
&\delta(t,t+h,s)\\
:=&\widehat{E}\Big[(\partial_\mu V)(t+h,X_s^{t,x,P_{\xi}},P_{X_s^{t,\xi}},\widehat{X}_s^{t,\widehat{\xi}})b(\widehat{X}_s^{t,\widehat{\xi}},P_{X_s^{t,\xi}})
+\frac{1}{2}\partial_y(\partial_\mu V)(t+h,X_s^{t,x,P_{\xi}},P_{X_s^{t,\xi}},\widehat{X}_s^{t,\widehat{\xi}})\cdot\\
&\ \ \ \sigma(\widehat{X}^{t,\widehat{\xi}},P_{X_s^{t,\xi}})^2+\int_K\int_0^1\Big(\partial_\mu V(t+h,X_s^{t,x,P_{\xi}},P_{X_s^{t,\xi}},\widehat{X}_s^{t,\widehat{\xi}}+\rho\beta(\widehat{X}_s^{t,\widehat{\xi}},P_{X_s^{t,\xi}},e))\\
&\ \ \ -\partial_\mu V(t+h,X_s^{t,x,P_{\xi}},P_{X_s^{t,\xi}},\widehat{X}_s^{t,\widehat{\xi}})\Big)\beta(\widehat{X}_s^{t,\widehat{\xi}},P_{X_s^{t,\xi}},e)d\rho\lambda(de)
\Big].
\end{split}
\]
On the other hand, since
\[Y_t^{t,x,P_{\xi}}-Y_{t+h}^{t,x,P_{\xi}}=\int_t^{t+h}f(\Pi_s^{t,x,P_{\xi}},P_{\Pi_s^{t,\xi}})ds-\int_t^{t+h}Z_s^{t,x,P_{\xi}}dB_s
-\int_t^{t+h}\int_KH_s^{t,x,P_{\xi}}(e)N_\lambda(ds,de),
\]
we have from (\ref{(*101)}) and (\ref{C2-1})
\begin{equation}\label{C2-2}
\begin{split}
&V(t,x,P_\xi)-V(t+h,x,P_\xi)=\int_t^{t+h}\Big(\theta(t,t+h,s)+\delta(t,t+h,s)+f(\Pi_s^{t,x,P_{\xi}},P_{\Pi_s^{t,\xi}})\Big)ds\\
&+\int_t^{t+h}\Big((\partial_xV)(t+h,X_s^{t,x,P_{\xi}},P_{X_s^{t,\xi}})\sigma(X_s^{t,x,P_{\xi}},P_{X_s^{t,\xi}})-Z_s^{t,x,P_{\xi}}\Big)dB_s\\
&+\int_t^{t+h}\int_K\Big(V(t+h,X_{s-}^{t,x,P_{\xi}}+\beta(X_{s-}^{t,x,P_{\xi}},P_{X_s^{t,\xi}},e),P_{X_s^{t,\xi}})-V(t+h,X_{s-}^{t,x,P_{\xi}},
P_{X_s^{t,\xi}})\\
&\ \ \ -H_s^{t,x,P_{\xi}}(e)\Big)N_\lambda(ds,de).
\end{split}
\end{equation}
As $|\theta(t,t+h,s)|,\ |\delta(t,t+h,s)|\leq C$, $f$ is bounded, and $V(t,x,P_\xi)-V(t+h,x,P_\xi)$ is deterministic, we get by taking expectation in the preceding equality:
\begin{equation}\label{C2-3}
|V(t,x,P_\xi)-V(t+h,x,P_\xi)|\leq C|h|,\ t,\ t+h\in [0,T].
\end{equation}

\noindent {\it Step 2}. We have the following representation formulas for $Z^{t,x,P_\xi}$ and $H^{t,x,P_\xi}(.)$:
\begin{equation}\label{(*102)}
\begin{split}
&\!\! Z_s^{t,x,P_{\xi}}=\partial_xV(s,X_s^{t,x,P_{\xi}},P_{X_s^{t,\xi}})\sigma(X_s^{t,x,P_{\xi}},P_{X_s^{t,\xi}}),\ dsdP\mbox{-}a.e.,\\
&\!\! H_s^{t,x,P_{\xi}}(e)=V(s,X_s^{t,x,P_{\xi}}+\beta(X_s^{t,x,P_{\xi}},P_{X_s^{t,\xi}},e),P_{X_s^{t,\xi}})\!-\!V(s,X_s^{t,x,P_{\xi}},
P_{X_s^{t,\xi}}),\ dsd\lambda dP\mbox{-}a.e.
\end{split}
\end{equation}

Indeed, we get from (\ref{C2-2}) combined with (\ref{C2-3})
\[
\begin{split}
&E[\int_t^{t+h}\Big(|\partial_xV(s,X_s^{t,x,P_{\xi}},P_{X_s^{t,\xi}})\sigma(X_s^{t,x,P_{\xi}},P_{X_s^{t,\xi}})-Z_s^{t,x,P_{\xi}}|^2\\
&\ \ +\int_K|V(s,X_s^{t,x,P_{\xi}}+\beta(X_s^{t,x,P_{\xi}},P_{X_s^{t,\xi}},e),P_{X_s^{t,\xi}})-V(s,X_s^{t,x,P_{\xi}},P_{X_s^{t,\xi}})
-H_s^{t,x,P_{\xi}}(e)|^2\lambda(de)\Big)ds|{\cal F}_t]\\
&\leq C|h|^2,\ 0\leq t\leq t+h\leq T,\ x\in\mathbb{R},\ \xi\in L^2({\cal F}_t).
\end{split}
\]
Consequently, considering a partition $t_i^n=t+hi2^{-n},\ 0\leq i\leq 2^n$, we have from the preceding estimate applied to $(t_i^n,t_{i+1}^n)$ instead of $(t,t+h)$:
\[
\begin{split}
&E[\int_t^{t+h}(|\partial_xV(s,X_s^{t,x,P_{\xi}},P_{X_s^{t,\xi}})\sigma(X_s^{t,x,P_{\xi}},P_{X_s^{t,\xi}})-Z_s^{t,x,P_{\xi}}|^2\\
&+\int_K|V(s,X_s^{t,x,P_{\xi}}+\beta(X_s^{t,x,P_{\xi}},P_{X_s^{t,\xi}},e),P_{X_s^{t,\xi}})-V(s,X_s^{t,x,P_{\xi}},P_{X_s^{t,\xi}})-H_s^{t,x,P_{\xi}}(e)|^2\lambda(de))ds]\\
&=E[\sum_{i=0}^{2^n-1}E[\int_{t_i^n}^{t_{i+1}^n}\Big(|\partial_xV(s,X_s^{t_i^n,y,P_{\eta}
},P_{X_s^{t_i^n,\eta}})\sigma(X_s^{t_i^n,y,P_{\eta}
},P_{X_s^{t_i^n,\eta}})-Z_s^{t_i^n,y,P_{\eta}}|^2\\
&\ \ \ +\int_K|V(s,X_s^{t_i^n,y,P_{\eta}}+\beta(X_s^{t_i^n,y,P_{\eta}},P_{X_s^{t_i^n,{\eta}}},e),P_{X_s^{t_i^n,{\eta}}})-V(s,X_s^{t_i^n,y,P_{\eta}},
P_{X_s^{t_i^n,{\eta}}})-H_s^{t_i^n,y,P_{\eta}}(e)|^2\lambda(de)\Big)\\
&\ \ \ ds|{\cal F}_{t_i^n}]|_{y=X_{t_i^n}^{t,x,P_{\xi}},\eta=X_{t_i^n}^{t,\xi}}]\\
&\leq C\sum_{i=0}^{2^n-1}(h2^{-n})^2=Ch^22^{-n}\rightarrow 0, \ \mbox{as}\ n\rightarrow \infty.
\end{split}
\]
Then (\ref{(*102)}) follows.

\noindent {\it Step 3}. To prove ii) and  iii) of Lemma \ref{newlem2}.

We restrict here to prove iii) which is slightly more involved, but uses in principle the same argument as that needed for ii). We notice that, in virtue of Step 2, as $V(s,\cdot,\cdot)\in C^{2,2}({\mathbb R}\times{ {\cal P}_2}({\mathbb R}))$, we have
\begin{equation}\label{10.11-1}
\begin{split}
&\partial_xZ_s^{t,x,P_{\xi}}=(\partial_x^2V\sigma+\partial_xV\partial_x\sigma)(s,X_s^{t,x,P_{\xi}},P_{X_s^{t,\xi}})\partial_xX_s^{t,x,P_{\xi}},\ dsdP\mbox{-}a.e.;\\
&\partial_xH_s^{t,x,P_{\xi}}(e)=\partial_xX_s^{t,x,P_{\xi}}\Big\{(\partial_xV(s,X_s^{t,x,P_{\xi}}+\beta(X_s^{t,x,P_{\xi}},P_{X_s^{t,\xi}},e),P_{X_s^{t,\xi}})(1+\partial_x\beta(X_s^{t,x,P_{\xi}},P_{X_s^{t,\xi}},e))\\
&\quad\quad\quad\quad\quad\quad\quad-\partial_xV(s,X_s^{t,x,P_{\xi}},P_{X_s^{t,\xi}})\Big\},\ dsd\lambda dP\mbox{-}a.e.,
\end{split}
\end{equation}
hence we can get
\[
\begin{split}
&E[\esssup\limits_{s\in [t,T]}|\partial_xZ_s^{t,x,P_{\xi}}|^p]\leq CE[\sup\limits_{s\in [t,T]}|\partial_xX_s^{t,x,P_{\xi}}|^p]\leq C_p,\\
&E[\esssup\limits_{s\in [t,T]}(\int_K|\partial_xH_s^{t,x,P_{\xi}}(e)|^2\lambda(de))^\frac{p}{2}]\leq C_pE[\sup\limits_{s\in [t,T]}|\partial_xX_s^{t,x,P_{\xi}}|^p]\leq C_p.
\end{split}
\]
For simplicity, in order to concentrate on the hard kernel of the proof, let $\Phi(x,\mu)=\Phi(x), f=f(\Pi_{r}^{t,x,P_{\xi}}$, $P_{\Pi_{r}^{t,\xi}})$ with $\Pi_{r}^{t,x,P_{\xi}}
=\int_{K}H_{r}^{t,x,P_{\xi}}(e)l(e)\lambda(de)$, and recall that $\Pi_{r}^{t,\xi}=\Pi_{r}^{t,\xi,P_{\xi}}$. Then
\begin{equation}\label{eq:zx0}
Y_{s}^{t,x,P_{\xi}}=\Phi(X_{T}^{t,x,P_{\xi}})
                    +\int_{s}^{T}\!\!f(\Pi_{r}^{t,x,P_{\xi}},P_{\Pi_{r}^{t,\xi}})dr
                    -\int_{s}^{T}\!\!Z_{r}^{t,x,P_{\xi}}dBr
                    -\int_{s}^{T}\!\!\int_{K}\!\!H_{r}^{t,x,P_{\xi}}(e)N_{\lambda}(dr,de),\ s\in[t,T].
\end{equation}
From Theorems \ref{th 6.1} and \ref{Pro8.1} we have
\begin{equation}\label{C4-obis}
\begin{split}
\partial_{x}Y_{s}^{t,x,P_{\xi}}=&\partial_{x}\Phi(X_{T}^{t,x,P_{\xi}})\partial_{x}X_{T}^{t,x,P_{\xi}}
                                +\int_{s}^{T}\!\!(\partial_{h}f)(\Pi_{r}^{t,x,P_{\xi}},P_{\Pi_{r}^{t,\xi}})\partial_{x}\Pi_{r}^{t,x,P_{\xi}}dr
                                -\int_{s}^{T}\!\!\partial_{x}Z_{r}^{t,x,P_{\xi}}dB_r\\
                                &
                                -\int_{s}^{T}\int_{K}\!\!\partial_{x}H_{r}^{t,x,P_{\xi}}(e)N_{\lambda}(dr,de),\ s\in[t,T],
\end{split}
\end{equation}
and
\begin{equation}\label{*103}
\begin{split}
&\partial_{x}^{2}Y_{s}^{t,x,P_{\xi}}\!=\!\left(\!\partial_{x}^{2}\Phi(X_{T}^{t,x,P_{\xi}})(\partial_{x}X_{T}^{t,x,P_{\xi}})^{2}
                                     \!+\!\partial_{x}\Phi(X_{T}^{t,x,P_{\xi}})\partial_{x}^{2}X_{T}^{t,x,P_{\xi}}\!\right)
                                     \!+\!\int_{s}^{T}\!\!(\partial_{h}^{2}f)(\Pi_{r}^{t,x,P_{\xi}},P_{\Pi_{r}^{t,\xi}})(\partial_{x}\Pi_{r}^{t,x,P_{\xi}})^{2}dr\\
&\ \ \ \!+\!\int_{s}^{T}\!\!(\partial_{h}f)(\Pi_{r}^{t,x,P_{\xi}},P_{\Pi_{r}^{t,\xi}})\partial_{x}^{2}\Pi_{r}^{t,x,P_{\xi}}dr
                                     -\int_{s}^{T}\partial_{x}^{2}Z_{r}^{t,x,P_{\xi}}dB_r
                                     -\int_{s}^{T}\int_{K}\partial_{x}^{2}H_{r}^{t,x,P_{\xi}}(e)N_{\lambda}(dr,de)\\
&=I_1(t)+\int_s^TI_2(t,r)dr\!+\!\int_{s}^{T}\!\!(\partial_{h}f)(\Pi_{r}^{t,x,P_{\xi}},P_{\Pi_{r}^{t,\xi}})\partial_{x}^{2}\Pi_{r}^{t,x,P_{\xi}}dr
                                     -\int_{s}^{T}\partial_{x}^{2}Z_{r}^{t,x,P_{\xi}}dB_r\\
&\ \ \ \ -\int_{s}^{T}\int_{K}\partial_{x}^{2}H_{r}^{t,x,P_{\xi}}(e)N_{\lambda}(dr,de),
\end{split}
\end{equation}
where
\[
\begin{aligned}
&I_{1}(t):=\partial_{x}^{2}\Phi(X_{T}^{t,x,P_{\xi}})(\partial_{x}X_{T}^{t,x,P_{\xi}})^{2}+\partial_{x}\Phi(X_{T}^{t,x,P_{\xi}})\partial_{x}^{2}X_{T}^{t,x,P_{\xi}},\\
&I_{2}(t,r):=(\partial_{h}^{2}f)(\Pi_{r}^{t,x,P_{\xi}},P_{\Pi_{r}^{t,\xi}})(\partial_{x}\Pi_{r}^{t,x,P_{\xi}})^{2}.\\
\end{aligned}
\]
It is obvious that for $0\leq t<t'\leq T$,
\begin{equation}\label{600}
\begin{split}
&E[|I_{1}(t)-I_{1}(t')|^2]\leq C(E[|\partial_{x}X_{T}^{t,x,P_{\xi}}|^{8}])^{\frac{1}{2}}(E[|X_{T}^{t,x,P_{\xi}}-X_{T}^{t',x,P_{\xi}}|^{4}])^{\frac{1}{2}}\\
&\quad\quad+C(E[|\partial_{x}X_{T}^{t,x,P_{\xi}}-\partial_{x}X_{T}^{t',x,P_{\xi}}|^{\frac{8}{3}}])^{\frac{3}{4}}(E[|\partial_{x}X_{T}^{t,x,P_{\xi}}
                                +\partial_{x}X_{T}^{t',x,P_{\xi}}|^{8}])^{\frac{1}{4}}\\
&\quad\quad+C(E[|\partial_{x}^{2}X_{T}^{t,x,P_{\xi}}|^{4}])^{\frac{1}{2}}(E[|X_{T}^{t,x,P_{\xi}}-X_{T}^{t',x,P_{\xi}}|^{4}])^{\frac{1}{2}}
+CE[|\partial_{x}^2X_{T}^{t,x,P_{\xi}}-\partial_{x}^2X_{T}^{t',x,P_{\xi}}|^{2}]\\
                           \leq&C(E[|X_{T}^{t,x,P_{\xi}}-X_{T}^{t',x,P_{\xi}}|^{4}])^{\frac{1}{2}}
                                +C(E[|\partial_{x}X_{T}^{t,x,P_{\xi}}-\partial_{x}X_{T}^{t',x,P_{\xi}}|^{\frac{8}{3}}])^{\frac{3}{4}} +CE[|\partial_{x}^2X_{T}^{t,x,P_{\xi}}-\partial_{x}^2X_{T}^{t',x,P_{\xi}}|^{2}]\\
                                =&I\!\!I_1(t,t')+I\!\!I_2(t,t')+I\!\!I_3(t,t'),
\end{split}
\end{equation}
where $
I\!\!I_{1}(t,t'):=C(E[|X_{T}^{t,x,P_{\xi}}-X_{T}^{t',x,P_{\xi}}|^{4}])^{\frac{1}{2}},\
I\!\!I_{2}(t,t'):=C(E[|\partial_{x}X_{T}^{t,x,P_{\xi}}-\partial_{x}X_{T}^{t',x,P_{\xi}}|^{\frac{8}{3}}])^{\frac{3}{4}},$ $ I\!\!I_{2}(t,t'):=$

\noindent$CE[|\partial_{x}^2X_{T}^{t,x,P_{\xi}}-\partial_{x}^2X_{T}^{t',x,P_{\xi}}|^{2}].
$
From Lemma \ref{le 3.1} we get
\begin{equation*}
\begin{split}
I\!\!I_{1}(t,t')=&C\left(E\left[E[|X_{T}^{t',x',P_{\eta}}-X_{T}^{t',x,P_{\xi}}|^{4}
|\mathcal{F}_{t}]|_{x'=X_{t'}^{t,x,P_{\xi}},\eta=X_{t'}^{t,\xi}}\right]\right)^{\frac{1}{2}}\\
                                        \leq&C\left(E[|X_{t'}^{t,x,P_{\xi}}-x|^{4}+W_{2}(P_{X_{t'}^{t,\xi}},
                                        P_{\xi})^{4}]\right)^{\frac{1}{2}}\leq C|t'-t|^{\frac{1}{2}}.
\end{split}
\end{equation*}
On the other hand we have
\[
\partial_{x}(X_{T}^{t,x,P_{\xi}})=\partial_{x}(X_{T}^{t',X_{t'}^{t,x,P_{\xi}},P_{X_{t'}^{t,\xi}}})
                                  =\partial_{x}X_{T}^{t',X_{t'}^{t,x,P_{\xi}},P_{X_{t'}^{t,\xi}}}\cdot\partial_{x}X_{t'}^{t,x,P_{\xi}},
                                  \]
and from Theorem \ref{th 5.1},
\begin{equation*}
\begin{split}
I\!\!I_{2}(t,t')&\leq C(E[|\partial_{x}X_{T}^{t',X_{t'}^{t,x,P_{\xi}},P_{X_{t'}^{t,\xi}}}|^{\frac{8}{3}}|\partial_{x}X_{t'}^{t,x,P_{\xi}}-1|^{\frac{8}{3}}])^{\frac{3}{4}}
                 \!+\!C(E[|\partial_{x}X_{T}^{t',X_{t'}^{t,x,P_{\xi}},P_{X_{t'}^{t,\xi}}}\!\!-\partial_{x}X_{T}^{t',x,P_{\xi}}|^{\frac{8}{3}}])^{\frac{3}{4}}\\
              &\leq C(E[|\partial_{x}X_{t'}^{t,x,P_{\xi}}-1|^{4}])^{\frac{1}{2}}
                   +C(E[|X_{t'}^{t,x,P_{\xi}}-x|^{\frac{8}{3}}+W_{2}\left(P_{X_{t'}^{t,\xi}},P_{\xi}\right)^{\frac{8}{3}}])^{\frac{3}{4}}
              \leq C(t'-t)^{\frac{1}{2}}.
\end{split}
\end{equation*}

\noindent Considering that $
\partial_{x}^2(X_{T}^{t,x,P_{\xi}})=(\partial_{x}^2X_{T}^{t',X_{t'}^{t,x,P_{\xi}},P_{X_{t'}^{t,\xi}}})(\partial_{x}X_{t'}^{t,x,P_{\xi}})^2
+\partial_{x}X_{T}^{t',X_{t'}^{t,x,P_{\xi}},P_{X_{t'}^{t,\xi}}}\cdot\partial_{x}^2X_{t'}^{t,x,P_{\xi}},$ a straight-forward estimate using Theorem 7.1 and, in particular, $(E[|\partial_{x}^2X_{t'}^{t,x,P_{\xi}}|^4])^{\frac{1}{2}}\leq C|t'-t|^{\frac{1}{2}}$, $(E[|(\partial_{x}X_{t'}^{t,x,P_{\xi}})^2-1|^4])^{\frac{1}{2}}\leq C(E[|\partial_{x}X_{t'}^{t,x,P_{\xi}}-1|^8])^{\frac{1}{4}}\leq C|t'-t|^{\frac{1}{4}}$, yields now $ I\!\!I_{3}(t,t') \leq C|t'-t|^{\frac{1}{4}}.$

Consequently, from (\ref{600}) we get \begin{equation}\label{599}E[|I_{1}(t)-I_{1}(t')|^2]\leq C|t'-t|^{\frac{1}{4}}.\end{equation}

\noindent Let us estimate now $|I_{2}(t,s)-I_{2}(t',s)|$ for $t<t'$. Using that (\ref {10.11-1}) yields
\begin{equation}\label{eq:zx1}
\esssup_{s\in[t,T]}|\partial_{x}\Pi_{s}^{t,x,P_{\xi}}|
\leq C\esssup_{s\in[t,T]}|\partial_{x}H_{s}^{t,x,P_{\xi}}|_{L^{2}(\lambda)}
\leq C\sup_{s\in[t,T]}|\partial_{x}X_{s}^{t,x,P_{\xi}}|.
\end{equation}
As $(\partial_h^2 f)$ is bounded and Lipschitz, we get
\begin{equation*}
\begin{split}
&E[\Big(\int_{t'}^{T}|I_{2}(t,r)-I_{2}(t',r)|dr\Big)^2]\leq CE[\Big(\int_{t'}^T|(\partial_{x}\Pi_{r}^{t,x,P_{\xi}})^2-(\partial_{x}\Pi_{r}^{t',x,P_{\xi}})^2|dr\Big)^2]\\
&\ \ +C\sup_{x\in {\mathbb{R}}^d}E[\int_{t'}^T|\Pi_{r}^{t,x,P_{\xi}}-\Pi_{r}^{t',x,P_{\xi}}|^2dr]\leq C\Big(E[\int_{t'}^T|\partial_{x}\Pi_{r}^{t',X_{t'}^{t,x,P_{\xi}},P_{X_{t'}^{t,\xi}}}\partial_xX_{t'}^{t,x,P_{\xi}}
-\partial_{x}\Pi_{r}^{t',x,P_{\xi}}|^\frac{8}{3}dr]\Big)^{\frac{3}{4}}\\
&\ \  +C\sup_{x\in {\mathbb{R}}^d}E[\int_{t'}^T|\Pi_{r}^{t',X_{t'}^{t,x,P_{\xi}},P_{X_{t'}^{t,\xi}}}-\Pi_{r}^{t',x,P_{\xi}}|^2dr].
\end{split}
\end{equation*}
Then from (\ref{eq:zx1}), Propositions \ref{pro 5.1}, \ref{pro 4.3} and (\ref{equ 6.11}) as well as Lemma 3.1, we obtain
\begin{equation}\label{598}\displaystyle E[\Big(\int_{t'}^{T}|I_{2}(t,r)-I_{2}(t',r)|dr\Big)^2]\leq C|t'-t|^{\frac{1}{4}}.\end{equation}

\noindent Applying Lemma 10.1-2) to the equation (\ref{*103}) it follows from (\ref{599}), (\ref{598}), and $\partial_{h}f$ is Lipschitz that

\begin{equation}\label{eq:zx2}
\begin{split}
&E\Big[\sup_{s\in[t',T]}|\partial_{x}^{2}Y_{s}^{t,x,P_{\xi}}-\partial_{x}^{2}Y_{s}^{t',x,P_{\xi}}|^{2}
+\int_{t'}^{T}\big(|\partial_{x}^{2}Z_{r}^{t,x,P_{\xi}}-\partial_{x}^{2}Z_{r}^{t',x,P_{\xi}}|^{2}\\
&\quad\quad+\int_{K}|\partial_{x}^{2}H_{r}^{t,x,P_{\xi}}(e)-\partial_{x}^{2}H_{r}^{t',x,P_{\xi}}(e)|^{2}\lambda(de)\big)dr\Big]\\
&\leq CE[|I_{1}(t)-I_{1}(t')|^2]+CE[\Big(\int_{t'}^{T}|I_{2}(t,r)-I_{2}(t',r)|dr\Big)^2]\\
&\quad\quad+CE[\Big(\int_{t'}^{T}|(\partial_{h}f(\Pi_{r}^{t,x,P_{\xi}},P_{\Pi_{r}^{t,\xi}})
-\partial_{h}f(\Pi_{r}^{t',x,P_{\xi}},P_{\Pi_{r}^{t',\xi}}))\partial_{x}^{2}\Pi_{r}^{t,x,P_{\xi}}|dr\Big)^2]\\
&\leq C|t'-t|^{\frac{1}{4}}+CE[\int_{t'}^{T}(|\Pi_{r}^{t,x,P_{\xi}}-\Pi_{r}^{t',x,P_{\xi}}|^2+W_2(P_{\Pi_{r}^{t,\xi}},P_{\Pi_{r}^{t',\xi}})^2)dr]\\
&\leq C|t'-t|^{\frac{1}{4}}.\ \ \ \ \  \hfill{\mbox{(The proof of the last inequality is similar to that of (\ref{598}).)}}\\
\end{split}
\end{equation}
On the other hand, from (\ref{eq:zx1}) and Theorem \ref{Pro8.1} we get also
\begin{equation}\label{eq:zx3}
\begin{split}
&|E[\partial_{x}^{2}Y_{t}^{t,x,P_{\xi}}-\partial_{x}^{2}Y_{t'}^{t,x,P_{\xi}}]|\\
&\leq\int_{t}^{t'}E[|(\partial_{h}^{2}f)(\Pi_{r}^{t,x,P_{\xi}},P_{\Pi_{r}^{t,\xi}})(\partial_{x}\Pi_{r}^{t,x,P_{\xi}})^{2}
 +(\partial_{h}f)(\Pi_{r}^{t,x,P_{\xi}},P_{\Pi_{r}^{t,\xi}})\partial_{x}^{2}\Pi_{r}^{t,x,P_{\xi}}|]dr\\
&\leq C(t'-t)+CE[(\int_{t}^{t'}|\partial_{x}^{2}\Pi_{r}^{t,x,P_{\xi}}|^{2}dr)^{\frac{1}{2}}](t'-t)^{\frac{1}{2}}\leq C(t'-t)^{\frac{1}{2}}.\\
\end{split}
\end{equation}
Consequently, from (\ref{eq:zx2}) and (\ref{eq:zx3}) we have
\begin{equation*}
\begin{array}{lll}
&|\partial_{x}^{2}V(t,x,P_{\xi})-\partial_{x}^{2}V(t',x,P_{\xi})|\\
\leq&|E[\partial_{x}^{2}Y_{t}^{t,x,P_{\xi}}-\partial_{x}^{2}Y_{t'}^{t,x,P_{\xi}}]| +(E[\sup_{s\in[t',T]}|\partial_{x}^{2}Y_{s}^{t,x,P_{\xi}}-\partial_{x}^{2}Y_{s}^{t',x,P_{\xi}}|^{2}])^{\frac{1}{2}}\\
\leq& C|t'-t|^\frac{1}{8},\ t,\ t'\in[0,T],\ x\in\mathbb{R},\ \xi\in L^{2}(\mathcal{F}_t).
\end{array}
\end{equation*}

\noindent {\it Step 4}. To prove iv) and  v) of Lemma \ref{newlem2}.

Let us prove v) which is more complicate, similar to prove iv). From (\ref{222}) we have
\begin{equation}\label{501}
\begin{split}
&\partial_{y}(\partial_{\mu}Y_{s}^{t,x,P_{\xi}}(y))
=\partial_x\Phi(X_{T}^{t,x,P_{\xi}})\partial_{y}(\partial_{\mu}X_{T}^{t,x,P_{\xi}}(y))
 +\int_{s}^{T}(\partial_{h}f)(\Pi_{r}^{t,x,P_{\xi}},P_{\Pi_{r}^{t,\xi}})\partial_{y}(\partial_{\mu}\Pi_{r}^{t,x,P_{\xi}}(y))dr\\
&+\int_{s}^{T}\widehat{E}[\partial_{y}(\partial_{\mu}f)(\Pi_{r}^{t,x,P_{\xi}},P_{\Pi_{r}^{t,\xi}},\widehat{\Pi}_{r}^{t,y,P_{\xi}})
(\partial_{x}\widehat{\Pi}_{r}^{t,y,P_{\xi}})^{2}
+(\partial_{\mu}f)(\Pi_{r}^{t,x,P_{\xi}},P_{\Pi_{r}^{t,\xi}},\widehat{\Pi}_{r}^{t,y,P_{\xi}})\partial_{x}^{2}\widehat{\Pi}_{r}^{t,y,P_{\xi}}]dr\\
&+\int_{s}^{T}\widehat{E}[(\partial_{\mu}f)(\Pi_{r}^{t,x,P_{\xi}},P_{\Pi_{r}^{t,\xi}},\widehat{\Pi}_{r}^{t,\widehat{\xi}})\partial_{y}
(\partial_{\mu}\widehat{\Pi}_{r}^{t,\widehat{\xi}}(y))]dr-\int_{s}^{T}\partial_{y}(\partial_{\mu}Z_{r}^{t,x,P_{\xi}}(y))dB_{r}\\
&-\int_{s}^{T}\int_{K}\partial_{y}(\partial_{\mu}H_{r}^{t,x,P_{\xi}}(y,e))N_{\lambda}(dr,de),\ s\in[t,T].
\end{split}
\end{equation}

\noindent Now replace $x$ by $\xi$ in above equation (\ref{501}) we get the solution $(\partial_{y}(\partial_{\mu}Y_{s}^{t,{\xi}}(y)), \partial_{y}(\partial_{\mu}Z_{s}^{t,{\xi}}(y)),$ $ \partial_{y}(\partial_{\mu}H_{s}^{t,{\xi}}(y)))$ of the following BSDE:
\begin{equation*}
\begin{split}
&\partial_{y}(\partial_{\mu}Y_{s}^{t,{\xi}}(y))
=\partial_x\Phi(X_{T}^{t,{\xi}})\partial_{y}(\partial_{\mu}X_{T}^{t,{\xi}}(y))
 +\int_{s}^{T}(\partial_{h}f)(\Pi_{r}^{t,{\xi}},P_{\Pi_{r}^{t,\xi}})\partial_{y}(\partial_{\mu}\Pi_{r}^{t,{\xi}}(y))dr\\
\end{split}
\end{equation*}
\begin{equation}\label{502}
\begin{split}
&+\int_{s}^{T}\widehat{E}[\partial_{y}(\partial_{\mu}f)(\Pi_{r}^{t,{\xi}},P_{\Pi_{r}^{t,\xi}},\widehat{\Pi}_{r}^{t,y,P_{\xi}})
(\partial_{x}\widehat{\Pi}_{r}^{t,y,P_{\xi}})^{2}
+(\partial_{\mu}f)(\Pi_{r}^{t, {\xi}},P_{\Pi_{r}^{t,\xi}},\widehat{\Pi}_{r}^{t,y,P_{\xi}})\partial_{x}^{2}\widehat{\Pi}_{r}^{t,y,P_{\xi}}]dr\\
&+\int_{s}^{T}\widehat{E}[(\partial_{\mu}f)(\Pi_{r}^{t, {\xi}},P_{\Pi_{r}^{t,\xi}},\widehat{\Pi}_{r}^{t,\widehat{\xi}})\partial_{y}
(\partial_{\mu}\widehat{\Pi}_{r}^{t,\widehat{\xi}}(y))]dr-\int_{s}^{T}\partial_{y}(\partial_{\mu}Z_{r}^{t, {\xi}}(y))dB_{r}\\
&-\int_{s}^{T}\int_{K}\partial_{y}(\partial_{\mu}H_{r}^{t, {\xi}}(y,e))N_{\lambda}(dr,de),\ s\in[t,T].
\end{split}
\end{equation}

\noindent On the other hand, notice that we have for $0\leq t<s\leq T$, P-a.s.,
\begin{equation}\label{503}
\begin{split}
& \partial_{y}(\partial_{\mu}X_{T}^{t,{\xi}}(y))\!
=\!(\partial_x X_{T}^{s,X_s^{t,\xi},P_{X_s^{t,\xi}}})\partial_{y}(\partial_{\mu}X_{s}^{t,{\xi}}(y))\!+\!\widehat{E}[\partial_y((\partial_\mu X_{T}^{s,X_s^{t,\xi},P_{X_s^{t,\xi}}})(\widehat{X}_s^{t,y,P_\xi}))(\partial_x\widehat{X}_s^{t,y,P_\xi})^2]\\
&\quad\quad\quad\quad\quad\quad\quad + \widehat{E}[(\partial_\mu X_{T}^{s,X_s^{t,\xi},P_{X_s^{t,\xi}}})(\widehat{X}_s^{t,\widehat{\xi}})\partial_y(\partial_\mu\widehat{X}_s^{t,\widehat{\xi},P_\xi}(y))].\\
\end{split}
\end{equation}

Using similar arguments as for (\ref{eq:zx2}), applying Theorem 10.2 to above equation (\ref{502}) by using (\ref{503}), (\ref{eq:zx1}), (\ref{eq:zx2}), Theorem 7.1,  Propositions \ref{pro 5.1} and \ref{pro 4.3}, (\ref{equ 6.11}), Lemma 3.1, Theorem 8.1, we obtain
\begin{equation}\label{504}
\begin{split}
&E\Big[\sup_{s\in[t',T]}|\partial_{y}(\partial_{\mu}Y_{s}^{t,{\xi}}(y))-\partial_{y}(\partial_{\mu}Y_{s}^{t',{\xi}}(y))|^{2}
+\int_{t'}^{T}\Big(|\partial_{y}(\partial_{\mu}Z_{r}^{t,{\xi}}(y))-\partial_{y}(\partial_{\mu}Z_{r}^{t',{\xi}}(y))|^2\\
&\quad+\int_{K}|\partial_{y}(\partial_{\mu}H_{r}^{t,{\xi}}(y,e))-\partial_{y}(\partial_{\mu}H_{r}^{t',{\xi}}(y,e))|^{2}\lambda(de)\Big)dr\Big]\\
&\leq CE[|J_1(t)-J_1(t')|^2]+CE[\Big(\int_{t'}^T|J_2(t,r)-J_2(t',r)|dr\Big)^2]\\
&\quad+CE[\Big(\int_{t'}^T|((\partial_{h}f)(\Pi_{r}^{t,{\xi}},P_{\Pi_{r}^{t,\xi}})-(\partial_{h}f)(\Pi_{r}^{t',{\xi}},P_{\Pi_{r}^{t',\xi}}))
\partial_{y}(\partial_{\mu}\Pi_{r}^{t,{\xi}}(y))|dr\Big)^2]\\
&\quad+CE[\Big(\int_{t'}^T|\widehat{E}[((\partial_{\mu}f)(\Pi_{r}^{t, {\xi}},P_{\Pi_{r}^{t,\xi}},\widehat{\Pi}_{r}^{t,\widehat{\xi}})-(\partial_{\mu}f)(\Pi_{r}^{t', {\xi}},P_{\Pi_{r}^{t',\xi}},\widehat{\Pi}_{r}^{t',\widehat{\xi}}))\partial_{y}(\partial_{\mu}\widehat{\Pi}_{r}^{t,\widehat{\xi}}(y))]|dr\Big)^2]\\
&\leq C|t-t'|^{\frac{1}{4}},\\
\end{split}
\end{equation}
where $J_1(t):=\partial_x\Phi(X_{T}^{t,{\xi}})\partial_{y}(\partial_{\mu}X_{T}^{t,{\xi}}(y))$, $J_2(t,r):=\widehat{E}[\partial_{y}(\partial_{\mu}f)(\Pi_{r}^{t,{\xi}},P_{\Pi_{r}^{t,\xi}},\widehat{\Pi}_{r}^{t,y,P_{\xi}})
(\partial_{x}\widehat{\Pi}_{r}^{t,y,P_{\xi}})^{2}$\\ $ +(\partial_{\mu}f)(\Pi_{r}^{t, {\xi}},P_{\Pi_{r}^{t,\xi}},\widehat{\Pi}_{r}^{t,y,P_{\xi}})\partial_{x}^{2}\widehat{\Pi}_{r}^{t,y,P_{\xi}}].$

Furthermore, still applying similar arguments to the equation (\ref{501}), but also using the estimate (\ref{504}), we obtain
\begin{equation*}\label{505}
\begin{split}
&E\Big[\sup_{s\in[t',T]}|\partial_{y}(\partial_{\mu}Y_{s}^{t,x,P_{\xi}}(y))-\partial_{y}(\partial_{\mu}Y_{s}^{t',x,P_{\xi}}(y))|^{2}
+\int_{t'}^{T}\Big(|\partial_{y}(\partial_{\mu}Z_{r}^{t,x,P_{\xi}}(y))-\partial_{y}(\partial_{\mu}Z_{r}^{t',x,P_{\xi}}(y))|^2\\
&+\int_{K}|\partial_{y}(\partial_{\mu}H_{r}^{t,x,P_{\xi}}(y,e))-\partial_{y}(\partial_{\mu}H_{r}^{t',x,P_{\xi}}(y,e))|^{2}\lambda(de)\Big)dr\Big]\leq C|t'-t|^{\frac{1}{4}}.\\
\end{split}
\end{equation*}
On the other hand, similar to (\ref{eq:zx3}) using (\ref{eq:zx1}) and Theorem \ref{Pro8.1} we get
\begin{equation*}\label{506}
\begin{split}
&|E[\partial_{y}(\partial_{\mu}Y_{t}^{t,x,P_{\xi}}(y))-\partial_{y}(\partial_{\mu}Y_{t'}^{t,x,P_{\xi}}(y))]|\\
&\leq E[\int_{t}^{t'}|(\partial_{h}f)(\Pi_{r}^{t,x,P_{\xi}},P_{\Pi_{r}^{t,\xi}})\partial_{y}(\partial_{\mu}\Pi_{r}^{t,x,P_{\xi}}(y))
 +\widehat{E}[\partial_{y}(\partial_{\mu}f)(\Pi_{r}^{t,x,P_{\xi}},P_{\Pi_{r}^{t,\xi}},\widehat{\Pi}_{r}^{t,y,P_{\xi}})(\partial_{x}
 \widehat{\Pi}_{r}^{t,y,P_{\xi}})^{2}\\
&\quad+(\partial_{\mu}f)(\Pi_{r}^{t,x,P_{\xi}},P_{\Pi_{r}^{t,\xi}},\widehat{\Pi}_{r}^{t,y,P_{\xi}})\partial_{x}^{2}\widehat{\Pi}_{r}^{t,y,P_{\xi}}
 +(\partial_{\mu}f)(\Pi_{r}^{t,x,P_{\xi}},P_{\Pi_{r}^{t,\xi}},\widehat{\Pi}_{r}^{t,\widehat{\xi}})\partial_{y}
 (\partial_{\mu}\widehat{\Pi}_{r}^{t,\widehat{\xi}}(y))]|dr]\\
&\leq C(t'-t)^{\frac{1}{2}}.
\end{split}
\end{equation*}
Hence, from above two estimates we get
\begin{equation*}
\begin{split}
&|\partial_{y}(\partial_{\mu}V)(t,x,P_{\xi},y)-\partial_{y}(\partial_{\mu}V)(t',x,P_{\xi},y)|\\
&\leq|E[\partial_{y}(\partial_{\mu}Y_{t}^{t,x,P_{\xi}}(y))-\partial_{y}(\partial_{\mu}Y_{t'}^{t,x,P_{\xi}}(y))]|
+(E[\sup_{s\in[t',T]}|\partial_{y}(\partial_{\mu}Y_{s}^{t,x,P_{\xi}}(y))-\partial_{y}(\partial_{\mu}Y_{s}^{t',x,P_{\xi}}(y))|^{2}])^{\frac{1}{2}}\\
&\leq C|t'-t|^{\frac{1}{8}},\ t,\ t'\in[0,T],\ x\in\mathbb{R},\ y\in\mathbb{R},\ \xi\in L^2(\Omega, {\cal F}_t, P).
\end{split}
\end{equation*}
\end{proof}

\end{document}